\documentclass{amsproc}
\usepackage{graphicx}
\usepackage{calc}
\usepackage{url}
\usepackage{mathrsfs}
\usepackage{rotating}
\newlength{\depthofsumsign}
\usepackage[normalem]{ulem}
\usepackage{amsthm}

\usepackage{enumitem}[2011/09/28]
\setenumerate{align=left, leftmargin=0pt,labelsep=.5em, labelindent=0\parindent,listparindent=\parindent,itemindent=*}
\usepackage{array}
\usepackage{natbib}
\usepackage{multirow}
\usepackage{multirow}

\usepackage{caption}[2012/02/19]

\usepackage{longtable,lscape}

\usepackage{amssymb,bm,amsmath}
\usepackage{amsfonts}

\usepackage[OT2,T1,T2A]{fontenc}
\usepackage[utf8x]{inputenc}

\DeclareRobustCommand{\cyrins}[1]{%
  \begingroup\fontfamily{erewhon-TLF}%
  \foreignlanguage{russian}{#1}%
  \endgroup
}

\usepackage[english,french,german,russian]{babel}
\DeclareMathOperator{\Lyn}{Lyn}

\DeclareMathOperator{\Ls}{Ls}

\DeclareMathOperator{\Gal}{Gal}

\DeclareMathOperator{\Span}{span}

\DeclareMathOperator{\sgn}{sgn}

\usepackage{extarrows}
\DeclareMathOperator{\D}{d}
\DeclareMathOperator{\I}{Im}
\DeclareMathOperator{\R}{Re}

\DeclareMathOperator{\Li}{Li}

\DeclareMathOperator*{\fib}{fib}
\DeclareMathOperator{\Fib}{Fib}

\def\Xint{\int\kern-1.1em}

\def\counterint{\Xint\rotcirclearrowleft}

\bibpunct{[}{]}{,}{n}{}{;}

\def\qed{\hfill$ \blacksquare$}
\def\eor{\hfill$ \square$}

\def\RotSymbol#1#2#3{\rotatebox[origin=c]{#1}{$#2#3$}}
\def\rotcirclearrowleft{\mathpalette{\RotSymbol{-14}}\circlearrowleft}

\theoremstyle{plain}
\newtheorem{theorem}{Theorem}[section]
\newtheorem{proposition}[theorem]{Proposition}
\newtheorem{lemma}[theorem]{Lemma}
\newtheorem{corollary}[theorem]{Corollary}

\theoremstyle{definition}

\theoremstyle{remark}
\newtheorem{remark}[theorem]{Remark}

\renewcommand{\thesection}{\arabic{section}}
\renewcommand*\thetable{\arabic{table}}
\numberwithin{equation}{section}

\DeclareSymbolFont{egr}{OML}{antt}{m}{it}
\DeclareMathSymbol{\varsigma}{\mathord}{egr}{"26}
\DeclareMathSymbol{\zeta}{\mathord}{egr}{"10}
\DeclareMathSymbol{\xi}{\mathord}{egr}{"18}

\usepackage{colonequals}

\begin{document}

\pagenumbering{roman}
\selectlanguage{english}
\title{Hyper-Mahler measures via Goncharov--Deligne cyclotomy}
\author[Yajun Zhou]{Yajun Zhou}\dedicatory{\begin{flushright}To Prof.\  Weinan E, on the occasion of his 60th birthday\end{flushright}}
\address{Program in Applied and Computational Mathematics (PACM), Princeton University, Princeton, NJ 08544} \email{yajunz@math.princeton.edu}\curraddr{ \textsc{Academy of Advanced Interdisciplinary Studies (AAIS), Peking University, Beijing 100871, P. R. China}}\email{yajun.zhou.1982@pku.edu.cn}

\date{\today}
\thanks{This research was supported in part  by the Applied Mathematics Program within the Department of Energy
(DOE) Office of Advanced Scientific Computing Research (ASCR) as part of the Collaboratory on
Mathematics for Mesoscopic Modeling of Materials (CM4).
}

\begin{abstract} The hyper-Mahler measures  $ m_k( 1+x_1+x_2),k\in\mathbb Z_{>1}$ and   $ m_k( 1+x_1+x_2+x_3),k\in\mathbb Z_{>1}$  are evaluated in closed form via  Goncharov--Deligne periods, namely  $\mathbb Q$-linear combinations of multiple polylogarithms at  cyclotomic points   (complex-valued coordinates that are roots of unity). Some infinite series related to these hyper-Mahler measures are also explicitly represented as Goncharov--Deligne periods of  levels $1$, $2$, $ 3$, $4$, $6$, $8$, $10$ and $12$. \end{abstract}
\maketitle


\pagenumbering{arabic}

\section{Introduction}For an $ n$-variate Laurent polynomial $ P\in\mathbb C[x_1^{\pm1},\dots,x_n^{\pm1}]$ that does not vanish identically, its   $ k$-Mahler measure  \cite[Definition 1]{KurokawaLalinOchiai2008} is given by the following integral: \begin{align}
m_{k}(P)\colonequals\int_0^1\D t_1\cdots\int_0^1\D t_n\,\log^{k}\left|P\left(e^{2\pi i t_1},\dots,e^{2\pi i t_n}\right)\right|,
\end{align}where $k$ is a positive integer. When $ k=1$, this becomes the logarithmic Mahler measure $ m(P)=m_1(P)$ of $P$. When the integer $k$ is greater than $1$, the quantity $m_k(P)$ is a hyper-Mahler measure.

Smyth \cite[Appendix 1]{Boyd1981b} has evaluated a logarithmic Mahler measure \begin{align}m_1(1+x_1+x_2+x_3)=\frac{7\zeta_{3}}{2\pi^2} \end{align}    in terms of Ap\'ery's constant $ \zeta_{3}\colonequals\sum_{n=1}^\infty\frac{1}{n^3}$. Borwein \textit{et al.}\ (see \cite[(6.10)]{BSWZ2012} and \cite[(4.17)]{BorweinStraub2012Mahler}) have
computed a hyper-Mahler measure \begin{align}
m_2(1+x_1+x_2+x_3)=\frac{24\Li_4\left(\frac12\right)-\frac{\pi ^4}{5}+21\zeta_3\log2-\pi^2\log^22+\log^42}{\pi^{2}}
\label{eq:m2Borwein}\end{align} involving a polylogarithmic constant   $ \Li_4\left(\frac12\right)\colonequals\sum_{n=1}^\infty \frac{1}{2^{n}n^4}$. If we introduce an alternating double sum  $ \zeta_{- 3, 1}\colonequals\sum_{m>n>0}\frac{(-1)^{m}}{m^{3}n}\equiv\sum_{m=1}^\infty\sum_{n=1}^{m-1}\frac{(-1)^{m}}{m^{3}n}$, then we can condense \eqref{eq:m2Borwein} into\begin{align}m_2(1+x_1+x_2+x_3)=\frac{12\zeta_{- 3, 1}}{\pi^{2}}+\frac{ \pi ^2}{20},
\tag{\ref{eq:m2Borwein}$'$}
\end{align}in the light of a proven relation between $ \zeta_{- 3, 1}$ and $ \Li_4\left(\frac12\right)$ \cite[(5.5)]{BZB2008}.

In this work, we study the hyper-Mahler measures  $ m_k( 1+x_1+x_2)$  and  $ m_k( 1+x_1+x_2+x_3)$  in detail, for all $ k\in\mathbb Z_{>1}$. For every positive integer $ k$,  we express the quantity $ \pi i\,m_k( 1+x_1+x_2)$    as a $ \mathbb Q$-linear combination of    special values for Goncharov's \cite{Goncharov1997} multiple polylogarithm\footnote{While defining $ \Li_{a_1,\dots,a_n}(z_1,\dots,z_n)$ and  $ \zeta_{ {a_1}, \dots,a_n}$, we sort the positive integer parameters $  {a_1}, \dots,a_n$ as in \texttt{Maple 2022}'s built-in functions   \texttt{MultiPolylog} and \texttt{MultiZeta} (which is Hoffman's \cite{Hoffman1992} order---see also \cite{Waldschmidt2002,Deligne2010,MZVdatamine2010,Broadhurst2013MZV,Frellesvig2016,Frellesvig2018Maple,Au2022a,Fresan2022period} for the same practice). In the  \texttt{HyperInt} package \cite{Panzer2015}  for \texttt{Maple}, our  $ \Li_{a_1,\dots,a_n}(z_1,\dots,z_n)$  is represented by $\mathtt{Mpl}  ([  {a_n}, \dots,a_1],[z_n,\dots,z_1])$    (the latter of which follows Zagier's \cite{Zagier1994} order---see also \cite{Goncharov1997,GoncharovManin2004,Brown2009a,Brown2009arXiv,Brown2009b,LalinLechasseur2016,DuhrDulat2019}  for the same convention).   }\begin{align}
\Li_{a_1,\dots,a_n}(z_1,\dots,z_n)\colonequals \sum_{\ell_{1}>\dots>\ell_{n}>0}\prod_{j=1}^n\frac{z_{j}^{\ell_{j}}}{\ell_j^{a_j}},
\label{eq:Mpl_defn}\end{align}where $ a_1,\dots,a_n$ are positive integers and   $ z_1,\dots,z_n$ are third roots of unity (see Table  \ref{tab:mk_eval_W3} for the first few), with $(a_1, z_1) \neq (1, 1)$ to ensure convergence \cite[(0.1)]{Deligne2010}. Meanwhile,  $ \mathbb Q$-linear combinations of   Hoffman's (alternating) multiple zeta values \cite{Hoffman1992} (also known as Euler--Zagier numbers \cite{Zagier1994,Waldschmidt2002})\begin{align}\zeta_{ {A_1}, \dots,A_n}\colonequals\Li_{|A_1|,\dots,|A_n|}\left( \frac{A_{1}}{|A_{1}|} ,\dots,\frac{A_{n}}{|A_{n}|}\right),\label{eq:aMZV_defn}\end{align}enter the representations of   $ \pi^2 m_k( 1+x_1+x_2+x_3)$    (see  Table \ref{tab:mk_eval} for examples).

In \S\ref{sec:int_repn_mk}, we verify all the  statements in the last paragraph. On the qualitative side, we do so by linking Broadhurst's formula  \cite[(9)]{Broadhurst2009} for moments of random walks (\S\S\ref{subsec:BroadhurstWn}--\ref{subsec:HankelWn}) to Brown's homotopy-invariant
iterated integrals \cite[\S1.1]{Brown2009b} on the moduli spaces $ \mathfrak M_{0,n}$ of genus-zero curves with $n$ marked points (\S\S\ref{subsec:Zk(6)}--\ref{subsec:Zk(2)}). On the quantitative side, with Panzer's \texttt{HyperInt} package \cite{Panzer2015} that implements Brown's algorithm \cite{Brown2009a,Brown2009arXiv,Brown2009b} as well as  Au's \texttt{MultipleZetaValues} package  \cite{Au2022a} that  reduces  certain special values of multiple polylogarithms, we generate Tables   \ref{tab:mk_eval_W3} and \ref{tab:mk_eval} as stated.

\begin{table}[t] \caption{Selected closed forms for $k$-Mahler measures  $ m_{k}(1+x_1+x_2)$, where $ \omega\colonequals e^{2\pi i/3}$\label{tab:mk_eval_W3}}
\begin{scriptsize}\begin{align*}\begin{array}{r|l}\hline\hline k&m_{k}(1+x_1+x_2)\vphantom{\frac\int\int}\\\hline1&\displaystyle\frac{3}{2\pi}\I \Li_2(\omega)\vphantom{\frac{\frac11}{\int}}\\[8pt]2&\displaystyle-\frac{2\cdot3}{\pi}\I \Li_{2,1}(\omega,1)+\frac{\pi^2}{2^{2}\cdot3^{3}}\\[8pt]3&\displaystyle-\frac{3^{2}\cdot17}{2^{3}\pi}\I\Li_4(\omega)+\frac{2^{2}\cdot3^{2}}{\pi}\I\Li_{2,1,1}(\omega,1,1)+\frac{23\zeta_{3}}{2\cdot3}+\frac{\pi}{2^{3}}\I\Li_2(\omega) \\[8pt]4&\displaystyle -\frac{2^{2}\cdot3^2\cdot5\cdot11}{13\pi}\I \Li_{4,1}(\omega,1)-\frac{3^{3}\cdot7^{2}}{13\pi}\I \Li_{3,2}(\omega,1)-\frac{2^{5}\cdot3^2}{\pi}\I\Li_{2,1,1,1}(\omega,1,1,1)\\[5pt]&\displaystyle{}+2^{3}\cdot3\R \Li_{3,1}(\omega,1)-\pi\I \Li_{2,1}(\omega,1)+\frac{3\zeta_{3}}{\pi}\I \Li_2(\omega)+\frac{3^{2}}{2^{2}}[\I\Li_2(\omega)]^{2}-\frac{5923 \pi^{4}}{2^{4}\cdot 3^{4}\cdot 5\cdot13}\\[5pt]\hline\hline\end{array}\end{align*}\end{scriptsize}
\end{table}

\begin{table}[t] \caption{Selected closed forms for $k$-Mahler measures  $ m_{k}(1+x_1+x_2+x_3)$\label{tab:mk_eval}}
\begin{scriptsize}\begin{align*}\begin{array}{r|l}\hline\hline k&m_{k}(1+x_1+x_2+x_3)\vphantom{\frac\int\int}\\\hline1&\displaystyle\frac{7\zeta_{3}}{2 \pi ^2}\vphantom{\frac{\frac11}{\int}}\\[8pt]2&\displaystyle\frac{2^{2}\cdot3\zeta_{- 3, 1}}{\pi^{2}}+\frac{ \pi ^2}{2^{2}\cdot5}\\[8pt]3&\displaystyle -\frac{3\cdot5\cdot23 \zeta _{5}}{2^{2} \pi ^2}-\frac{2^{2}\cdot3^{3}\zeta_{-{3},1,1}}{\pi ^2}+\frac{61 \zeta_{3}}{2^{3}}\\[8pt]4&\displaystyle-\frac{2^{4}\cdot3^{2}\cdot13\zeta_{- 5, 1}}{\pi ^2}+\frac{2^{4}\cdot3^{4}\zeta_{-{3},1,1,1}}{\pi ^2}-2\cdot3\zeta_{- 3, 1}+\frac{3^{3}\cdot29 \zeta_{3}^2}{\pi ^2}-\frac{1867 \pi ^4}{2^{4}\cdot3\cdot5\cdot7}\\[8pt]{5}&\displaystyle \-\frac{3\cdot5\cdot22973 \zeta_{7}}{2^{2}\cdot17 \pi ^2}+\frac{2^{4}\cdot3^{2}\cdot5^{2}\cdot11 \zeta_{-{5}, 1, 1}}{17 \pi ^2}+\frac{2^{8}\cdot3\cdot5\cdot19 \zeta_{-{3}, 3, 1}}{17 \pi ^2}+\frac{2^{2}\cdot3\cdot5 \zeta_{-{3}, 1} \zeta_{3}}{\pi ^2}\vphantom{\frac{\frac11}{\int}}\\[5pt]&\displaystyle{}-\frac{2^{4}\cdot3^{5}\cdot5 \zeta_{-{3}, 1, 1, 1, 1}}{\pi ^2}-\frac{5^{2}\cdot29\cdot59 \zeta_{5}}{2^{3}\cdot17}+2\cdot3^{2}\cdot5 \zeta_{-{3}, 1, 1}-\frac{26267 \pi ^2 \zeta_{3}}{2^{5}\cdot3\cdot17}\\[8pt]{6}&\displaystyle-\frac{2^{5}\cdot3^{3}\cdot5\cdot317 \zeta_{-{7}, 1}}{17 \pi ^2}+\frac{2\cdot3^{3}\cdot11\cdot29\cdot43 \zeta_{{5}, 3}}{17 \pi ^2}+\frac{3^{4}\cdot5^{3}\cdot11\cdot113 \zeta_{5} \zeta_{3}}{2^{2}\cdot17 \pi ^2}-\frac{2^{5}\cdot3^{4}\cdot5^{2}\cdot11 \zeta_{-{5}, 1, 1, 1}}{17 \pi ^2}\\[5pt]&\displaystyle{}-\frac{2^{9}\cdot3^{3}\cdot5\cdot19 \zeta_{-{3}, 3, 1, 1}}{17 \pi ^2}-\frac{2^{3}\cdot3^{3}\cdot5 \zeta_{-{3}, 1, 1} \zeta_{3}}{\pi ^2}+\frac{2^{5}\cdot3^{7}\cdot5 \zeta_{-{3}, 1, 1, 1, 1, 1}}{\pi ^2}+2^{2}\cdot3^{2}\cdot5\cdot13 \zeta_{-{5}, 1} \\[5pt]&\displaystyle {}-2^{2}\cdot3^{4}\cdot5 \zeta_{-{3}, 1, 1, 1}+\frac{3\cdot11 \pi ^2 \zeta_{-{3}, 1}}{2^{2}}-\frac{3^{3}\cdot5^{2}\cdot1213 \zeta_{3}^2}{2^{3}\cdot17}-\frac{26459029 \pi ^6}{2^{6}\cdot3^{2}\cdot5^{2}\cdot7\cdot17}\\[8pt]\hline\hline\end{array}\end{align*}\end{scriptsize}
\end{table}

In \S\ref{sec:series_Zk(N)},
we extend our methods to several families of infinite series occurring in recent studies of number theory and high energy physics. While these infinite sums may or may not be directly related to hyper-Mahler measures, they can all be expressed algorithmically in terms of multiple polylogarithms at roots of unity. Our automated evaluations provide  either  new patterns for infinite series or  succinct alternatives to existent treatments of the same objects.   \section{Integral representations of $k$-Mahler measures\label{sec:int_repn_mk}}
\subsection{Broadhurst's formula for zeta Mahler measures\label{subsec:BroadhurstWn}}
Consider the zeta Mahler measure  (see \cite[(1.1)]{Akatsuka2009} and \cite[(2.3)]{BSWZ2012})\begin{align}W_{n}(s)\colonequals{}&\int_0^1\D t_1\cdots\,\int_0^1\D t_{n}\left|\sum_{j=1}^ne^{2\pi i t_j}\right|^{s}=\int_0^1\D t_1\cdots\,\int_0^1\D t_{n-1}\,\left|1+\sum_{j=1}^{n-1}e^{2\pi i t_j}\right|^{s},
\end{align}whose $k$-th order derivative at $s=0$ evaluates the $k$-Mahler measure:\begin{align}
\left.\frac{\D^k W_{n}(s)}{\D s^k}\right|_{s=0}=\int_0^1\D t_1\cdots\,\int_0^1\D t_{n-1}\,\,\log^{k}\left|1+\sum_{j=1}^{n-1}e^{2\pi i t_j}\right|\equalscolon m_{k}\left(1+\sum_{j=1}^{n-1}x_{j}\right).\label{eq:dkW_mk}
\end{align}By definition, we have\begin{align}
W_1(s)=1.\label{eq:W1}
\end{align}The evaluation \begin{align}
W_2(s)={}&\frac{2^s \Gamma \left(\frac{1+s}{2}\right)}{\sqrt{\pi } \Gamma \left(1+\frac{s}{2}\right)}\label{eq:W2}
\end{align}can be found in \cite[Theorem 1(1)]{Akatsuka2009},  where $ \Gamma(\sigma)\colonequals\int_0^\infty t^{\sigma-1}e^{-t}\D t,\sigma>0$ defines Euler's gamma function. In \cite[Theorem 14]{KurokawaLalinOchiai2008} and \cite[(25)]{BNSW2011},  the quantity $W_2(s)$ was written as an equivalent binomial coefficient $ \binom{s}{s/2}$.

In the next lemma, we recapitulate Broadhurst's integral representation \cite[(9)]{Broadhurst2009} of $ W_n(s)$,  which will be instrumental in  the rest of this section.

\begin{lemma}[Broadhurst's formula for $ W_n(s)$]For $ n\in\mathbb Z_{>0}$, the  integral representation\begin{align}
W_n(s)=2^{s+1}\frac{\Gamma\left(1+\frac{s}{2}\right)}{\Gamma\left(-\frac{s}{2}\right)}\int_0^\infty J_0^{n}(x)\frac{\D x}{x^{s+1}}
\end{align} is valid for  $ s\in\left( \max\{-2,-\frac{n}{2}\} ,0\right)$, where      $ J_0(x)\colonequals\sum_{\ell=0}^\infty\frac{(-1)^\ell}{(\ell!)^{2}}\left(\frac x2\right)^{2\ell}$  is the zeroth-order Bessel function of the first kind. \qed \end{lemma}
 \begin{remark}For $ n\in\{1,2\}$,  Broadhurst's integral representations are compatible with \eqref{eq:W1} and \eqref{eq:W2}, as noted by Borwein--Straub--Wan \cite[Example 2.2]{BSW2013}. For $n=2$, one may also check this compatibility against  a critical case of the Weber--Schafheitlin integral \cite[\S13.41(2)]{Watson1944Bessel}. \eor\end{remark}
\subsection{Some integral formulae involving Bessel functions\label{subsec:HankelWn}}
In this subsection, we will reformulate Broadhurst's integral representations of  $ W_3(s)$ and $ W_4(s)$ through the Parseval--Plancherel identity for  Hankel transforms  \cite[\S14.3(3)]{Watson1944Bessel} \begin{align}\begin{split}&
\frac12\int_0^\infty\left[ \int_0^\infty J_\nu\left(\sqrt{u}x\right)f(x)x\D x \right]\left[  \int_0^\infty J_\nu\left(\sqrt{u}x\right)g(x)x\D x \right]\D u\\={}&\int_0^\infty f(x)g(x)x\D x,\quad \nu\geq-\frac12,\end{split}\label{eq:PP_Jnu}
\end{align}where $ J_\nu(X)\colonequals\sum_{\ell=0}^\infty\frac{(-1)^\ell}{\ell!\Gamma(\nu+\ell+1)}\left(\frac X2\right)^{2\ell+\nu}$ is the Bessel function of order $\nu$.
This motivates the study of the integrals in the next proposition.

\begin{proposition}[Weber--Schafheitlin integral and its generalizations]\label{prop:WS}For $ s\in(-2,2)$, we have \begin{align}
\int_0^\infty J_{1+\frac{s}{2}}\left(\sqrt{u}x\right ) \frac{J_{0}(x)}{x^{s/2}}\D x={}&\begin{cases}0, & u\in(0,1),\ \\
\displaystyle\frac{1}{u^\frac{2+s}{4}{\Gamma\left(1+ \frac{s}{2} \right)}}\left(\frac{u-1}{2}\right)^{\frac{s}{2}}, & u\in(1,\infty),
\end{cases}\\[5pt]\int_0^\infty J_{1+\frac{s}{2}}\left(\sqrt{u}x\right ) \frac{J_{0}^{2}(x)}{x^{s/2}}\D x={}&\begin{cases}\displaystyle\int_0^{2\arcsin\smash{\frac{\sqrt{u}}{2}}}\frac{\smash[t]{\left(u-4\sin^2\smash{\frac{\phi}{2}}\right)^{\frac{s}{2}}}}{2^{^{\frac{s}{2}}}\pi u^{^\frac{2+s}{4}}\Gamma\left(1+ \frac{s}{2} \right)}\D\phi, & u\in(0,4),\ \\
\displaystyle\int_0^{\pi}\frac{\left(u-4\sin^2\smash{\frac{\phi}{2}}\right)^{\frac{s}{2}}}{2^{^{\frac{s}{2}}}\pi u^{^\frac{2+s}{4}}\smash{\Gamma\left(1+ \frac{s}{2} \right)}}\D\phi, & u\in(4,\infty).
\end{cases}{}&
\end{align}\end{proposition}\begin{proof}See \cite[\S13.46(1)]{Watson1944Bessel}.\end{proof}Armed with the formulae in the proposition above, as well as the observation that \begin{align}\begin{split}&
\int_0^\infty J_{1+\frac{s}{2}}\left(\sqrt{u}x\right ) \frac{J_{0}^{2}(x)-1}{x^{s/2}}\D x\\={}&\int_0^\infty J_{1+\frac{s}{2}}\left(\sqrt{u}x\right ) \frac{J_{0}^{2}(x)}{x^{s/2}}\D x+\int_{0}^\infty\frac{\partial}{\partial x}\frac{J_{\frac{s}{2}}\left(\sqrt{u}x\right )}{\sqrt{u}x^{s/2}}\D x\\={}&\int_0^\infty J_{1+\frac{s}{2}}\left(\sqrt{u}x\right ) \frac{J_{0}^{2}(x)}{x^{s/2}}\D x-\frac{1}{2^{^{\frac{s}{2}}} u^{^\frac{2-s}{4}}\smash{\Gamma\left(1+ \frac{s}{2} \right)}},\end{split}
\end{align} we will apply the Parseval--Plancherel identity \eqref{eq:PP_Jnu} to  $W_3(s)-W_1(s)$ and   $W_4(s)-W_2(s)$, in the corollary below.

\begin{corollary}[Hankel--Broadhurst representations for $ W_3(s)$ and $ W_4(s)$]Set  $ \varrho=e^{\pi i/3}$ and  $\mathsf Q(x)\colonequals \left(1-\frac{x}{\varrho}\right)(1- x\varrho)$.  Define five rational functions\begin{align}
\mathsf A_{3}(t,\tau)\colonequals{}&\frac{\left(1-\frac{t}{\tau}\right)(1-t\tau)}{\mathsf Q(t)},&\mathsf B_3(\tau)\colonequals{}&\frac{\tau^{2}}{(1+\tau)^{2}\mathsf Q(\tau)},&\mathsf C_3(\tau)\colonequals{}&\frac{(1+\tau)^{2}}{3\tau},\label{eq:A3B3C3}\\\mathsf A_4(t,\tau)\colonequals{}&\frac{\left(1-\frac{t}{\tau}\right)(1-t\tau)}{t},&\mathsf B_4(\tau)\colonequals{}&-\frac{\tau}{\mathsf (1-\tau)^{2}}.\label{eq:A4B4}
\end{align}When $ s\in\left(-\frac{1}{2},0\right)$, we have \begin{align}\begin{split}&
W_{3}(s)-1\\={}&\frac{3^{s+\frac{3}{2}}\sin\frac{\pi  s}{2}}{\pi ^2}\int_0^1\left[\int_{-1}^\tau\mathsf A_{3}^{\frac{s}{2}}(t,\tau)\frac{\D t}{\mathsf Q(t)}-\mathsf C_{3}^{\frac{s}{2}}(\tau)\int_{-1}^{1}\frac{\D t}{\mathsf Q(t)}\right]\frac{\tau -1}{\tau +1}\frac{\mathsf B_3^{\frac{s}{2}}(\tau)\D\tau}{\mathsf Q(\tau)}\\&{}+\frac{3^{s+\frac{3}{2}}\sin\frac{\pi  s}{2}}{\pi ^2}\int_1^\varrho\left\{\int_{-1}^1\left[\mathsf A_{3}^{\frac{s}{2}}(t,\tau) -\mathsf C_{3}^{\frac{s}{2}}(\tau)\right]\frac{\D t}{\mathsf Q(t)}\right\}\frac{\tau -1}{\tau +1}\frac{\mathsf B_3^{\frac{s}{2}}(\tau)\D\tau}{\mathsf Q(\tau)},\end{split}\label{eq:W3_per}
\end{align}where all the paths of integration  are straight line segments; when $ s\in\left(-1,0\right)$, we have {\small\begin{align}\begin{split}&
W_{4}(s)-\frac{2^s \Gamma \left(\frac{1+s}{2}\right)}{\sqrt{\pi } \Gamma \left(1+\frac{s}{2}\right)}\\={}&\frac{ \sin\frac{\pi  s}{2}}{\pi ^3}\curvearrowleft\kern-1.15em\int_{1}^{-1}\left[\curvearrowleft\kern-1.15em\int_{1}^{\tau}\mathsf A_{4}^{\frac{s}{2}}(t,\tau)\frac{\D t}{t}\right]\left[ \curvearrowleft\kern-1.15em\int_{1}^{\tau}\mathsf A_{4}^{\frac{s}{2}}(T,\tau)\frac{\D T}{T}-\pi i\mathsf B_4^{-\frac{s}{2}}(\tau)\right]\frac{\tau +1}{\tau -1}\frac{\mathsf B_4^{\frac{s}{2}}(\tau)\D\tau}{\tau}\\&{}+\frac{\sin\frac{\pi  s}{2}}{\pi ^3}\int_{-1}^{0}\left[\curvearrowleft\kern-1.15em\int_{1}^{-1}\mathsf  A_{4}^{\frac{s}{2}}(t,\tau)\frac{\D t}{ t}\right]\left\{ \curvearrowleft\kern-1.15em\int_{1}^{-1}\left[\mathsf  A_{4}^{\frac{s}{2}}(T,\tau)-\mathsf B_4^{-\frac{s}{2}}(\tau)\right]\frac{\D T}{T}\right\}\frac{\tau +1}{\tau -1}\frac{\mathsf B_4^{\frac{s}{2}}(\tau)\D\tau}{\tau},\end{split}
\end{align}}where the  path of integration for each $^{_{_\curvearrowleft}}\kern-.9em\int$ runs counterclockwise along the unit circle, and the path of integration for $\tau$ from $ -1$ to $0$ is a  straight line segment.\end{corollary}\begin{proof}For $ W_3(s)$,  make the variable substitutions $
\phi=2\arctan\frac{1+t}{\sqrt{3}(1-t)}$ \big[which satisfies $ 4\sin^2\frac{\phi}{2}=\frac{(1+t)^2}{1-t+t^{2}}$\big] and $u=\frac{(1+\tau)^{2}}{1-\tau+\tau^{2}}
$ after invoking \eqref{eq:PP_Jnu}. For  $ W_4(s)$,   use $ \phi=\arg t$ \big[which satisfies $ 4\sin^2\frac{\phi}{2}=-\frac{(1-t)^2}{t}$ for $ |t|=1$\big] and $ u=-\frac{(1-\tau)^2}{\tau}$ instead.\end{proof}
\subsection{Period structure of $ m_k(1+x_1+x_2)$\label{subsec:Zk(6)}}
As a variation on Goncharov's notation   \cite[\S1.2]{Goncharov1998}, we define  $ \mathbb Q$-vector spaces\footnote{In this article, we write ``$\cdots\in\mathbb C $'' (as opposed to ``$ \cdots\in\mathbb C\cup\{\infty\}$'') to indicate that the series or integral in question is convergent. Thus, the condition of convergence  $(a_1, z_1) \neq (1, 1)$ for  \eqref{eq:Mpl_defn} is imposed tacitly here, despite being absent from the predicate in \eqref{eq:Zk(N)_defn}.}\begin{align}
\mathfrak Z_{ k}(N)\colonequals\Span_{\mathbb Q}\left\{\Li_{a_1,\dots,a_n}(z_1,\dots,z_n)\in\mathbb C\left|\begin{smallmatrix}a_1,\dots,a_n\in\mathbb Z_{>0}\\z_{1}^{N}=\dots=z_n^N=1\\\sum _{j=1}^{n}a_{j}=k\end{smallmatrix}\right. \right\}\label{eq:Zk(N)_defn}
\end{align}for positive integers $k$ and $N$. With the understanding that empty sums are zero (so  $ \Span_{\mathbb Q}\varnothing=\{0\}$), we have $ \mathfrak Z_1(1)=\{0\}$.
Retroactively, we set $ \mathfrak Z_0(N)\colonequals\mathbb Q$.

A number $ \Li_{a_1,\dots,a_n}(z_1,\dots,z_n)\in \mathfrak Z_{ k}(N)$ is referred to as a cyclotomic multiple zeta value (CMZV) of weight $k$ and level $N$ \cite[(1)]{SingerZhao2020}. Some authors \cite{SingerZhao2020,Au2022a} write $ \mathsf{CMZV}_{k,N}$ or $ \mathsf{CMZV}_k^N$ for $ \mathfrak Z_k(N)$. At level 1, a CMZV  $ \zeta_{a_1,\dots,a_n}=\Li_{a_1,\dots,a_n}(1,\dots,1)\in\mathfrak Z_{a_1+\dots+a_n}(1)$ is simply a multiple zeta value (MZV); at level 2, a CMZV becomes an alternating multiple zeta value (AMZV).

It is clear from the definition \eqref{eq:Zk(N)_defn} that we have a set inclusion $ \mathfrak Z_k(N)\subseteq\mathfrak Z_k(M)$ if $N$ divides $M$. Furthermore, we have an elementary fact that (see \cite[Remark 1.3]{SingerZhao2020} or \cite[Lemma 4.1]{Au2022a}) \begin{align}
\pi i=\frac{N}{N-2}\left[ \Li_1(e^{2\pi i/N })-\Li_1(e^{-2\pi i/N }) \right]\in\mathfrak Z_1(N)\label{eq:piIZ1}
\end{align}when $ N\in\mathbb Z_{\geq3}$.

For each fixed $N$, the $ \mathbb Q$-algebra $ \mathfrak Z(N)\colonequals \sum_{k\in\mathbb Z_{\geq0}}\mathfrak Z_{ k}(N)$ is filtered by weight    \cite[\S1.2]{Goncharov1998}:\begin{align}\mathfrak Z_j(N)\mathfrak
Z_{ k}(N)\subseteq \mathfrak
Z_{j+ k}(N),\quad \text{for all } j,k\in \mathbb Z_{\geq0}.\label{eq:Zk(N)_filt}
\end{align}
In what follows, we will refer to members of $ \mathfrak Z_{ k}(N)$ as Goncharov--Deligne periods,\footnote{Instead of  pointing to Kontsevich--Zagier periods \cite{KontsevichZagier} or  Deligne periods \cite{Deninger1997}  in a broad sense, our discussions of ``period structure'' in the rest of this article will pertain exclusively to Goncharov--Deligne periods.   } in honor of the pioneers  \cite{Goncharov1998,DeligneGoncharov2005,Deligne2010} who worked on them.\footnote{It is worth mentioning that the provable upper bound for   $ \dim_{\mathbb Q}\mathfrak Z_k(N)$ in the joint work of Deligne and Goncharov \cite[Corollaire 5.25]{DeligneGoncharov2005} provides guidance to Au's software \cite[Theorem 5.5]{Au2020} for symbolic reductions of CMZVs.}\begin{theorem}[$ \pi i\,m_k(1+x_1+x_2)$ as Goncharov--Deligne periods]\label{thm:Zk(6)}For $ n\in\mathbb Z_{\geq0}$, define \begin{align}
\mathsf X_3^{(n)}(\tau)\colonequals{}&\int_{-1}^\tau\frac{\log^{n} \mathsf A_{3}(t,\tau)\D t}{\mathsf Q(t)}-\frac{\pi\log^{n}\mathsf C_{3}(\tau)}{\sqrt{3}},&\mathsf Y_3^{(n)}(\tau)\colonequals{}&\frac{\tau -1}{\tau +1}\frac{\log^{n}\mathsf B_3(\tau)}{\mathsf Q(\tau)},\\\mathsf \Xi_3^{(n)}(\tau)\colonequals{}&\int_{-1}^1\left[\log^{n} \mathsf A_{3}(t,\tau)-\log^{n}\mathsf C_{3}(\tau)\right]\frac{\D t}{\mathsf Q(t)},
\end{align}using the rational functions in \eqref{eq:A3B3C3}. For each positive integer $k$, we have \begin{align}\begin{split}&
\pi m_k(1+x_1+x_2)\\\in{}&\Span_{\mathbb Q}\left\{\frac{\sqrt{3}\log^{n_{1}}3}{\pi^{-2n_2}}\left. \int_0^1\mathsf X_3^{(n_{3})}(\tau)\mathsf Y_3^{(n_{4})}(\tau)\D\tau\right| \begin{smallmatrix}n_1,n_2,n_3,n_4\in\mathbb Z_{\geq0}\\n_1+(2n_2+1)+n_3+n_4=k\end{smallmatrix}\right\}\\&{}+\Span_{\mathbb Q}\left\{\frac{\sqrt{3}\log^{n_{1}}3}{\pi^{-2n_2}}\left. \int_1^\varrho\mathsf \Xi_3^{(n_{3})}(\tau)\mathsf Y_3^{(n_{4})}(\tau)\D\tau\right| \begin{smallmatrix}n_1,n_2,n_3,n_4\in\mathbb Z_{\geq0}\\n_1+(2n_2+1)+n_3+n_4=k\end{smallmatrix}\right\}\\\subseteq {}&i\mathfrak Z_{k+1}(6).\end{split}\label{eq:mk(1+x+y)}
\end{align}\end{theorem}\begin{proof}The ``$ \in$'' part is a direct consequence of \eqref{eq:dkW_mk} and \eqref{eq:W3_per}.

To prove the ``$\subseteq$'' part, we follow Brown's algorithmic studies \cite{Brown2009a,Brown2009arXiv,Brown2009b} of
   $  L( \mathfrak M_{0,n}) $, the totality of   homotopy-invariant
iterated integrals \cite[\S1.1]{Brown2009b} on the moduli spaces $ \mathfrak M_{0,n}$ of Riemann spheres with $n$ marked points. Thanks to the filtration property  \eqref{eq:Zk(N)_filt} and the facts that \begin{align}
\log3=-\Li_1(-\varrho)-\Li_1\left( -\frac{1}{\varrho} \right)\in\mathfrak Z_1(6),\quad \pi^2=6\Li_2(1)\in\mathfrak Z_2(6),
\end{align}we may reduce our task to the demonstration of \begin{align}
\left\{\sqrt{3}\int_0^1\mathsf X_3^{(n_{3})}(\tau)\mathsf Y_3^{(n_{4})}(\tau)\D\tau,\sqrt{3}\int_1^\varrho\mathsf \Xi_3^{(n_{3})}(\tau)\mathsf Y_3^{(n_{4})}(\tau)\D\tau\right\}\subset{}&i\mathfrak Z_{k+1}(6)\label{eq:W3_n3n4}
\end{align}for $ n_3,n_4\in\mathbb Z_{\geq0}$ and $ n_3+n_4=k-1$.

Using the notation  of Frellesvig--Tommasini--Wever \cite[\S2]{Frellesvig2016}, we define the generalized polylogarithm\footnote{The GPL $ G(\alpha_1,\dots,\alpha_n;z)$ is implemented in \texttt{Maple}  as $\mathtt{GeneralizedPolylog([\alpha_1,\dots,\alpha_n],z)}$ by Frellesvig \cite{Frellesvig2018Maple}. The same GPL corresponds to $ \mathtt{Hlog(z,[\alpha_1,\dots,\alpha_n])}$ in Panzer's \texttt{HyperInt}
   package \cite{Panzer2015} for \texttt{Maple}, where $ \mathtt{Hlog}$ stands for ``hyperlogarithm'', another name for GPL. We also note that Panzer used ``GPL'' as an abbreviation for ``Goncharov polylogarithm'' \cite[\S1]{Panzer2015}, referring to the same object as our GPL.} (GPL) by the  recursion\begin{align}
G(\alpha_{1},\dots,\alpha_n;z)\colonequals\int_0^z\frac{\D x}{x-\alpha_1}G(\alpha_2,\dots,\alpha_n;x)\label{eq:GPL_rec}
\end{align}  for any  $\bm\alpha=(\alpha_1,\dots,\alpha_n)\neq\bm0\in\mathbb C^n$, along with the additional settings that \begin{align}
G(\underset{n }{\underbrace{0,\dots,0 }};z)\colonequals\frac{\log^nz}{n!},\quad G(-\!\!-;z)\colonequals1.
\end{align}Here, it is always understood that the integration path in \eqref{eq:GPL_rec} is a straight line segment
joining $0$ to $z$.
For convenience, we will use a short-hand $ G({\bm \alpha};z)=G(\alpha_{1},\dots,\alpha_{w(\bm \alpha)};z)$ and refer to $ w(\bm \alpha)$ \big(the number of components in the vector  $ \bm \alpha\in\mathbb C^{w(\bm \alpha)}$\big) as the weight of the GPL  $G({\bm \alpha};z) $.

Set $ \mathfrak G^{(z)}_0(N)=\mathbb Q$ and  $ \mathfrak H^{(z)}_0(N)=\mathbb Q$. For every  $k\in\mathbb Z_{>0}$, we construct two $ \mathbb Q$-vector spaces\footnote{According to Goncharov  \cite[Theorem 2.1]{Goncharov1998}, the space $ \mathfrak G^{(z)}_k(N)$ is spanned by  a special class of GPLs that can be converted to CMZVs, when $z^N=1$. [See also \eqref{eq:MPL_GPL}--\eqref{eq:GPL_MPL} below.] In the space  $ \mathfrak H^{(z)}_k(N)$, one allows  $ \alpha_{k-\ell}=0$,    as in the ``harmonic polylogarithms''   for Ma\^itre's \texttt{HPL} \cite{Maitre2005,Maitre2012} or in the ``hyperlogarithms'' for Panzer's \texttt{HyperInt} \cite{Panzer2015}. }{\begin{align}
\mathfrak G^{(z)}_k(N)\colonequals{}&\Span_{\mathbb Q}\left\{G(\alpha_{1},\dots,\alpha_k;z)\in\mathbb{C}\left | \begin{smallmatrix}\alpha_1^N,\dots,\alpha_{k-1}^{N}\in\{0,1\}\\\alpha_k^N=1\end{smallmatrix}\right.\right\},\label{eq:Gkz(N)_Qvec}\\
\mathfrak H^{(z)}_k(N)\colonequals{}&\Span_{\mathbb Q}\left\{(\pi i)^{\ell}G(\alpha_{1},\dots,\alpha_{k-\ell};z)\in\mathbb{C}\left |\begin{smallmatrix} \alpha_1^N,\dots,\alpha_{k-\ell}^{N}\in\{0,1\}\\0\leq\ell\leq k\end{smallmatrix}\right.\right\}.\label{eq:Hkz(N)_Qvec}\end{align}}The $ \mathbb Q$-algebras $ \mathfrak G^{(z)}(N)\colonequals \sum_{k\in\mathbb Z_{\geq0}}\mathfrak G^{(z)}_k(N)$ and  $ \mathfrak H^{(z)}(N)\colonequals \sum_{k\in\mathbb Z_{\geq0}}\mathfrak H^{(z)}_k(N)$  are  filtered by weight: \begin{align}\mathfrak G^{(z)}_j(N)\mathfrak G^{(z)}_k(N)\subseteq \mathfrak G^{(z)}_{j+k}(N),\quad \mathfrak H^{(z)}_j(N)\mathfrak H^{(z)}_k(N)\subseteq \mathfrak H^{(z)}_{j+k}(N),\end{align} thanks to  a shuffle identity (see \cite[\S5.4]{BorweinBradleyBroadhurstLisonek2001} or \cite[(2.4)]{Frellesvig2016})\begin{align}
G(\alpha_{1},\dots,\alpha_{w_{1}};z)G(\alpha_{w_{1}+1},\dots,\alpha_{w_{1}+w_{2}};z)=\sum_{\sigma\in \text{\cyrins{Ш}} _{w_1,w_2}}G(\alpha_{\sigma(1)},\dots,\alpha_{\sigma(w_{1}+w_{2})}; z),\label{eq:G_shuffle}
\end{align} where $  \text{\cyrins{Ш}} _{w_1,w_2}$ consists of $ \frac{(w_1+w_2)!}{w_1!w_2!}$ permutations $ \sigma\colon\mathbb{Z}\cap[1, w_1+w_2]\longrightarrow\mathbb{Z}\cap[1, w_1+w_2]$ that satisfy $ \sigma^{-1}(m)<\sigma^{-1}(n)$ whenever $1\leq m<n\leq w_1 $ or $ {w_{1}+1}\leq m<n\leq w_1+w_2$.  A multiple polylogarithm (MPL) can be converted to a GPL (see \cite[(1.3)]{Panzer2015} or \cite[(2.5)]{Frellesvig2016}){\small\begin{align} \Li_{a_1,\dots,a_n}(z_1,\dots,z_n)=(-1)^{n}G\left(\smash[b]{\underset{a_1-1 }{\underbrace{0,\dots,0 }}},\frac{1}{z_{1}},\smash[b]{\underset{a_2-1 }{\underbrace{0,\dots,0 }}},\frac{1}{z_{1}z_2},\dots,\smash[b]{\underset{a_n-1 }{\underbrace{0,\dots,0 }}},\frac{1}{\prod_{j=1}^nz_j};1\right)\label{eq:MPL_GPL}\end{align}}when  $ \prod_{j=1}^nz_j\neq0$, and \textit{vice versa} (see \cite[(1.3)]{Panzer2015} or \cite[(2.6)]{Frellesvig2016}){\small\begin{align}
\vphantom{\underset{1}{\underbrace{0}}}G\left(\smash[b]{\underset{a_1-1 }{\underbrace{0,\dots,0 }}},\widetilde \alpha_1,\smash[b]{\underset{a_2-1 }{\underbrace{0,\dots,0 }}},\widetilde \alpha_2,\dots,\smash[b]{\underset{a_n-1 }{\underbrace{0,\dots,0 }}},\widetilde \alpha_n;z\right)=(-1)^n\Li_{a_1,\dots,a_n}\left( \frac{z}{\widetilde \alpha_1} ,\frac{\widetilde \alpha_1}{\widetilde \alpha_2},\dots, \frac{\widetilde \alpha_{n-1}}{\widetilde \alpha_n}\right)\label{eq:GPL_MPL}
\end{align}}when $ \prod_{j=1}^n\widetilde \alpha_j\neq0$.\footnote{The equations \eqref{eq:MPL_GPL} and \eqref{eq:GPL_MPL} can also be used as working  definitions for an MPL when its series form  \eqref{eq:Mpl_defn}  does not converge. } As a result,  we have $ \mathfrak G_{k}^{(1)}(N)= \mathfrak Z_{k}(N)$. Furthermore,  the convergence condition for MPLs can be transferred onto GPLs, allowing us to reformulate \eqref{eq:Gkz(N)_Qvec} and \eqref{eq:Hkz(N)_Qvec} as{\begin{align}
\mathfrak G^{(z)}_k(N)\colonequals{}&\Span_{\mathbb Q}\left\{G(\alpha_{1},\dots,\alpha_k;z)\left | \begin{smallmatrix}\alpha_1^N,\dots,\alpha_{k-1}^{N}\in\{0,1\}\\\alpha_{1}\neq z;\alpha_k^N=1\end{smallmatrix}\right.\right\},\label{eq:Gkz(N)_Qvec'}\tag{\ref{eq:Gkz(N)_Qvec}$'$}\\
\mathfrak H^{(z)}_k(N)\colonequals{}&\Span_{\mathbb Q}\left\{(\pi i)^{\ell}G(\alpha_{1},\dots,\alpha_{k-\ell};z)\left |\begin{smallmatrix} \alpha_1^N,\dots,\alpha_{k-\ell}^{N}\in\{0,1\}\\\alpha_{1}\neq z;0\leq\ell\leq k\end{smallmatrix}\right.\right\}.\label{eq:Hkz(N)_Qvec'}\tag{\ref{eq:Hkz(N)_Qvec}$'$}\end{align}}

To facilitate  further analysis, we  define{\small\begin{subequations}\begin{align}
\mathfrak g_{k}^{(z)}(N)\colonequals{}&\Span_{\mathbb Q}\left\{ G(\bm \alpha;1)G(\bm \beta;z)(\pi i)^{\ell}(\log z)^{k-\ell-w(\bm \alpha)-w(\bm \beta)}\left|\begin{smallmatrix} G(\bm \alpha;1)\in\mathfrak Z_{w(\bm \alpha)}(N)\\G(\bm \beta;z)\in\mathfrak G^{(z)}_{w(\bm \beta)}(N)\\\ell ,w(\bm \alpha),w(\bm \beta)\in\mathbb Z_{\geq0}\\\ell +w(\bm \alpha)+w(\bm \beta)\leq k\end{smallmatrix}\right.\right\}\\\equiv{}&\Span_{\mathbb Q}\left\{ G(\bm \alpha;1)G(\bm \beta;z)(\pi i)^{\ell}\left|\begin{smallmatrix} G(\bm \alpha;1)\in\mathfrak Z_{w(\bm \alpha)}(N)\\G(\bm \beta;z)(\pi i)^{\ell}\in\mathfrak H^{(z)}_{\ell +w(\bm \beta)}(N)\\\ell ,w(\bm \alpha),w(\bm \beta)\in\mathbb Z_{\geq0}\\\ell +w(\bm \alpha)+w(\bm \beta)=k\end{smallmatrix}\right.\right\},\label{eq:gk(N)_GH}
\end{align}\end{subequations}}where the equivalence between two definitions stems from the GPL shuffle algebra \cite[p.\ 24]{DuhrDulat2019}. By the shuffle identity \eqref{eq:G_shuffle}, we have $ \mathfrak g_{j}^{(z)}(N)\mathfrak g_{k}^{(z)}(N)\subseteq \mathfrak g_{j+k}^{(z)}(N)$  for $ j,k\in\mathbb Z_{\geq0}$. Combining \eqref{eq:gk(N)_GH} with the recursive definition of GPL in \eqref{eq:GPL_rec}, we see that if  $ \alpha^N\in\{0,1\}$, then\begin{align}
\int_0^z\frac{g(x)\D x}{x-\alpha}\in\mathfrak g_{k+1}^{(z)}(N)\label{eq:cl_int_N}
\end{align}so long as
the integral in question converges, and  $ g(x)\in \mathfrak g_{k}^{(x)}(N)$ on the path of integration.

In view of  the facts that \begin{align}\log \mathsf A_{3}(t,\tau)\colonequals{}&\log
\frac{\left(1-\frac{t}{\tau}\right)(1-t\tau)}{\mathsf Q(t)}=G(\tau;t)+G\left( \frac{1}{\tau} ;t\right)-G(\varrho;t)-G\left( \frac{1}{\varrho} ;t\right), \\\begin{split}\log \mathsf C_{3}(\tau)\colonequals{}&\log\frac{(1+\tau)^{2}}{3\tau}=2G(-1;\tau)-G(0;\tau)-G(-\varrho;1)-G\left( -\frac{1}{\varrho} ;1\right)\\ \in{}&\mathfrak g_1^{(\tau)}(6),\end{split}
\end{align}and \begin{align}
\frac{1}{\mathsf Q(t)}=\frac{1}{i\sqrt{3}}\left(\frac{1}{t-\varrho}-\frac{1}{t-\frac1\varrho}\right),\quad \frac{\pi i}{3}=G\left( -\frac{1}{\varrho} ;1\right) -G(-\varrho;1)\in\mathfrak g_1^{(\tau)}(6),
\end{align}we have the following arguments for $ k\in\mathbb Z_{>0}$: \begin{align}\begin{split}&
\mathsf X_3^{(k-1)}(\tau)\colonequals\int_{-1}^\tau\frac{\D t}{\mathsf Q(t)}\log^{k-1}\frac{\left(1-\frac{t}{\tau}\right)(1-t\tau)}{\mathsf Q(t)}-\frac{\pi\log^{k-1} \frac{(1+\tau)^{2}}{3\tau}}{\sqrt{3}}\\\in{}&\Span_{\mathbb Q}\left\{ \left.\frac{G(\alpha_{1},\dots,\alpha_{k};z)}{i\sqrt{3}}\right| \begin{smallmatrix}\alpha_{1}\in\left\{ \varrho, \frac1\varrho\right\}\smallsetminus\{\tau\}\\\alpha_2,\dots,\alpha_k\in\left\{ \varrho, \frac1\varrho,\tau,\frac{1}{\tau}\right\}\\z\in\left\{ -1,\tau \right\}\end{smallmatrix}\right\}+\frac{\mathfrak g_{k}^{(\tau)}(6)}{i\sqrt{3}}\subseteq\frac{\mathfrak g_{k}^{(\tau)}(6)}{i\sqrt{3}}.\end{split}\label{eq:X3_per_struct}
\end{align}Here, the ``$\in$'' step issues from the GPL recursion \eqref{eq:GPL_rec} and the shuffle identity \eqref{eq:G_shuffle} (along with Panzer's logarithmic regularizations at $t=\tau$ \cite[\S2.3]{Panzer2015} that preserve the shuffle structure), while the ``$ \subseteq$'' step
is a result of the fibration  basis in Brown's algorithm \cite[Lemma 2.14 and Corollary 3.2]{Panzer2015}. Concretely speaking, through the fibration procedure, one can rewrite  a GPL  $ G(\bm\alpha;\tau)$   \big(in which the components of $ \bm \alpha$ belong to $ \big\{ \varrho, \frac1\varrho,\tau,\frac{1}{\tau}\big\}$\big) as a member in the following $ \mathbb Z$-module:\footnote{Such fibration procedures are automated by Panzer's \texttt{HyperInt} package \cite[Lemma 2.14 and Corollary 3.2]{Panzer2015}, where  $\mathtt{fibrationBasis(Hlog(\tau,[\alpha_{1},\dots,\alpha_{w(\bm \alpha)}]),[\tau])}$ produces a unique representation of $ G(\bm\alpha;\tau)$ as a member of \eqref{eq:fib_prod}. }  \begin{align}
\Span_{\mathbb Z}\left\{G(\beta_{1},\dots,\beta_{w(\bm \alpha)};\tau)\left| \beta_{1},\dots,\beta_{w(\bm\alpha)}\in \left\{\varrho, \frac1\varrho,-1,0,1\right\} \right. \right\},\label{eq:fib_prod}
\end{align}by induction on the weight $ w(\bm \alpha)$ and applications of the GPL recursion \eqref{eq:GPL_rec} (where $ z=\tau$) to the GPL differential form\footnote{As in \cite[\S8.1, p.\ 297]{Weinzierl2022Book}, we  adopt the convention that the 1-form $ \D \log(\alpha_j-\alpha_{j'})$ is  zero when $\alpha_{j}=\alpha_{j'}$. Accordingly, logarithmic regularizations \cite[\S2.3]{Panzer2015} are implicit when we encounter $ G(z,\dots;z)$ and $ G(\dots,0;0)$.   }   \cite[(8.8)]{Weinzierl2022Book} \begin{align}
\begin{split}{}&\D G(\alpha_1,\dots,\alpha_n;z)\\={}&\sum_{j=1}^n G(\alpha_1,\dots,\widehat{\alpha_j},\dots,\alpha_n;z)[\D\log(\alpha_{j-1}-\alpha_j)-\D\log(\alpha_{j+1}-\alpha_j)],\label{eq:GPL_diff_form}\end{split}
\end{align}where an element under the caret is removed, and $\alpha_0\colonequals z $, $ \alpha_{n+1}\colonequals0$. Thus, it is clear that $ G(\bm\alpha;\tau)\in \mathfrak g_{w(\bm \alpha)}^{(\tau)}(6) $.
Likewise, one can show inductively that  $ G(\bm\alpha;-1)$ \big(in which the components of $ \bm \alpha$ belong to $ \big\{ \varrho, \frac1\varrho,\tau,\frac{1}{\tau}\big\}$\big) is a member of \begin{align}
\Span_{\mathbb Z}\left\{G(\beta_{1},\dots,\beta_{j};\tau)G(\beta_{j+1},\dots,\beta_{w(\bm \alpha)};-1)\left| \begin{smallmatrix}\beta_{1},\dots,\beta_{w(\bm\alpha)}\in \big\{\varrho, \frac1\varrho,-1,0,1\big\}\\0\leq j \leq w(\bm\alpha)\end{smallmatrix} \right. \right\}.
\end{align} This justifies the relation $ \mathsf X_3^{(k-1)}(\tau)\in \frac{\mathfrak g_{k}^{(\tau)}(6)}{i\sqrt{3}}$ for every $k\in\mathbb{Z}_{>0}$.

Reasoning in a similar vein as the last paragraph, one can verify that $
\mathsf\Xi_ 3^{(k-1)}(\tau)\in \frac{\mathfrak g_{k}^{(\tau)}(6)}{i\sqrt{3}}$ for $k\in\mathbb{Z}_{>0}$.

With the additional observation that \begin{align}\begin{split}&\mathsf
 Y_3^{(n)}(\tau)\\\colonequals{}&\frac{1}{3}\left(\frac{1}{\tau-\varrho}+\frac{1}{\tau-\frac1\varrho}-\frac{2}{\tau+1}\right)\log^{n} \frac{\tau^{2}}{(1+\tau)^{2}\mathsf Q(\tau)}\\={}&\frac{1}{3}\left(\frac{1}{\tau-\varrho}+\frac{1}{\tau-\frac1\varrho}-\frac{2}{\tau+1}\right)\left[ 2G(0;\tau)-2G(-1;\tau)-G(\varrho;\tau)-G\left( \frac{1}{\varrho} ;\tau\right) \right]^n\end{split}\label{eq:Y3G}
\end{align}together with the scaling property \cite[(2.3)]{Frellesvig2016} of GPL  $ G(\varrho\bm \alpha;\varrho z)=G(\bm \alpha; z)$ when $\alpha_{w(\bm \alpha)}\neq0 $, we may complete all the integrations over $\tau$ (with Panzer's logarithmic regularizations at $\tau=\varrho$) and arrive at\begin{align}
\pi m_k(1+x_1+x_2)\in i \mathfrak G_{k+1}^{(1)}(6)+i \mathfrak g_{k+1}^{(\varrho)}(6)=i \mathfrak G_{k+1}^{(1)}(6)= i \mathfrak Z_{k+1}(6)\label{eq:mk(1+x+y)hex}
\end{align} for all $ k\in\mathbb Z_{>0}$. Here, the set inclusion  $  \mathfrak g_{k+1}^{(\varrho)}(6)\subseteq \mathfrak G_{k+1}^{(1)}(6) $ descends from the relation $ G\big(\frac{1}{\varrho};1\big)=\log(1-\varrho)=-\frac{\pi i}{3}\in \mathfrak G_1^{(1)}(6)$ and parameter rescaling by a factor of $\varrho$ for members of $ \mathfrak G^{(\varrho)}(6)$.

While the last paragraph finishes the proof of \eqref{eq:mk(1+x+y)}, we still have a slightly stronger result: {\small\begin{align}\begin{split}&
\pi i\,m_k(1+x_1+x_2)\\\in{}&\Span_{\mathbb Q}\left\{ (\pi i)^{\ell}G(\alpha_{1},\dots,\alpha_{j};\tau)G(\alpha_{j+1},\ldots,\alpha_{n};z)\log^{r}3\left|\begin{smallmatrix}\alpha_1,\dots,\alpha_{n},z\in\left\{\varrho, \frac1\varrho,-1,0,1\right\}\\\tau\in\{1,\varrho\}\\\ell ,j, n,r\in\mathbb Z_{\geq0};\ell+n+r=k+1\end{smallmatrix}\right. \right\}\\\subseteq{}&\mathfrak g^{(\varrho)}_{k+1}(2).\end{split}\tag{\ref{eq:mk(1+x+y)hex}$'$}\label{eq:mk(1+x+y)hex'}
\end{align}}Here, the ``$\in$'' step  traces back to  \eqref{eq:X3_per_struct}--\eqref{eq:Y3G}, along with all the foregoing fibrations, while the ``$\subseteq$'' step comes from the observation that \begin{align}
\log 3=G(1;\varrho)+2G(-1;\varrho)\in{}& \mathfrak G^{(\varrho)}_{1}(2),\label{eq:log3H}
\end{align} and  the fibration of GPLs with respect to $ \varrho$.
\end{proof}\begin{remark} If one defines a multiple zeta value (MZV) $ \zeta_{a_1,\dots,a_n}=
\linebreak\Li_{a_1,\dots,a_n}(1,\dots,1)$ via the iterative  constructions of GPL on the right-hand side of \eqref{eq:MPL_GPL}, then one gets the Drinfel'd integral representation of MZV (see \cite[\S2]{Drinfeld1990} or \cite[\S9]{Zagier1994}). In general, the representation of an MPL via the iterative integrations for the corresponding GPL is called the Leibniz--Kontsevich formula (see \cite[Theorem 2.2]{Goncharov2001} or \cite[\S1.2]{GoncharovManin2004}) by Goncharov, who ascribed it to a private communication with Kontsevich that generalized the Leibniz representation for $ \zeta_n$ as an $n$-dimensional integral.

By considering mixed Tate motives \cite[(2) and (3)]{GoncharovManin2004}, Goncharov--Manin asked \cite[Conjecture 4.5]{GoncharovManin2004} if certain integral periods on   $ \mathfrak M_{0,n}$ (generalizations of the Drinfel'd integral representation of MZV) are still representable as $ \mathbb Q$-linear combinations of MZVs. This deep question has been answered in the positive by  Brown \cite[Theorem 1.1]{Brown2009b}, whose algorithmic approach \cite{Brown2009a,Brown2009arXiv,Brown2009b}
 to the  Goncharov--Manin periods on  $ \mathfrak M_{0,n}$ [along with generalizations to a larger set    $  L( \mathfrak M_{0,n}) $] has been implemented in Panzer's \texttt{HyperInt}
   package \cite{Panzer2015} for  \texttt{Maple}.
\eor\end{remark}

In the following,  the absolute fibration operation  $  \Fib(f)$ (where $ f$ is an expression involving a formal variable $\varrho$) consists of four tasks  in a sequence:\begin{enumerate}[leftmargin=*,  label={\arabic*.},ref=\arabic*.,
widest=a, align=left]
\item
Evaluate \texttt{fibrationBasis(f)} in  \texttt{HyperInt}, which (among other things) reduces the alternating multiple zeta values (AMZVs)   $ \zeta_{ {A_1}, \dots,A_n}\in\mathfrak Z_{|A_1|+\dots+|A_n|}(2)$ [see \eqref{eq:aMZV_defn}]  and  a special class of multiple polylogarithms    \cite[(2--3)]{BorweinBroadhurstKamnitzer2001} \begin{align}
\Li_{\underset{n}{\underbrace{\scriptstyle{1,\dots,1}}}}(z,\underset{n-1}{\underbrace{1,\dots,1}})=\frac{(-1)^{n}}{n!}\log^n(1-z);
\end{align}
\item Set  $\varrho=\frac{1}{2}+\frac{i\sqrt{3}}{2}$ and $ \delta_\varrho\colonequals\frac{\I \varrho}{|\I\varrho|}=1$, before evaluating logarithms and simplifying all the complex arguments of (multiple) polylogarithms;\item Set $\frac{1}{2}+\frac{i\sqrt{3}}2=\varrho $, $- \frac{1}{2}+\frac{i\sqrt{3}}2=-\frac{1}{\varrho}$, $ - \frac{1}{2}-\frac{i\sqrt{3}}2=-{\varrho}$, and $  \frac{1}{2}-\frac{i\sqrt{3}}2=\frac{1}{\varrho}$, using the \texttt{eval} command in \texttt{Maple};\item Convert to \texttt{Hlog} form in  \texttt{HyperInt}.\end{enumerate}

Define the  $k$-th  Hankel--Broadhurst integral  \begin{align}
\mathscr I^{\mathrm{HB}}_k\colonequals\fib_{\varrho}\Fib\left.\frac{\D ^k W_3(s)}{\D s^k}\right|_{s=0}
\end{align}by computing the $k$-th order derivative of \eqref{eq:W3_per} at $ s=0$ with  \texttt{HyperInt} \big[where $ \mathsf Q(x)\colonequals \big(1-\frac{x}{\varrho}\big)(1- x\varrho)$ and $ \varrho$ is treated as a formal variable\big], before evaluating its absolute fibration (``$\Fib$'') and relative  fibration basis  \big(``$ \fib_\varrho$''\big) with respect to the formal variable $\varrho$.

The last relative fibration sends the $k$-th\ Hankel--Broadhurst integral  to \linebreak $ \mathfrak g_{k+1}^{(\varrho)}(2)=\sum_{\ell=0}^{k+1}\mathfrak H^{(\varrho)}_{\ell}(2)\mathfrak Z_{k+1-\ell}(2)$, where  $\mathfrak H^{(\varrho)}_{\ell}(2) $ is a $ \mathbb Q$-vector space generated by    the  Remiddi--Ver\-maseren \cite{RemiddiVermaseren2000} harmonic polylogarithms\footnote{If all the components of $ \bm\alpha=(\alpha_{1},\dots,\alpha_k)$ belong to $ \{-1,0,1\}$, then \texttt{HPL[\{}$  \alpha_1,\dots,\alpha_k$\texttt{\},  z]} (in Ma\^itre's \texttt{HPL} package \cite{Maitre2005}) is equal to $ (-1)^{n_1(\bm \alpha)}G(\bm \alpha;z)$, where $ n_1(\bm \alpha)\colonequals\#\{j\in\mathbb Z\cap[1,k]|\alpha_j=1\}$ counts the number of 1's in the components of $ \bm\alpha$.  } of weight $\ell$. Members in the latter  $ \mathbb Q$-vector space can be decomposed into Lyndon words \cite{Radford1979} by Ma\^itre's \texttt{HPL} package \cite{Maitre2005}  for \texttt{Mathematica}. We use ``$\Lyn$'' to denote this Lyndon word decomposition together with  subsequent reduction by  ``$ \fib_{\varrho}\Fib$''.

The  processed Lyndon words $\Lyn  \mathscr I^{\mathrm{HB}}_k$, $ 1\leq k\leq4$ can now be fed into Au's  \texttt{Mathematica}    package \texttt{MultipleZetaValues} (v1.2.0), which performs automated (and provable \cite[\S5]{Au2022a}) reductions for members  of  $   \mathfrak Z_k(6),1\leq k\leq5$.\footnote{Prior to Au's analytic work on MPLs, the  Henn--Smirnov--Smirnov database \cite{HennSmirnovSmirnov2017} provided empirical reduction formulae for CMZVs in $   \mathfrak Z_k(6),1\leq k\leq6$.  }   This confirms all the entries in Table \ref{tab:mk_eval_W3}.
It is worth noting that all those tabulated formulae suggest that\footnote{The filtration property of $ \mathfrak Z(3)$ helps us to quickly identify certain terms as Goncharov--Deligne periods at third roots of unity. For example, the relation $ \pi i\zeta_3\in\mathfrak Z_4(3)$ builds upon $\pi i=3[\Li_1(\omega)-\Li_1(\omega^2)]\in\mathfrak Z_1(3)$ [cf.~\eqref{eq:piIZ1}], $\zeta_3=\Li_3(1)\in\mathfrak Z_3(3) $, and $ \mathfrak Z_1(3)\mathfrak Z_3(3)\subseteq\mathfrak Z_4(3)$. } \begin{align}
\pi i\,m_k(1+x_1+x_2)\in{}&\mathfrak Z_{k+1}(3),\label{eq:mkZ3}
\end{align}which is stronger than the relation   \eqref{eq:mk(1+x+y)} established in Theorem \ref{thm:Zk(6)}.

To get ready for the proof of a sharper statement \eqref{eq:mkZ3} in the next theorem, we recall  the multi-dimensional polylogarithm  \begin{align}
L_{a_{1},\dots,a_n}(x)\equiv\Li_{a_1,\dots,a_n}(x,\underset{n-1}{\underbrace{1,\dots,1}})\colonequals \sum_{\ell_{1}>\dots>\ell_{n}>0}\frac{x^{\ell_{1}}}{\prod_{j=1}^n\ell_j^{a_j}}\label{eq:mLi_defn}
\end{align}from the work of Borwein--Broadhurst--Kamnitzer \cite[\S2]{BorweinBroadhurstKamnitzer2001}.
For brevity, we write $ L_{\bm A}(x)\equiv L_{a_1,\dots,a_n}(x)$ for a vector $ \bm A=(a_1,\dots,a_{d(\bm A)})\in\mathbb Z_{>0}^{d(\bm A)}$, while calling $ d(\bm A) $ the depth of $L_{\bm A}$, and $ w(\bm A)\colonequals \sum_{j=1}^{d(\bm A)}a_j$ the weight of $ L_{\bm A}$.
In a trigonometric analysis of the hyper-Mahler measures for $  1+x_1+x_2$, Borwein--Borwein--Straub--Wan \cite[Theorem 26]{BBSW2012Mahler} introduced a $ \mathbb Q$-linear combination over certain multi-dimensional polylogarithms of weight $n\in\mathbb Z_{>1}$ and all possible depths
\begin{align}
\rho_{n}(X)\colonequals\frac{(-1)^nn!}{4^n}\sum_{\substack{w(\bm A)=n,a_1=2\\ a_2,\dots,a_{d(\bm A)}\in\{1,2\}}}4^{d(\bm A)}L_{\bm A}\left(X^2\right),\quad X\in(0,1),
\end{align} and defined additionally that $ \rho_0(X)=1$, $ \rho_1(X)=0$.\begin{theorem}[$ \pi i\,m_k(1+x_1+x_2)$ as Goncharov--Deligne periods, reprise]\label{thm:Zk(3)}For each positive integer $k$, we have {\small\begin{align}\begin{split}&
\pi m_k(1+x_1+x_2)\\={}&2\int_{0}^{\pi/6}\rho_k(2\sin\vartheta)\D \vartheta+2\sum_{j=0}^k\binom kj\int_{\pi/6}^{\pi/2}\rho_j\left( \frac{1}{2\sin\vartheta} \right)\log^{k-j}(2\sin\vartheta)\D \vartheta\in i\mathfrak Z_{k+1}(3).\end{split}\label{eq:BBSW_Zk(6)}
\end{align}}\end{theorem}\begin{proof}The integral representation in \eqref{eq:BBSW_Zk(6)} is due to Borwein--Borwein--Straub--Wan \cite[Theorem 17]{BBSW2012Mahler}.

We reparametrize the integral over $ \vartheta\in\left(0,\frac{\pi}{6}\right)$ with $\vartheta =\arctan\frac{1+t}{\sqrt{3}(1-t)},t\in(-1,0)$, so that \begin{align}
\int_{0}^{\pi/6}\rho_k(2\sin\vartheta)\D \vartheta={}&-\frac{i}{2}\int_{-1}^0\left( \frac{1}{t-\varrho}-\frac{1}{t-\frac{1}{\varrho}} \right)\rho_{k}\left(\frac{1+t}{\sqrt{1-t+t^{2}}}\right)\D t.\end{align}By induction on the weight $ w(\bm A)$ and applications of  the GPL recursion \eqref{eq:GPL_rec} to $ z=\frac{(1+t)^2}{1-t+t^{2}}$, we obtain{\small\begin{align}\begin{split}&
L_{\bm A}\left( \frac{(1+t)^2}{1-t+t^{2}} \right)\\\in{}&\Span_{\mathbb Z}\left\{ (\pi i)^{\ell}G(\alpha_{1},\dots,\alpha_{j};t)G(\alpha_{j+1},\ldots,\alpha_{{w(\bm A)-\ell}};-1)\in \mathbb C\left|\begin{smallmatrix}\alpha_{1},\dots,\alpha_{w(\bm A)-\ell}\in\left\{ 0,-1,\varrho,\frac{1}{\varrho} \right\}\\0\leq\ell\leq w(\bm A);0\leq j\leq w(\bm A)-\ell\end{smallmatrix}\right. \right\}\end{split}
\end{align}}for $ \varrho=e^{\pi i/3}$ and {\small\begin{align}\begin{split}&
i\int_{0}^{\pi/6}\rho_k(2\sin\vartheta)\D \vartheta\\\in{}&\Span_{\mathbb Q}\left\{ (\pi i)^{\ell}G(\alpha_{1},\dots,\alpha_{{k+1-\ell}};-1)\in\mathbb C\left|\begin{smallmatrix}\alpha_{1},\dots,\alpha_{k+1-\ell}\in\left\{ 0,-1,\varrho,\frac{1}{\varrho} \right\}\\0\leq\ell\leq k+1\end{smallmatrix}\right. \right\}\\\subseteq{}&\Span_{\mathbb Q}\left\{ (\pi i)^{k+1-n}G(\alpha_{1},\dots,\alpha_{n};1)\left|\begin{smallmatrix}\\\alpha_{1},\dots,\alpha_{n}\in\left\{ 0,1,-\varrho=\frac{1}{\varrho^{2}}=\omega^{2},-\frac{1}{\varrho} =\varrho^{2}=\omega\right\}\\0\leq n\leq k+1;\alpha_{1}\neq1;\alpha_n\neq0\end{smallmatrix}\right. \right\}\subseteq\mathfrak Z_{k+1}({3})\end{split}\label{eq:Zk3a}
\end{align}}for  $ \omega=e^{2\pi i/3}$.

Arguing  as  in the last paragraph, we can check that \begin{align}
i\int_{0}^{\pi/6}\log^{k}(2\sin\vartheta)\D \vartheta\in \mathfrak Z_{k+1}({3}).\label{eq:Ls_varrho}
\end{align}Meanwhile, a constructive proof for  the relation \cite[\S2.1]{BorweinStraub2012Mahler}\begin{align}
\int_{0}^{\pi/2}\log^{k}(2\sin\vartheta)\D \vartheta\in\pi \mathfrak Z_{k}(1)\subseteq\pi \mathfrak Z_{k}(3)\label{eq:BorweinStraubZ1}
\end{align}is known. Combining the last two displayed equations with the fact that  $ \pi i\in\mathfrak Z_1(3)$ [see \eqref{eq:piIZ1}], we get \begin{align}
i\int_{\pi/6}^{\pi/2}\log^{k}(2\sin\vartheta)\D \vartheta\in \mathfrak Z_{k+1}({3}).
\end{align}

Each one of the remaining integrals $ \int_{\pi/6}^{\pi/2}\rho_j\left( \frac{1}{2\sin\vartheta} \right)\log^{k-j}(2\sin\vartheta)\D \vartheta$ (where $ 2\leq j\leq k$)  contributes a summand to the right-hand side of a vanishing identity {\begin{align}\begin{split}
0= {}&\lim_{\varepsilon\to0^+}\left(\int_{i\varepsilon}^{1+i\varepsilon}+\curvearrowleft\kern-1.2em\int_{1-\varepsilon}^{(1-\varepsilon)\varrho}+ \curvearrowleft\kern-1.2em\int_{\varrho}^{1/\varrho}+\right.\\{}&\left.+\curvearrowleft\kern-1.2em\int^{1-\varepsilon}_{(1-\varepsilon)/\varrho}+\int_{1-i\varepsilon}^{-i\varepsilon}\right)\rho_j\left( \sqrt{\frac{-z}{(1-z)^{2}}} \right)\left[\log\frac{(1-z)^{2}}{-z}\right]^{k-j}\frac{\D z}{iz},\end{split}
\end{align}}where the path of integration for each $^{_{_\curvearrowleft}}\kern-.95em\int$ runs counterclockwise along a circular arc centered at the origin.
Concretely speaking, we have \begin{align}\begin{split}&
\curvearrowleft\kern-1.2em\int_{\varrho}^{1/\varrho}\rho_j\left( \sqrt{\frac{-z}{(1-z)^{2}}} \right)\left[\log\frac{(1-z)^{2}}{-z}\right]^{k-j}\frac{\D z}{iz}\\={}&2^{k-j+2}\int_{\pi/6}^{\pi/2}\rho_j\left( \frac{1}{2\sin\vartheta} \right)\log^{k-j}(2\sin\vartheta)\D \vartheta,\end{split}
\end{align}which is equal to the sum of \begin{align}\begin{split}&
- \lim_{\varepsilon\to0^+}\left(\int_{i\varepsilon}^{1+i\varepsilon}+\int_{1-i\varepsilon}^{-i\varepsilon}\right)\rho_j\left( \sqrt{\frac{-z}{(1-z)^{2}}} \right)\left[\log\frac{(1-z)^{2}}{-z}\right]^{k-j}\frac{\D z}{iz}\\\in{}&\Span_{\mathbb Q}\left\{\left. \int_0^1\rho_j\left( \sqrt{\frac{-z}{(1-z)^{2}}} \right)\left[\log\frac{(1-z)^{2}}{z}\right]^{k-j-n}(\pi i)^n\frac{\D z}{iz}\right| 0\leq n\leq k-j\right\}\end{split}\label{eq:keyhole_key}
\end{align}and {\begin{align}
\begin{split}&
- \lim_{\varepsilon\to0^+}\left(\curvearrowleft\kern-1.2em\int_{1-\varepsilon}^{(1-\varepsilon)\varrho}+\curvearrowleft\kern-1.2em\int^{1-\varepsilon}_{(1-\varepsilon)/\varrho}\right)\rho_j\left( \sqrt{\frac{-z}{(1-z)^{2}}} \right)\left[\log\frac{(1-z)^{2}}{-z}\right]^{k-j}\frac{\D z}{iz}\\\in{}&i\Span_{\mathbb Q}\left\{\curvearrowleft\kern-1.2em\int^{\varrho}_{1}(\pi i)^{n}Z_{\ell}G\left(\alpha_{1},\dots,\alpha_{j-\ell-n};\frac{(1-z)^{2}}{-z}\right)\times \right.\\{}&\left.\left.\times\left[\log\frac{(1-z)^{2}}{-z}\right]^{k-j}\frac{\D z}{iz}\right|\begin{smallmatrix}Z_\ell\in\mathfrak Z_\ell(1)\\\alpha_{1},\dots,\alpha_{{j-\ell-n}}\in\left\{ 0,1 \right\}\\ 0\leq n\leq j;0\leq\ell\leq j-n\end{smallmatrix} \right\}\\{}&+i\Span_{\mathbb Q}\left\{ \curvearrowleft\kern-1.2em\int^{1}_{1/\varrho}(\pi i)^{n}Z_{\ell}G\left(\alpha_{1},\dots,\alpha_{j-\ell-n};\frac{(1-z)^{2}}{-z}\right)\times \right.\\{}&\left.\times\left.\left[\log\frac{(1-z)^{2}}{-z}\right]^{k-j}\frac{\D z}{iz}\right|\begin{smallmatrix}Z_\ell\in\mathfrak Z_\ell(1)\\\alpha_{1},\dots,\alpha_{{j-\ell-n}}\in\left\{ 0,1 \right\}\\ 0\leq n\leq j;0\leq\ell\leq j-n\end{smallmatrix} \right\}.\end{split}\label{eq:keyhole_hole}
\end{align}}Here,   the multi-dimensional polylogarithm behaves like  $ L_{\bm A}(\varXi\pm i0^{+})\in \mathfrak g_{w(\bm A)}^{(1/\varXi)}(1)+\pi i \mathfrak g_{w(\bm A)-1}^{(1/\varXi)}(1)$, with   a  jump discontinuity \cite{Maitre2012,Panzer2015}  $ L_{\bm A}(\varXi+i0^{+})-L_{\bm A}(\varXi-i0^{+})\in\pi i \mathfrak g_{w(\bm A)-1}^{(1/\varXi)}(1) $ [see \eqref{eq:gk(N)_GH} for the definition of $\mathfrak g^{(z)}_k(N)$], across its branch cut where  $ \varXi>1$.   To handle the right-hand side of \eqref{eq:keyhole_key}, we may switch back to the parameter $t=\frac{  \varrho z-1}{z- \varrho}$, as follows:{\allowdisplaybreaks\begin{align}\begin{split}&
\int_0^1\rho_j\left( \sqrt{\frac{-z}{(1-z)^{2}}} \right)\left[\log\frac{(1-z)^{2}}{z}\right]^{k-j-n}(\pi i)^n\frac{\D z}{iz}\\={}&i\int_{-1}^{1/\varrho}\rho_j\left( \frac{\sqrt{1-t+t^2}}{1+t} \right)\left[\log\frac{-(1+t)^{2}}{1-t+t^2}\right]^{k-j-n}(\pi i)^n\left( \frac{1}{t-\varrho}-\frac{1}{t-\frac{1}{\varrho}} \right)\D t\\\in{}&i\Span_{\mathbb Q}\left\{ (\pi i)^{k+1-\ell}G(\alpha_{1},\dots,\alpha_{j};u) G(\alpha_{j+1},\dots,\alpha_{\ell};{\varrho})\left|\begin{smallmatrix}\\\alpha_{1},\dots,\alpha_{{\ell}}\in\left\{ 0,-1,\varrho,\frac{1}{\varrho} \right\}\\u\in\left\{ -1,\frac{1}{\varrho}\right\}\\\\\ 0\leq j\leq\ell\leq k+1\end{smallmatrix}\right. \right\}\\\subseteq{}&i\Span_{\mathbb Q}\left\{ (\pi i)^{k+1-\ell}G(\alpha_{1},\dots,\alpha_{\ell};1)\left|\begin{smallmatrix}\\\alpha_{1},\dots,\alpha_{{\ell}}\in\left\{ 0,1,\omega=-\frac{1}{\varrho}=\varrho^{2},\omega^{2}=-\varrho=\frac{1}{\varrho^{2}} \right\}\\0\leq\ell\leq k+1;\alpha_{1}\neq1;\alpha_\ell\neq0\end{smallmatrix}\right. \right\}\\\subseteq{}& i\mathfrak Z_{k+1}(3). \end{split}
\end{align}}Likewise, by integrating GPLs over  $ t=\frac{  \varrho z-1}{z- \varrho}\in[-1,0]$, one can identify the right-hand side of \eqref{eq:keyhole_hole} with a subset of $i\mathfrak Z_{k+1}(3) $, as we did in \eqref{eq:Zk3a}.

Therefore, we have $ \pi m_k(1+x_1+x_2)\in i\mathfrak Z_{k+1}(3)$ as claimed. \end{proof}
\begin{remark}
If we reparametrize all the integrals in \eqref{eq:BBSW_Zk(6)} with $\vartheta =\linebreak\arctan\frac{1+t}{\sqrt{3}(1-t)}$, $t\in(-1,1)$, then we have \begin{align} \begin{split}\mathscr I^{\mathrm{BBSW}}_k\colonequals{}&
\fib_\varrho\int_{-1}^0\left( \frac{1}{t-\varrho}-\frac{1}{t-\frac{1}{\varrho}} \right)\frac{\rho_{k}\left(\frac{1+t}{\sqrt{\mathsf Q(t)}}\right)}{\pi i}\D t\\&{}+\sum_{j=0}^k\binom kj\fib_\varrho\int_{0}^1\left( \frac{1}{t-\varrho}-\frac{1}{t-\frac{1}{\varrho}} \right)\frac{\rho_j\left( \frac{\sqrt{\mathsf Q(t)}}{1+t} \right)}{2^{k-j}\pi i}\log^{k-j}\frac{(1+t)^{2}}{\mathsf Q(t)}\D t,\end{split}\tag{\ref{eq:BBSW_Zk(6)}$'$}
\end{align}where $ \mathsf Q(t)\colonequals \big(1-\frac{t}{\varrho}\big)(1- t\varrho)$.
Evaluating in \texttt{HyperInt} and \texttt{HPL}, one can check that  $ \Lyn \mathscr I^{\mathrm{BBSW}}_k=\Lyn  \mathscr I^{\mathrm{HB}}_k$, $ k\in\{1,2,3,4\}$.\eor\end{remark}
\begin{remark}
Let \begin{align}
\Ls_k(\varphi)\colonequals-\int_0^\varphi\log^{k-1}\left\vert 2\sin\frac{\vartheta}{2} \right\vert\D\vartheta
\end{align}be Lewin's log-sine integral \cite{Lewin,Lewin1981} for $ k\in\mathbb Z_{>0}$. From Borwein--Straub \cite[Theorem 3]{BorweinStraub2011ISSAC}, one sees that $ \Ls_k(2\pi/3)\in i\mathfrak Z_k(3)$. According to \eqref{eq:Ls_varrho}, we also have  $ \Ls_k(\pi/3)\in i\mathfrak Z_k(3)$, which is not immediate from the Borwein--Straub reduction \cite[Theorem 3]{BorweinStraub2011ISSAC} of log-sine integrals.   Borwein--Borwein--Straub--Wan \cite[(108) and (109)]{BBSW2012Mahler} proposed conjectural expressions for $ m_3(1+x_1+x_2)$ and $ m_4(1+x_1+x_2)$, using log-sine integrals at $ \pi/3$ and $ 2\pi /3$. These  conjectures are  equivalent to the corresponding formulae in Table \ref{tab:mk_eval_W3}, as  one can check with Au's algorithm \cite[\S5]{Au2022a}. \eor\end{remark}\begin{remark}Let \begin{align}
\Ls_k^{(m)}(\varphi)\colonequals-\int_0^\varphi\vartheta^{m}\log^{k-1-m}\left\vert 2\sin\frac{\vartheta}{2} \right\vert\D\vartheta
\end{align}be Lewin's generalized log-sine integral \cite[Appendix A.1(20)]{Lewin1981} for  $ k,m\in\mathbb Z_{>0}$.  By the explicit construction of Borwein--Straub \cite[Theorem 3]{BorweinStraub2011ISSAC}, we have  $ \Ls_k^{(m)}(2\pi/3)\in i^{m+1}\mathfrak Z_{k}(3)$. Meanwhile, we note that the proof of the theorem above can be readily adapted to \begin{align}\begin{split}&
\Ls_k^{(m)}\left(\frac{\pi}{3}\right)\\={}&-\frac{1}{2^{k-1-m}}\curvearrowleft\kern-1.2em\int_{1}^\varrho\left( \frac{\log z}{i} \right)^m\left[ \log\frac{(1-z)^2}{-z} \right]^{k-1-m}\frac{\D z}{ iz}\\\in{}& i^{m+1}\Span_{\mathbb Q}\left\{ (\pi i)^{k-\ell}G(\alpha_{1},\dots,\alpha_{\ell};u)\in\mathbb C\left|\begin{smallmatrix}\\\alpha_{1},\dots,\alpha_{{\ell}}\in\left\{ 0,1 \right\};u\in\{1,\varrho\}\\0\leq\ell\leq k;\alpha_\ell\neq0\end{smallmatrix}\right. \right\}\\\subseteq{}& i^{m+1}\mathfrak Z_{k}(3).\end{split}
\end{align} Here, the ``$ \subseteq$'' step builds upon a  more general result \begin{align}
\Span_{\mathbb Q}\left\{ G(\alpha_{1},\dots,\alpha_{k};z)\left|\begin{smallmatrix}\\\alpha_{1},\dots,\alpha_ k,z\in\left\{ 0,1,\varrho,\frac{1}{\varrho} \right\}\\\alpha_{1}\neq z;\alpha_k\neq0\end{smallmatrix}\right. \right\}\subseteq\mathfrak Z_k(3).\label{eq:Zk3Au}
\end{align} To prove this, use the GPL recursion \eqref{eq:GPL_rec} (with an integration path connecting  $\omega^2$ and $z$) to show inductively that \begin{align}\begin{split}&
\Span_{\mathbb Q}\left\{ G(\alpha_{1},\dots,\alpha_{k};w)\left|\begin{smallmatrix}\\\alpha_{1},\dots,\alpha_{{k}},w\in\left\{ 0,1,\frac{\omega^{2}-1}{z-1} ,\frac{z}{\omega^{2}}\frac{\omega^{2}-1}{z-1}\right\}\\\alpha_{1}\neq1;\alpha_k\neq0\end{smallmatrix}\right. \right\}\\\subseteq{}&\Span_{\mathbb Q}\left\{ (\pi i)^{k-\ell}G(\alpha_{1},\dots,\alpha_{j};z)G(\alpha_{j+1},\dots,\alpha_{\ell};\omega^{2})\left|\begin{smallmatrix}\alpha_{1},\dots,\alpha_{{\ell}}\in\left\{ 0,1,\omega^{2}\right\}\\0\leq j\leq\ell\leq k\end{smallmatrix}\right. \right\},\end{split}
\end{align} before specializing to   $ z=\omega\colonequals e^{2\pi i/3}$ (with logarithmic regularizations \cite[\S2.3]{Panzer2015} in case of $ \alpha_1=z$ and $ \alpha_{j+1}=\omega^2$). \eor\end{remark}
\subsection{Period structure of $ m_k(1+x_1+x_2+x_3)$\label{subsec:Zk(2)}}With the $ \mathbb Q$-vector space $ \mathfrak g_k^{(z)}(2)$ introduced in  \eqref{eq:gk(N)_GH}, we can  extend our proof of Theorem \ref{thm:Zk(6)} to the $k$-Mahler measures of $ 1+x_1+x_2+x_3$.\begin{theorem}[$ \pi^2 m_k(1+x_1+x_2+x_3)$ as Goncharov--Deligne periods]\label{thm:Zk(2)}For $ n\in\mathbb Z_{\geq0}$, define\footnote{We set the Kronecker delta as $ \delta_{a,b}=1$ when $ a=b$, and $ \delta_{a,b}=0$ otherwise.} {\small\begin{align}
\mathsf X_4^{(n)}(\tau)\colonequals{}&\curvearrowleft\kern-1.15em\int_{1}^{\tau}\frac{[\log\mathsf A_{4}(t,\tau)+\log\mathsf B_4(\tau)]^{n}\D t}{t}-\pi i\delta_{n,0},& \widetilde{\mathsf X}_4^{(n)}(\tau)\colonequals{}&\curvearrowleft\kern-1.15em\int_{1}^{\tau}\frac{\log^{n} \mathsf A_{4}(t,\tau)\D t}{t},\\\mathsf \Xi_4^{(n)}(\tau)\colonequals{}&\curvearrowleft\kern-1.15em\int_{1}^{-1}\frac{[\log\mathsf A_{4}(t,\tau)+\log\mathsf B_4(\tau)]^{n}-\delta_{n,0}}{t}\D t,&\widetilde{\mathsf \Xi}_4^{(n)}(\tau)\colonequals{}&\curvearrowleft\kern-1.15em\int_{1}^{-1}\frac{\log^{n} \mathsf A_{4}(t,\tau)}{t}\D t,
\end{align}}using the rational functions in \eqref{eq:A4B4}. For each positive integer $k$, we have  \begin{align}\begin{split}&
\pi^{2} m_k(1+x_1+x_2+x_3)-\pi^{2} m_k(1+x_1)\\\in{}& \Span_{\mathbb Q}\left\{ \left.\frac{1}{\pi^{-2n_1}}\curvearrowleft\kern-1.2em\int_{1}^{-1}\widetilde{\mathsf X}_4^{(n_{2})}(\tau)\mathsf X_4^{(n_{3})}(\tau)\frac{\mathsf \tau+1}{\tau-1}\frac{\D\tau}{\tau}\right| \begin{smallmatrix}n_1,n_2,n_3\in\mathbb Z_{\geq0}\\(2n_1+1)+n_{2}+n_3=k\end{smallmatrix}\right\}\\{}&+\Span_{\mathbb Q}\left\{ \left.\frac{1}{\pi^{-2n_1}}\int_{-1}^{0}\widetilde{\mathsf \Xi}_4^{(n_{2})}(\tau)\mathsf \Xi_4^{(n_{3})}(\tau)\frac{\mathsf \tau+1}{\tau-1}\frac{\D\tau}{\tau}\right| \begin{smallmatrix}n_1,n_2,n_3\in\mathbb Z_{\geq0}\\(2n_1+1)+n_{2}+n_3=k\end{smallmatrix}\right\}\\\subseteq{}&\mathfrak Z_{k+2}(2),\end{split}\label{eq:W4-W2}
\end{align}which also implies \begin{align}
\pi^{2} m_k(1+x_1+x_2+x_3)\in \mathfrak Z_{k+2}(2).\label{eq:mk_Zk+2(2)}
\end{align}\end{theorem}\begin{proof}As we may recall, Borwein--Straub \cite[(2.4), (2.5), (4.3)]{BorweinStraub2012Mahler} devised a recursion for $ m_k(1+x_1),k\in\mathbb Z_{>0}$ and showed that \begin{align}\begin{split}
m_k(1+x_1)\in{}&\Span_{\mathbb Q}\left\{\left.\prod_{j=1}^n\zeta(a_{j})\right|a_{j}\in\mathbb Z_{>1},1\leq j\leq n;\sum_{j=1}^n a_j=k\right\}\\\subseteq{}&\mathfrak Z_{k}(1)\subseteq\mathfrak Z_{k}(2),\end{split}
\end{align}which simplified an earlier proof of the same statement by Kurokawa--Lal\'in--Ochiai \cite[Proposition 4]{KurokawaLalinOchiai2008}. Thus, it indeed suffices to work out \eqref{eq:W4-W2} for our goal in \eqref{eq:mk_Zk+2(2)}.

Again, owing to the filtration property  \eqref{eq:Zk(N)_filt} and the fact that \begin{align}
\pi^2=6\Li_2(1)\in\mathfrak Z_2(2),
\end{align}we may reduce our task to the demonstration of \begin{align}
\left\{\curvearrowleft\kern-1.2em\int_{1}^{-1}\widetilde{\mathsf X}_4^{(n_{2})}(\tau)\mathsf X_4^{(n_{3})}(\tau)\frac{\mathsf \tau+1}{\tau-1}\frac{\D\tau}{\tau},\int_{-1}^{0}\widetilde{\mathsf \Xi}_4^{(n_{2})}(\tau)\mathsf \Xi_4^{(n_{3})}(\tau)\frac{\mathsf \tau+1}{\tau-1}\frac{\D\tau}{\tau} \right\}\subset{}&\mathfrak Z_{k+2}(2)\label{eq:W4_n3n4}
\end{align}for $ n_2,n_3\in\mathbb Z_{\geq0}$ and $ n_2+n_3=k-1$.

With \begin{align}\log\mathsf A_{4}(t,\tau)\colonequals{}&\log
\frac{\left(1-\frac{t}{\tau}\right)\left(1-t\tau\right)}{t}=G(\tau;t)+G\left(\frac{1}{\tau};t\right)-G(0;t)
\end{align}in hand, we have the following result for $ k\in\mathbb Z_{>0}$: \begin{align}\begin{split}&\widetilde{\mathsf X}_4^{(k)}(\tau)\colonequals{}\curvearrowleft\kern-1.2em\int_1^{\tau}\frac{\D t}{t}\log^k
\frac{\left(1-\frac{t}{\tau}\right)\left(1-t\tau\right)}{t}
\\\in{}&\Span_{\mathbb Q}\left\{ G(0,\alpha_{2},\dots,\alpha_{k+1};z)\left|\begin{smallmatrix}\alpha_{2},\dots,\alpha_{k+1}\in\left\{\tau,\frac{1}{\tau},0\right\}\\z\in\{1,\tau\}\end{smallmatrix}\right.\right\}\\\subseteq{}& \mathfrak g^{(\tau)}_{k+1}(2).\end{split}\end{align}Here, the rationales behind both the ``$\in$'' step and the ``$ \subseteq$'' step are identical to what we provided for \eqref{eq:X3_per_struct}. For the ``$ \subseteq$'' step, the fibration output resides in the following $ \mathbb Z$-module:\begin{align}
\Span_{\mathbb Z}\left\{ (\pi i)^{\ell}G(\beta_{1},\dots,\beta_{j};\tau)Z_{m}\in\mathbb C\left| \begin{smallmatrix}\beta_1,\dots,\beta_j\in\{-1,0,1\};Z_{m}\in\mathfrak{Z}_{m}(1)\\j,\ell,m\in\mathbb Z_{\geq0};j+\ell+ m=k+1\\\end{smallmatrix} \right.\right\},
\end{align}
as we inductively apply the GPL recursion \eqref{eq:GPL_rec} (with $ z=\tau$ and $z=1$) to the GPL differential form \eqref{eq:GPL_diff_form} \big[where $ \D\log(\alpha_{j-1}-\alpha_{j}),\D\log(\alpha_{j+1}-\alpha_{j})\in\big\{\D\log\tau,\D\log(\tau-1),\D\log\left( \frac{1}{\tau}-1 \right),\D\log\left(\tau-\frac{1}{\tau} \right),0\big\}$\big]. Taking  into account\begin{align}
\log\mathsf B_4(\tau)\colonequals{}&\log \frac{-\tau}{(1-\tau)^2}\equiv\pi i+G(0;\tau)-2G(1;\tau)\pmod{2\pi i\mathbb{Z}},
\end{align}we may check that  $\mathsf X_4^{(k)}(\tau)\in\pi i\log^k\mathsf B_4(\tau)+\Span_{\mathbb Q} \big\{\widetilde{\mathsf X}_4^{(j)}(\tau)\log^{k-j}\mathsf B_4(\tau)\big|0\leq j\leq k\big\}\subseteq \mathfrak g^{(\tau)}_{k+1}(2) $.

Likewise,  one can verify that   $ \big\{\widetilde {\mathsf \Xi}_4^{(k)}(\tau),\mathsf \Xi_4^{(k)}(\tau)\big\}\subset i \mathfrak g^{(\tau)}_{k+1}(2)$ holds for all  $k\in\mathbb Z_{\geq0}$.

Thus far, we see that \begin{align}
\widetilde{\mathsf X}_4^{(n_{2})}(\tau)\mathsf X_4^{(n_{3})}(\tau)\in{}& \mathfrak g^{(\tau)}_{n_2+n_3+2}(2),\\\widetilde{\mathsf \Xi}_4^{(n_{2})}(\tau)\mathsf \Xi_4^{(n_{3})}(\tau)\in{}& \mathfrak g^{(\tau)}_{n_2+n_3+2}(2).
\end{align} As we expand\begin{align}
\frac{\mathsf \tau+1}{\tau-1}\frac{1}{\tau}=\frac{2}{\tau-1}-\frac{1}{\tau},
\end{align} and invoke the closure property \eqref{eq:cl_int_N} for integrations over $\tau$, we may deduce $ \pi^2 m_k(1+x_1+x_2+x_3)\in \mathfrak g^{(1)}_{k+2}(2)\cap\mathbb R=\mathfrak Z_{k+2}(2)$. \end{proof}

Thanks to the data mine of Bl\"umlein--Broadhurst--Vermaseren \cite{MZVdatamine2010},
we can reduce every AMZV (up to weight $8$) to a polynomial in $ \mathbb Q[\log2,\pi^2,\zeta_3,\zeta_{-3,1},\zeta_{5},\linebreak\zeta_{-3,1,1},\dots]$, using a  look-up table that ships with \texttt{HyperInt}. This leads us to the evaluations in Table \ref{tab:mk_eval}.

All the entries in the aforementioned look-up table for $ \mathfrak Z_k(2),1\leq k\leq8$ are analytically provable by Deligne's standard relations  \cite[Definition 1.1]{Zhao2010} for  $\mathbb Q$-linearly dependent multiple polylogarithms.\footnote{Some provable non-standard $ \mathbb Q$-linear relations in $ \mathfrak Z(N)$ are covered by Au's method \cite[\S5.1]{Au2022a}, where $ N\in\{6,8,10,12\}$.}

Before closing this section, we compare our results to  some related works on (hyper-)Mahler measures.

In their foundational paper \cite[Theorem 3]{KurokawaLalinOchiai2008}, Kuro\-kawa, Lal\'in, and Ochiai established a
remarkable connection between  the  $ \mathbb Q$-vector space $ \mathfrak Z_k(1)$ incorporating MZVs of weight $k$ and the hyper-Mahler measure $ m_k(1+x_1)$ for $k\in\mathbb Z_{>1}$:\begin{align}
m_k(1+x_1)=(-1)^kk!\sum_{\substack{a_1+\dots+a_n=k\\a_j\geq2,1\leq j\leq n}}\frac{\zeta_{{a_1}, \dots,a_n}}{2^{2n}}\in\mathfrak Z_k(1).\label{eq:KLOsum}
\end{align}The same extends to the $k=1$ case, where $ m_1(1+x_1)=0$ and  an empty sum over  $ \mathfrak Z_1=\{0\}$ must vanish.\

  Using Lal\'in's method \cite{Lalin2003}, Kurokawa--Lal\'in--Ochiai   originally demonstrated  \eqref{eq:KLOsum}  by writing the hyper-Mahler measures  $ m_k(1+x_1)$ as analogs of the Drinfel'd integral representation of MZVs (see \cite[\S2]{Drinfeld1990} or \cite[\S9]{Zagier1994}), the latter of which  are special cases of Goncharov--Manin periods \cite{GoncharovManin2004} on $ \mathfrak M_{0,n}$.

Our Theorems \ref{thm:Zk(3)} and \ref{thm:Zk(2)}  extend the Kurokawa--Lal\'in--Ochiai sum rule \eqref{eq:KLOsum} to Goncharov--Deligne periods in $ \mathfrak Z_{k+1}(3)$ and $ \mathfrak Z_{k+2}(2)$ for $ k\in\mathbb Z_{>0}$:\begin{align}
\pi i\,m_k^{}(1+x_1+x_2)\in {}&\mathfrak Z_{k+1}^{}(3),\\(\pi i)^{2} m_k^{}(1+x_1+x_2+x_{3})\in {}&\mathfrak Z_{k+2}^{}(2).
\end{align}If we set $ m_0(1+x_1+x_2)=m_0(1+x_1+x_2+x_{3})=1$, then the same is true for  $k=0$.

The hyper-Mahler measures for the multivariate family $ R_m(x_1,\dots, x_m,z)=z+\prod_{j=1}^m\frac{1-x_j}{1+x_j}$  are expressible as $ \mathbb Q$-linear combinations of multiple polylogarithms whose arguments are fourth roots of unity, as established by Lal\'in--Lechasseur \cite[Theorem 1.1]{LalinLechasseur2016}. The results of  Lal\'in--Lechasseur hinged on a representation of $ m_k(a+x_{1})$ via multi-dimensional polylogarithms of $|a|^{-2\sgn\log |a|}$ \cite[Theorem 4.5]{LalinLechasseur2016} (extending a result of Akatsuka \cite[Theorem 7]{Akatsuka2009}, in a similar spirit as Borwein--Borwein--Straub--Wan \cite[Theorem 26]{BBSW2012Mahler}) and subsequent computations of certain motivic periods on   $ \mathfrak M_{0,n}$ \cite[\S\S7--8]{LalinLechasseur2016}.

We are not sure whether there are systematic classifications of all  the hyper-Mahler measures $ m_k\big(1+\sum_{j=1}^\ell x_j\big)$, where $k$ and $\ell $ are positive integers. The  conjectural characterizations of $ m_1\big(1+\sum_{j=1}^4x_j\big)$ and  $ m_1\big(1+\sum_{j=1}^5x_j\big)$ by Rodr\'iguez-Villegas \cite[\S8]{BoydLindRVDeninger2003} seem to suggest that the moduli spaces $\mathfrak M_{g,n} $ for $g>0$   are needed (probably by the routes of elliptic polylogarithms \cite{Rogers2011IMRN,BlochKerrVanhove2015,Samart2020} and elliptic multiple zeta values \cite{ZerbiniThesis,BSZ2019,ZagierZerbini2020}) to quantify these logarithmic Mahler measures for polynomials in $4$ and $5$ variables.

  At present, we do not have natural generalizations of Proposition \ref{prop:WS} to \begin{align}\int_0^\infty J_{1+\frac{s}{2}}\left(\sqrt{u}x\right ) \frac{J_{0}^3(x)}{x^{s/2}}\D x\text{ or }\int_0^\infty J_{1+\frac{s}{2}}\left(\sqrt{u}x\right ) \frac{J_{0}^{4}(x)}{x^{s/2}}\D x,\end{align}where $s$ is generic.
 On a different note, for $ s=-2$, Kluyver's random walk integrals \cite{Kluyver1906}\begin{align}
\int_0^\infty J_{0}\left(\sqrt{u}x\right ) J_{0}^3(x)x\D x\text{ and }\int_0^\infty J_{0}\left(\sqrt{u}x\right ) J_{0}^4(x)x\D x
\end{align} have been scrutinized by Borwein and coworkers \cite{BNSW2011,BSWZ2012,BSW2013}. A potential link from these random walk integrals to the two conjectures of  Rodr\'iguez-Villegas \cite[\S8]{BoydLindRVDeninger2003} has been explored by Straub--Zudilin \cite[\S\S5--7]{StraubZudilin2018} and Zudilin \cite[Chapter 6]{BrunaultZudilin2020}.

\section{Some infinite series behind hyper-Mahler measures\label{sec:series_Zk(N)}}

\subsection{Non-linear Euler sums revisited}We open this section with a non-exhaustive  overview of some  infinite sums that revolve around
the harmonic numbers
of order $r$  \begin{align} \mathsf H_{n}^{(r)}\colonequals\sum_{k=1}^\infty\left[ \frac{1}{k^r} -\frac{1}{(k+n)^{r}}\right],\quad n\in\mathbb C\smallsetminus\mathbb Z_{<0},\end{align}which become   $ \mathsf H_n^{(r)}=\sum_{k=1}^n\frac1{k^r}$ for $ n\in\mathbb Z_{>0}$. Aside from their connections to the hyper-Mahler measures  (see \S\S\ref{subsec:seriesW3}--\ref{subsec:seriesW4} for details) investigated in the current work,
these infinite series involving harmonic numbers  play significant r\^oles in the quantitative understanding of Feynman diagrams in quantum electrodynamics and quantum chromodynamics \cite{Broadhurst1996,KalmykovVeretin2000,DavydychevKalmykov2001,DavydychevKalmykov2004,MZVdatamine2010}.

Studies of ``linear Euler sums'' like \begin{align}
\sum_{n=1}^\infty\frac{\mathsf H_n^{(r)}}{n^{s}}=\zeta_{s,r}+\zeta_{r+s},\quad r\in\mathbb Z_{>0},s\in\mathbb Z_{>1}
\end{align}were initiated by Euler \cite{Euler1775}, in response to Goldbach's challenge  in 1742 \cite[\S1]{BZB2008}.  Borwein--Zucker--Boersma \cite{BZB2008} generalized the Euler--Goldbach double sums to $ [p,q](s,t)\colonequals\sum_{m>n>0}\frac{\chi_{p}(m)\chi_q(n)}{m^sn^t}$ for Dirichlet characters $p$ and $q$. This allows one to convert infinite sums like \cite[Table 1]{BZB2008}\begin{align}
\sum_{n=1}^\infty\frac{\mathsf H_n^{(r)}}{n^{s}},\quad
\sum_{n=1}^\infty\frac{(-1)^n\mathsf H_n^{(r)}}{n^{s}},\quad
\sum_{n=1}^\infty\frac{\mathsf H_n^{(r)}}{(2n-1)^{s}},\quad
\sum_{n=1}^\infty\frac{(-1)^n\mathsf H_n^{(r)}}{(2n-1)^{s}}
\end{align}into a single integral over polylogarithmic expressions \cite[Table 2]{BZB2008}. Although predating Brown's work  \cite{Brown2009a,Brown2009arXiv,Brown2009b}, the integral representations of Borwein--Zucker--Boersma \cite[Table 2]{BZB2008} belong to Brown's $ L(\mathfrak M_{0,n})$ class.
 A linear Euler sum $ \zeta_{- 3,1}$ in its polylogarithmic avatar   $ \zeta_{- 3,1}=2\Li_4\left( \frac{1}{2} \right)-\frac{\pi^{4}}{48}+\frac{7\zeta_3\log2}{4}-\frac{\pi^2\log^22}{12}+\frac{\log^42}{12}$  \cite[(5.5)]{BZB2008} featured prominently in the Laporta--Remiddi formula \cite[(5)]{LaportaRemiddi1996} for the 3-loop contribution to electron's magnetic moment.

The following ``non-linear Euler sums'' \begin{align}
\sum_{n=1}^\infty\frac{\left[\mathsf H_n^{(1)}\right]^2}{n^{2}},\quad
\sum_{n=1}^\infty\frac{(-1)^{n}\left[\mathsf H_n^{(1)}\right]^2}{n^{2}}, \quad
\sum_{n=1}^\infty\frac{\left[\mathsf H_n^{(2)}\right]^2}{n^{2}}\label{eq:nonlin_Euler_eg}
\end{align}had been considered by Borwein--Borwein \cite{BorweinBorwein1995zeta4}, De Doelder \cite{DeDoelder1991}, and Mez\H{o} \cite{Mezo2014} before   Panzer's \texttt{HyperInt} package \cite{Panzer2015} became available. In retrospect, these authors have effectively  converted their series of interest into  members of  $ L(\mathfrak M_{0,n})$ and have  referred to the classical treatise on polylogarithms \cite{Lewin1981} for integral evaluations, without regard to the algebraic geometry of the moduli spaces $ \mathfrak M_{0,n}$. Like their linear brethren, the non-linear Euler sums  found their ways into perturbative expansions in high energy physics. While considering on-shell charge renormalization, Broadhurst \cite[(42)]{Broadhurst1996} differentiated a certain hypergeometric expression and encountered non-linear Euler sums in the form of\begin{align}
\sum_{n=1}^\infty\frac{1}{(2n-1)^a}\prod_{i=1}^{k-1}\sum_{\ell_i=1}^{n-1}\frac{1}{(2\ell_i-1)^{b_i}}.
\label{eq:BroadhurstEuler}\end{align}More sophisticated analogs of such infinite series  showed up in the $\varepsilon$-expansions of Kalmykov--Veretin \cite[(4)]{KalmykovVeretin2000}, Davydychev--Kalmykov \cite[(1.11)]{DavydychevKalmykov2001} and Kalmykov--Ward--Yost \cite[Theorems A and B]{Kalmykov2007}.

For the prototypic non-linear Euler sums like those in \eqref{eq:nonlin_Euler_eg}, we have the following characterization of their period structure, which in turn, generalizes the Xu--Wang algorithm \cite{XuWang2020}.

\begin{theorem}[Non-linear Euler sums, HPLs, and AMZVs]\label{thm:nonlin_Euler}

\begin{enumerate}[leftmargin=*,  label=\emph{(\alph*)},ref=(\alph*),
widest=a, align=left]
\item For $ |z|<1$ and   $ r_1,\dots,r_M\in\mathbb Z_{>0}$, we have \begin{align}
\sum_{n=1}^\infty \left[\prod_{j=1}^M\mathsf H_{n}^{(r_j)}\right]z^n\in{}&\frac{1 }{1-z}\Span_{\mathbb Q}\left\{ L_{\bm A}(z)\left| w(\bm A)=\sum_{j=1}^M r_j\right.\right\}\equiv\frac{\mathfrak G^{(z)}_{\sum_{j=1}^M r_j}(1)}{1-z},\label{eq:nonlin_Euler_per}
\end{align}where the multi-dimensional polylogarithms $ L_{\bm A}(z)$ are defined by \eqref{eq:mLi_defn}, and the  $\mathbb Q $-vector space $ \mathfrak G^{(z)}_k(1)$ is given in \eqref{eq:Gkz(N)_Qvec}.
 By extension,  we have  (alternating) MZV representations (cf.\  \cite[Theorems 2.2 and 4.2]{XuWang2020})\begin{align}
\sum_{n=1}^\infty \left[\prod_{j=1}^M\mathsf H_{n}^{(r_j)}\right]\frac{1}{n^{s+1}}\in{}&\mathfrak Z_{k+1}(1),\label{eq:nonlin_Euler_Zk(1)}\\
\sum_{n=1}^\infty \left[\prod_{j=1}^M\mathsf H_{n}^{(r_j)}\right]\frac{(-1)^{n}}{n^{s}}\in{}&\mathfrak Z_{k}(2),\label{eq:nonlin_Euler_Zk(2)}\end{align}where $ s\in\mathbb Z_{>0}$ and $ k=s+\sum_{j=1}^M r_j$.\item Define\begin{align}
\smash[t]{\overline{\mathsf H}}_n^{(r)}\colonequals\sum_{k=1}^n\frac{(-1)^{k-1}}{k^r},\quad n\in\mathbb Z_{>0}.
\end{align}For $ |z|<1$,  $ r_1,\dots,r_M\in\mathbb Z_{>0}$, and  $ \overline{r}_1,\dots,\overline{r}_{\overline{M}}\in\mathbb Z_{>0}$, we have
\begin{align}
\sum_{n=1}^\infty \left[\prod_{j=1}^M\mathsf H_{n}^{(r_j)}\right]\left[\prod_{\smash[b]{\overline{j}}=1}^{\overline M}\smash[t]{\overline{\mathsf H}}_{n}^{(\overline{r}_{\overline{j}})}\right]z^n\in{}&\frac{\mathfrak G^{(z)}_{\sum_{j=1}^M r_j+\sum_{\overline{j}=1}^{\overline{M}} \overline{r}_{\overline{j}}}(2)}{1-z},\label{eq:nonlin_Euler_per_alt}
\end{align}where $ \mathfrak G^{(z)}_k(2)$ [defined in \eqref{eq:Gkz(N)_Qvec}] is  a subspace of  $ \mathfrak H^{(z)}_k(2)$  [defined in \eqref{eq:Hkz(N)_Qvec}], the latter being the  $\mathbb Q $-vector space spanned by all the Remiddi--Vermaseren harmonic polylogarithms (HPLs) \cite{RemiddiVermaseren2000} of weight $k$.  By extension,  we have   AMZV representations (cf.\  \cite[Theorem 4.2]{XuWang2020})\begin{align}\sum_{n=1}^\infty \left[\prod_{j=1}^M\mathsf H_{n}^{(r_j)}\right]\left[\prod_{\smash[b]{\overline{j}}=1}^{\overline M}\smash[t]{\overline{\mathsf H}}_{n}^{(\overline{r}_{\overline{j}})}\right]\frac{1}{n^{s+1}}\in{}&\mathfrak Z_{k+1}(2),\label{eq:nonlin_Euler_Zk(2)_alt1}\\\sum_{n=1}^\infty \left[\prod_{j=1}^M\mathsf H_{n}^{(r_j)}\right]\left[\prod_{\smash[b]{\overline{j}}=1}^{\overline M}\smash[t]{\overline{\mathsf H}}_{n}^{(\overline{r}_{\overline{j}})}\right]\frac{(-1)^{n}}{n^{s}}\in{}&\mathfrak Z_{k}(2),\label{eq:nonlin_Euler_Zk(2)_alt2}
\end{align} where $ s\in\mathbb Z_{>0}$ and $ k=s+\sum_{j=1}^M r_j+\sum_{\overline{j}=1}^{\overline{M}} \overline{r}_{\overline{j}}$. \end{enumerate} \end{theorem}\begin{proof}\begin{enumerate}[leftmargin=*,  label={(\alph*)},ref=(\alph*),
widest=a, align=left]
\item First, we note that \cite[\S5.1]{Mezo2014}\begin{align}
\sum_{n=1}^\infty \mathsf H_{n}^{(r)}z^n=\frac{\Li_{r}(z)}{1-z}
\end{align}fits into \eqref{eq:nonlin_Euler_per} when $M=1$. Suppose that \eqref{eq:nonlin_Euler_per} holds for positive integers $M$ up to a certain $M_{0}\in\mathbb Z_{>1}$. In the next paragraph, we show that the same will apply to the $M_{0}+1$ case.

According to our induction hypothesis, we have the following relation for  $ 1\leq M\leq M_0$:\begin{align}\begin{split}
\sum_{n=1}^\infty \left[\prod_{j=1}^M\mathsf H_{n}^{(r_j)}\right]\frac{z^n}{n}\in{}&\Span_{\mathbb Q}\left\{ \int_0^z\frac{L_{\bm A}(x)\D x}{x(1-x)}\left| w(\bm A)=\sum_{j=1}^M r_j\right.\right\}\\\subseteq{}&\Span_{\mathbb Q}\left\{  L_{\bm A}(z)\left| w(\bm A)=1+\sum_{j=1}^M r_j\right.\right\},\end{split}\label{eq:LA_promo}
\end{align} where the ``$ \subseteq$'' step involves the standard recursion for GPL [see \eqref{eq:GPL_rec}] and the conversions between MPL\ and GPL [see \eqref{eq:MPL_GPL}--\eqref{eq:GPL_MPL}]. Since \begin{align}
\int_0^z\frac{L_{a_1,a_{2},\dots,a_n}(x)}{x}\D x=L_{a_1+1,a_{2},\dots,a_n}(z),\label{eq:LA_promo1}
\end{align}it also follows that \begin{align}
\sum_{n=1}^\infty \left[\prod_{j=1}^M\mathsf H_{n}^{(r_j)}\right]\frac{z^n}{n^s}\in\Span_{\mathbb Q}\left\{  L_{\bm A}(z)\left| w(\bm A)=s+\sum_{j=1}^M r_j\right.\right\}\label{eq:Lx_int}
\end{align}for $ s\in\mathbb Z_{>0}$ and   $ 1\leq M\leq M_0$. As we note that $ \mathsf H_0^{(r)}=0$  and $ \mathsf H_n^{(r)}-\mathsf H_{n-\smash1}^{(r)}=\frac{1}{n^r}$, we may deduce the following result for $ k=\sum_{j=1}^{M_0+1} r_j$:{\small\begin{align}\begin{split}&
(1-z)\sum_{n=1}^\infty \left[\prod_{j=1}^{M_{0}+1}\mathsf H_{n}^{(r_j)}\right]z^n=\sum_{n=1}^\infty \left[\prod_{j=1}^{M_{0}+1}\mathsf H_{n}^{(r_j)}-\prod_{j=1}^{M_{0}+1}\mathsf H_{n-\smash[t]{1}}^{(r_j)}\right]z^n
\\\in{}&\mathbb Q\Li_{k}(z)+\Span_{\mathbb Q}\left\{ \left.\sum_{n=1}^\infty \left[\prod_{j=1}^{M'}\mathsf H_{n}^{(r_j')}\right]\frac{z^n}{n^s}\right|\begin{smallmatrix} k=s+\sum_{j=1}^{M '}r_j\\s\in\mathbb Z_{>0}\\1\leq M'\leq M_0\end{smallmatrix}\right\}\\\subseteq{}&\Span_{\mathbb Q}\left\{  L_{\bm A}(z)\left| w(\bm A)=k\right.\right\},\end{split}\label{eq:Mezo_ind}\end{align}}where the ``$ \subseteq$'' step  follows from induction.

Therefore, the relation in  \eqref{eq:nonlin_Euler_per} holds for all positive integers $M$.
As a consequence, we may also extend
 \eqref{eq:Lx_int} to all $ M\in\mathbb Z_{>0}$.

  As we closely examine the derivation of \eqref{eq:Lx_int}, we  see that on its right-hand side, the first component of the vector $ \bm A$ is at least $ s$. This implies that the left-hand side of \eqref{eq:Lx_int} converges at the point  $ z=1$ (resp.\ $z=-1$) when $ s\in\mathbb Z_{>1}$ (resp.\ $ s\in\mathbb Z_{>0}$). Specializing \eqref{eq:Lx_int}  to  the $ z=1$ (resp.\ $z=-1$) scenario, we arrive at \eqref{eq:nonlin_Euler_Zk(1)} [resp.\ \eqref{eq:nonlin_Euler_Zk(2)}].   \item Direct computations reveal that \begin{align}
\sum_{n=1}^\infty\smash[t]{\overline{\mathsf H}}_n^{(\overline{r})}z^{n}=-\frac{\Li_{\overline{r}}(-z)}{1-z}
\end{align}holds for $ |z|<1$, where \begin{align}
-\Li_{\overline{r}}(-z)=G\left( \smash[b]{\underset{\overline{r}-1}{\underbrace{0,\dots,0}}},1;-z \right)=\vphantom{\underset{r-1}{\underbrace{0,\dots,0}}}G\left( \smash[b]{\underset{\overline{r}-1}{\underbrace{0,\dots,0}}},-1;z \right)\in\mathfrak G^{(z)}_{\overline{r}}(2).
\end{align}Here, in the penultimate step, we have resorted to the scaling property  \cite[(2.3)]{Frellesvig2016} for GPLs: $ G(a_1,a_{2},\dots,a_k;-z)=G(-a_1,-a_{2},\dots,-a_k;z)$ if $ a_k\neq0$. This scaling property also entails the relation $\mathfrak G^{(z)}_k(2)=\mathfrak G^{(-z)}_k(2)$ for all $ k\in\mathbb Z_{>0}$.

 After this, repeated invocations of
the GPL recursion in \eqref{eq:GPL_rec} take us to\begin{align}
\sum_{n=1}^\infty\smash[t]{\overline{\mathsf H}}_n^{(\overline r)}\frac{z^{n}}{n^{s}}\in\mathfrak  G^{(z)}_{\overline r+s}(2),
\end{align}so it follows that \begin{align}
(1-z)\sum_{n=1}^\infty\mathsf H_n^{(r)}\smash[t]{\overline{\mathsf H}}_n^{(\overline{r})}z^{n}=\sum_{n=1}^\infty\smash[t]{\overline{\mathsf H}}_n^{(\overline{r})}\frac{z^{n}}{n^{r}}-\sum_{n=1}^\infty\mathsf H_n^{(r)}\frac{(-z)^{n}}{n^{\overline r}}\in\mathfrak G^{(z)}_{r+\overline r}(2),
\end{align}which confirms \eqref{eq:nonlin_Euler_per_alt} for $ M=\overline M=1$.

Subsequently, we can prove \eqref{eq:nonlin_Euler_per_alt} by inductions on $M$ and $\overline M$,  as we did for part (a).

Like what we have encountered before, the relation in  \eqref{eq:nonlin_Euler_per_alt} lifts to\begin{align}
\sum_{n=1}^\infty \left[\prod_{j=1}^M\mathsf H_{n}^{(r_j)}\right]\left[\prod_{\smash[b]{\overline{j}}=1}^{\overline M}\smash[t]{\overline{\mathsf H}}_{n}^{(\overline{r}_{\overline{j}})}\right]\frac{z^n}{n^{s}}\in{}&\mathfrak{G}^{(z)}_w(2)\label{eq:nonlin_Euler_per_alt_lift}
\end{align}for $ s\in\mathbb Z_{>0}$ and $ w=s+\sum_{j=1}^M r_j+\sum_{\overline{j}=1}^{\overline{M}} \overline{r}_{\overline{j}}$.  Here,  if $ s=1$, then the left-hand side of  \eqref{eq:nonlin_Euler_per_alt_lift}  is a  $ \mathbb Q$-linear combination of functions in the form of \begin{align}
\int_0^z\frac{G(\alpha_1,\dots,\alpha_k;x)}{x(1-x)}\D x=G(0,\alpha_1,\dots,\alpha_ k;z)-G(1,\alpha_1,\dots,\alpha_k;z)
\end{align}where $ \alpha_1,\dots,\alpha_k\in\{-1,0,1\}$, $ k=\sum_{j=1}^M r_j+\sum_{\overline{j}=1}^{\overline{M}} \overline{r}_{\overline{j}}$, so that it converges when $ z=-1$; if  $ s\in\mathbb Z_{>1}$, then the left-hand side of  \eqref{eq:nonlin_Euler_per_alt_lift}  is a  $ \mathbb Q$-linear combination of functions in the form of  \begin{align}
\int_0^z\frac{G(\alpha_1,\dots,\alpha_k;x)}{x}\D x=G(0,\alpha_1,\dots,\alpha_k;z)
\end{align}where $ \alpha_1,\dots,\alpha_k\in\{-1,0,1\}$, $ k=s-1+\sum_{j=1}^M r_j+\sum_{\overline{j}=1}^{\overline{M}} \overline{r}_{\overline{j}}$, so that it converges when $ z=\pm1$. Therefore, special cases of   \eqref{eq:nonlin_Euler_per_alt_lift}   lead us to \eqref{eq:nonlin_Euler_Zk(2)_alt1} and \eqref{eq:nonlin_Euler_Zk(2)_alt2}.    \qedhere\end{enumerate}\end{proof}\begin{remark}In the statements of the theorem above and its corollaries in  two subsections
to follow, we only pinpoint the $ \mathbb Q$-vector spaces in which the infinite sums reside. The accompanying proofs contain effective algorithms to evaluate each individual  series explicitly.

For example, in Table \ref{tab:G_Mezo}, we compute a few generating functions by the procedures in the proof of the theorem above, and convert some of them to classical polylogarithms via the  Frellesvig--Tommasini--Wever method \cite{Frellesvig2016}. \begin{table}[t]\caption{Some generating functions of Mez\H{o} type \label{tab:G_Mezo}}
\begin{scriptsize}\begin{align*}\begin{array}{c|r@{}c@{}l}\hline\hline r&\multicolumn{3}{c}{ \mathsf G^{(r),m}(z)\colonequals \displaystyle\sum_{n=1}^\infty\left[\mathsf H_n^{(r)}\right]^mz^n,0<|z|<1\vphantom{\frac{\int}{\int}}}\\\hline1&\mathsf G^{(1),1}(z)&{}={}&\displaystyle\frac{\Li_1(z)}{1-z}\vphantom{\dfrac{\frac\int1}{\int}}\\[8pt]&\mathsf G^{(1),2}(z)&{}={}&\displaystyle\frac{\Li_2(z)+2L_{1,1}(z)}{1-z}=\frac{[\Li_1(z)]^2+\Li_2(z)}{1-z} \\[8pt]&\mathsf G^{(1),3}(z)&{}={}&\displaystyle\frac{\Li_3(z)+3L_{2,1}(z)+3L_{1,2}(z)+6L_{1,1,1}(z)}{1-z}\\&&{}={}&\displaystyle\frac{\frac{\pi^{2}}{2}\Li_1(z)+\left[ \Li_1 (z)\right]^3+\frac{3}{2}\left[ \Li_1 (z)\right]^2\log z+3\Li_3(1-z)+\Li_3(z)-3\zeta_3}{1-z}\\[8pt]&\mathsf G^{(1),4}(z)&{}={}&\displaystyle\frac{\Li_4(z)+4L_{3,1}(z)+6L_{2,2}(z)+4L_{1,3}(z)}{1-z}\\&&&{}+\dfrac{12L_{2,1,1}(z)+12L_{1,2,1}(z)+12L_{1,1,2}(z)+24L_{1,1,1,1}(z)}{1-z}\\&&{}={}&\displaystyle\frac{-4\zeta_3\Li_1(z)+\frac{5\pi ^2}{3}  \left[ \Li_1 (z)\right]^2-\frac{2}{3} \left[ \Li_1 (z)\right]^3[2\log(-z)-3\log z]+\frac{1}{3} \left[ \Li_1 (z)\right]^4+\frac{2 \pi ^4}{9}}{1-z}\\[5pt]&&&{} \displaystyle+\frac{[\Li_2(z)]{}^2-4 \Li_1(z) \Li_3(z)-8 \Li_4\left(\frac{1}{1-z}\right)-8 \Li_4\left(\frac{z}{z-1}\right)-12 \Li_4(1-z)-7 \Li_4(z)}{1-z}\\[8pt]\hline2&\mathsf G^{(2),1}(z)&{}={}&\displaystyle\frac{\Li_2(z)}{1-z}\vphantom{\dfrac{\frac\int1}{\int}}\\[8pt]&\mathsf G^{(2),2}(z)&{}={}&\displaystyle\frac{\Li_4(z)+2L_{2,2}(z)}{1-z}\\&&{}={}&\dfrac{4\zeta_3\Li_1(z)+\frac{\pi ^2}{3} [ \Li_1(z)]{}^2-\frac{2}{3} [\Li_1(z)]^3 \log(-z)-\frac{1}{3} \left[ \Li_1 (z)\right]^4+\frac{2 \pi ^4}{45}}{1-z}\vphantom{\frac{\int}{\int}}\\&&&{}\displaystyle+\frac{[\Li_2(z)]{}^2-4 \Li_1(z) \Li_3(z)-4 \Li_4\left(\frac{1}{1-z}\right)-4 \Li_4\left(\frac{z}{z-1}\right)-3 \Li_4(z)}{1-z}\\[8pt]\hline\hline\end{array}\end{align*}

\begin{align*}\begin{array}{c|l}\hline\hline (r,m,\ell)&\multicolumn{1}{c}{ \mathsf G^{(r),m}_\ell(z)\colonequals \displaystyle\sum_{n=1}^\infty\left[\mathsf H_n^{(r)}\right]^m\frac{z^n}{n^{\ell}},|z|<1\vphantom{\frac{\int}{\int}}}\\\hline(r,1,\ell)&\Li_{\ell+r}(z)+L_{\ell,r}(z),\quad r,\ell\in\mathbb Z_{>0}\vphantom{\dfrac11}\\[3pt](1,2,1)&\Li_3(z)+2L_{2,1}(z)+L_{1,2}(z)+2L_{1,1,1}(z)\\[3pt](1,2,2)&\Li_4(z)+2L_{3,1}(z)+L_{2,2}(z)+2L_{2,1,1}(z)\\[3pt](1,3,1)&\Li_4(z)+3L_{3,1}(z)+3L_{2,2}(z)+L_{1,3}(z)\\&{}+6L_{2,1,1}(z)+3L_{1,2,1}(z)+3L_{1,1,2}(z)+6L_{1,1,1,1}(z)\\[3pt]\hline\hline\end{array}\end{align*}
\end{scriptsize}\end{table}Here, the generating functions $ \mathsf G^{(1),2}(x)$ and $  \mathsf G^{(1),3}(x)$ have been reported by Mez\H{o} \cite[Theorem 9]{Mezo2014} in the form of  classical polylogarithms. Our induction step  \eqref{eq:Mezo_ind} is inspired by Mez\H{o}'s derivation of  $ \mathsf G^{(1),2}(x)$ \cite[(24)]{Mezo2014}.

For $ r,\ell\in\mathbb Z_{>0}$ and $ |z|<1$, the formula $ \sum_{n=1}^\infty\mathsf H_n^{(r)}\frac{z^n}{n^{\ell}}=\Li_{\ell+r}(z)+\Li_{\ell,r}(z,1)\in\mathfrak G_{\ell+r}^{(z)}(1)$ in Table \ref{tab:G_Mezo} builds inductively on  GPL recursions, as in our proof of
 \eqref{eq:Lx_int}. [As pointed out by an anonymous referee, one can  confirm the same by writing $ \mathsf H_n^{(r)}=\frac{1}{n^r}+\sum_{0<m<n}\frac{1}{m^r}$ and appealing to the definition of MPLs in \eqref{eq:Mpl_defn}.] Likewise, for    $ r$, $\ell$, and $z$ in the same ranges, one can prove  $ \sum_{n=1}^\infty\overline{\mathsf H}_n^{(r)}\frac{z^n}{n^{\ell}}=-\Li_{\ell+r}(-z)-\Li_{\ell,r}(z,-1)\in\mathfrak G_{\ell+r}^{(z)}(2)$  recursively [or through an application of $\overline{ \mathsf H}_n^{(r)}=\frac{(-1)^{n-1}}{n^r}+\sum_{0<m<n}\frac{(-1)^{m-1}}{m^r}$ to \eqref{eq:Mpl_defn}].
       \eor\end{remark}
\begin{remark}From the generating functions given by the theorem above, one can construct moment sequences that express harmonic numbers. For instance, if we pick a generating function \begin{align}\mathsf G_s^{(r_1,\dots,r_M)}(z)\colonequals\sum_{n=1}^\infty \left[\prod_{j=1}^M\mathsf H_{n}^{(r_j)}\right]\frac{z^n}{n^{s}}\in\mathfrak G^{(z)}_{s+\sum_{j=1}^M r_j}(1)\label{eq:G1_gfun}\end{align}for certain positive integers $ r_1,\dots,r_M$  and  $ s$, then we can extract its Taylor coefficients by residue calculus:\begin{align}
\left[\prod_{j=1}^M\mathsf H_{n}^{(r_j)}\right]\frac{1}{n^{s}}={\counterint}_{\!\!\!\!|z|=1}\frac{\mathsf G_s^{(r_1,\dots,r_M)}(z)}{z^{n}}\frac{\D z}{2\pi iz}={\counterint}_{\!\!\!\!|w|=1}\frac{\mathsf G_s^{(r_1,\dots,r_M)}(1/w)}{(1/w)^{n}}\frac{\D w}{2\pi iw},
\end{align} where $ n\in\mathbb Z_{>0}$. In view of the jump discontinuities   \cite{Maitre2012,Panzer2015}  $ \mathsf G_s^{(r_1,\dots,r_M)}(1/w+i0^+)-\mathsf G_s^{(r_1,\dots,r_M)}(1/w-i0^+)\in\pi i\mathfrak g^{(w)}_{s-1+\sum_{j=1}^M r_j}(1)$ [cf.\ \eqref{eq:gk(N)_GH}] for $0<w<1$, we can reformulate the last contour integral in the complex $w $-plane as a moment relation\begin{align}
\left[\prod_{j=1}^M\mathsf H_{n}^{(r_j)}\right]\frac{1}{n^{s}}=\int_{0}^1 x^{n-1}\mathsf L_s^{(r_1,\dots,r_M)}(x)\D x,
\end{align} where $n\in\mathbb Z_{>0}$ and $ \mathsf L_s^{(r_1,\dots,r_M)}(x)\in\mathfrak g^{(x)}_{s-1+\sum_{j=1}^M r_j}(1)$.  Since the difference between the two sides of the equation above defines a bounded holomorphic function of $n$ for $ \R n>0$ that vanishes at all the positive integers, it is in fact identically vanishing for all  $ \R n>0$ \cite[Theorem 2.8.1]{AAR}. In particular, for the same  $ \mathsf L_s^{(r_1,\dots,r_M)}(x)\in\mathfrak g^{(x)}_{s-1+\sum_{j=1}^M r_j}(1)$ as given above, we have
\begin{align}\begin{split}&
\left[\prod_{j=1}^M\mathsf H_{n-\smash[t]{\frac{1}{2}}}^{(r_j)}\right]\frac{1}{(2n-1)^{s}}=\frac{1}{2^{s}}\int_{0}^1 x^{n-\frac{3}{2}}\mathsf L_s^{(r_1,\dots,r_M)}(x)\D x\\={}&\frac{1}{2^{s-1}}\int_{0}^1 X^{2n-2}\mathsf L_s^{(r_1,\dots,r_M)}(X^{2})\D X,\end{split}\label{eq:Hhalf_prod}
\end{align}a moment relation that will become useful later in Corollary \ref{cor:Sun}.   \eor\end{remark}
\subsection{Infinite series related to $ m_k(1+x_1+x_2)$\label{subsec:seriesW3}}
Following the footsteps of Broadhurst \cite[\S\S2 and 6]{Broadhurst1999}, Davydychev--Kalmykov \cite[\S2]{DavydychevKalmykov2004}, and Borwein--Borwein--Straub--Wan \cite[\S5]{BBSW2012Mahler} at an early stage of the present research, we   converted $ m_k(1+x_1+x_2)$ into infinite series  by differentiating a formula of Borwein--Straub--Wan--Zudilin \cite[Theorem 7.2]{BSWZ2012} (see also \cite[Theorem 4.2]{BorweinStraub2012Mahler}):\begin{align}
W_{3}(s)=\frac{3^{s+\frac{3}{2}}}{2 \pi }  \frac{\left[ {\Gamma \left(1+\frac{s}{2}\right)}{} \right]^2}{\Gamma \left(2+s\right)}{_3F_2}\left(\!\!\left.\begin{smallmatrix}1+\frac{s}{2},1+\frac{s}{2},1+\frac{s}{2}\\[2pt]1,\frac{3+s}{2}\end{smallmatrix}\right|\frac{1}{4}\right),\quad |s|<2.
\end{align}With  $ c_n\colonequals\frac{(n!)^{2}}{(2n+1)!}=\frac{1}{(n+1)\binom{2n+2}{n+1}}=\frac12\int_{0}^1 x^n(1-x)^n\D x$ and $ \mathfrak h_n^{(r)}\colonequals3\mathsf H_n^{(r)}-\mathsf H_{n+\smash{\frac{1}{2}}}^{(r)}+2 \delta _{1,r}\log \frac{3}{2}$, we found (cf.\ \cite[\S4]{Kalmykov2007})\begin{align}
\frac{\pi}{\sqrt{3}}\left.\dfrac{\D^k W_{3}(s)}{\D s^k}\right|_{s=0}\in\Span_{\mathbb Q}\left\{\left.\sum_{n=0}^\infty c_n\prod_{j=1}^M\mathfrak h_n^{(r_j)}\right|r_1,\dots,r_M\in\mathbb Z_{>0};\sum_{j=1}^Mr_j=k\right\},\label{eq:W3deriv}
\end{align}by slightly modifying the proof of \cite[Theorem 10]{BBSW2012Mahler}.

  In the terminology of Kalmykov--Ward--Yost \cite[\S2.1]{Kalmykov2007}, the infinite series appearing in \eqref{eq:W3deriv} are special cases of ``inverse binomial sums''. In the next two corollaries, we explore such infinite sums in broader context. As precursors to Goncharov--Deligne periods in $ \mathfrak Z_k(N)$, we will also need  new  $ \mathbb Q$-vector spaces\footnote{For   $ S\subset\mathbb C$, our notation $ \mathfrak g^S_{k,N}$ is related (but not identical) to Au's $ \mathsf {MZV}^S_k$ \cite[\S5.3]{Au2022a}.  The sets $ \mathfrak g_{k}[z_1,\dots,z_n;z](N)$ and $ \mathfrak g_{k,N}^{\{ z_1,\dots,z_n\}}$ are filtered by weight just as $ \mathfrak g_k^{(z)}(N)$, while  $ \mathfrak g_{k}[z_1,\dots,z_n;z](N)$  also behaves like \eqref{eq:cl_int_N} under GPL recursions with respect to the variable $z$. In the corollaries below, one will encounter  the  $ \mathbb Q$-vector spaces in \eqref{eq:gkNS_prep} and \eqref{eq:gkNS} as natural outcomes of  shuffle, recursion, or fibration of GPLs.} {\small\begin{align}
\mathfrak g_{k}[z_1,\dots,z_n;z](N)\colonequals{}&\Span_{\mathbb Q}\left\{(\pi i)^{k-\ell}G(\alpha_1,\dots,\alpha_j;z)Z_{\ell-j}\in\mathbb C\left|\begin{smallmatrix}\alpha_1,\dots,\alpha _j\in\{0\}\cup\{z_1,\dots,z_n\}\\Z_{\ell-j}\in\mathfrak Z_{\ell-j}(N)\\0\leq j\leq\ell\leq k\end{smallmatrix}\right.\right\},\label{eq:gkNS_prep}\\\mathfrak g_{k,N}^{\{ z_1,\dots,z_n\}}\colonequals{}&\Span_{\mathbb Q}\left\{\left.\prod_{j=1}^J G_{j}\right|\begin{smallmatrix}G_{j}\in \mathfrak g_{k_{j}}[z_1,\dots,z_n;z](N)\\z\in\{z_1,\dots,z_n\}\\\sum_{j=1}^Jk_{j}=k\end{smallmatrix}\right\},\label{eq:gkNS}
\end{align}}which extend $ \mathfrak g_k^{(z)}(N)=\mathfrak g_k\big[1,e^{2\pi i/N},\dots,e^{2\pi i(N-1)/N};z\big](N)$ [defined in\ \eqref{eq:gk(N)_GH}].\begin{corollary}[Inverse binomial sums via GPLs] \label{cor:inv_binom_sum}In the following, we set    $ r_1,\dots,r_M\in\mathbb Z_{>0}$ and  $ {r}'_1,\dots,{r}'_{{M'}}\in\mathbb Z_{>0}$.\begin{enumerate}[leftmargin=*,  label=\emph{(\alph*)},ref=(\alph*),
widest=a, align=left]
\item If\begin{align}
\begin{cases}|1-\xi|^{2}\leq 4|\xi|, \\
s\in\mathbb Z_{>0},  \\
\end{cases}\text{or }\;\begin{cases}|1-\xi|^{2}< 4|\xi|, \\
s\in\mathbb Z_{\geq0},  \\
\end{cases}
\end{align}  then (cf.\ \cite[(1.1), (5.1a), and   (5.1b)]{Kalmykov2007})\begin{align}
\sum_{n=1}^\infty \left[ \prod_{j=1}^M\mathsf H_{n-\smash[t]{1}}^{(r_j)} \right]\frac{\big(\frac{1+\xi}{1-\xi} \big)^{\delta_{0,s}}}{n^{s+1}\binom{2n}{n}}\left[ -\frac{(1-\xi)^{2}}{\xi} \right]^{n}\in\mathfrak g^{(\xi)}_{k+1}(2)\cap  \mathfrak g_{k+1,1}^{\big\{1,\frac{1}{1-\xi},\frac{\xi}{\xi-1}\big\}}\label{eq:KWY_inv_binom}
\end{align}where $ k=s+\sum_{j=1}^Mr_j$. More generally, for any moment  sequence\footnote{For example, the function $ f_1(x)=\log x\in \mathfrak H_{1}^{(x)}(1)$ produces a sequence $a_{n,1}=\frac{2}{n\binom{2n}{n}}\big[\mathsf H^{(1)}_{n-1}-\mathsf H^{(1)}_{2n-1}\big] $, while $f_2(x) =\log^2 x\in  \mathfrak H_{2}^{(x)}(1)$ brings us another sequence $ a_{n,2}=\frac{2}{n\binom{2n}{n}}\left\{\big[\mathsf H^{(1)}_{n-1}-\mathsf H^{(1)}_{2n-1}\big]^{2}+\mathsf H^{(2)}_{n-1}-\mathsf H^{(2)}_{2n-1}\right\}$.} $
a_{n,M'}=\int_0^1 x^{n-1}(1-x)^{n-1}f_{M'}(x)\D x, n\in\mathbb Z_{>0}
$ satisfying $ f_{M'}(x)\in\mathfrak H_{M'}^{(x)}(1)$, we have\begin{align}
\sum_{n=1}^\infty \left[ \prod_{j=1}^M\mathsf H_{n-\smash[t]{1}}^{(r_j)} \right]\frac{\big(\frac{1+\xi}{1-\xi} \big)^{\delta_{0,s}}}{n^{s}}\left[ -\frac{(1-\xi)^{2}}{\xi} \right]^{n}a_{n,M'}\in \mathfrak g^{(\xi)}_{k+1}(2)\cap  \mathfrak g_{k+1,1}^{\big\{1,\frac{1}{1-\xi},\frac{\xi}{\xi-1}\big\}}\label{eq:KWY_inv_binom_star}\tag{\ref{eq:KWY_inv_binom}$^*$}
\end{align} for  $ k=s+M'+\sum_{j=1}^Mr_j$, so long as the infinite series are convergent.
\item If\begin{align}
\begin{cases}|1+\tau^{2}|\leq 2|\tau|, \\
s\in\mathbb Z_{>0},  \\
\end{cases}\text{or }\;\begin{cases}|1+\tau^{2}|< 2|\tau|, \\
s\in\mathbb Z_{\geq0},  \\
\end{cases}
\end{align}   then we have \begin{align}\begin{split}&
\sum_{n=1}^\infty \left[\prod_{j=1}^M{{\mathsf H}}_{\smash[t]{}n-1}^{( {r}_{j})}\right]\left\{\prod_{\smash[b]{j'}=1}^{M'}\Big[{{\mathsf H}}_{\smash[t]{2}n-1}^{( {r}'_{j'})}+\smash[t]{\overline{\mathsf H}}_{\smash[t]{2}n-1}^{( {r}'_{j'})}\Big]\right\}\frac{\big(\frac{1-\tau^{2}}{1+\tau^{2}} \big)^{\delta_{0,s}}}{n^{s+1}\binom{2n}{n}}\left( \frac{1+\tau^{2}}{\tau} \right)^{2n}\\\in{}&\mathfrak g_{k+1}^{(\tau)}(4)\cap\mathfrak g_{k+1,2}^{\left\{ i,-i,\tau ,-\frac{1}{\tau},-\tau,\frac{1}{\tau}\right\}},\label{eq:KWY_inv_binom_alt}\end{split}
\end{align}where $ k=s+\sum_{j=1}^M r_j+\sum_{{j}'=1}^{{M'}} {r}'_{j'}$.
More generally, for any moment  sequence $
b_{n,M''}=\int_0^1 x^{n-1}(1-x)^{n-1}g_{M''}(x)\D x, n\in\mathbb Z_{>0}
$ satisfying $ g_{M''}(x)\in\mathfrak G_{M''}^{(x)}(1)$,  we have \begin{align}
\begin{split}{}&
\sum_{n=1}^\infty \left[\prod_{j=1}^M{{\mathsf H}}_{\smash[t]{}n-1}^{( {r}_{j})}\right]\left\{\prod_{\smash[b]{j'}=1}^{M'}\Big[{{\mathsf H}}_{\smash[t]{2}n-1}^{( {r}'_{j'})}+\smash[t]{\overline{\mathsf H}}_{\smash[t]{2}n-1}^{( {r}'_{j'})}\Big]\right\}\frac{\big(\frac{1-\tau^{2}}{1+\tau^{2}} \big)^{\delta_{0,s}}}{n^{s}}\left( \frac{1+\tau^{2}}{\tau} \right)^{2n}b_{n,M''}\\\in{}&\mathfrak g_{k+1}^{(\tau)}(4)\cap\mathfrak g_{k+1,2}^{\left\{ i,-i,\tau ,-\frac{1}{\tau},-\tau,\frac{1}{\tau}\right\}}\end{split}\label{eq:KWY_inv_binom_alt_star_star}\tag{\ref{eq:KWY_inv_binom_alt}$^{*}$}
\end{align}for  $ k=s+M''+\sum_{j=1}^M r_j+\sum_{{j}'=1}^{{M'}} {r}'_{j'}$, so long as the infinite series are convergent.
\end{enumerate}\end{corollary}\begin{proof}From the asymptotic behavior $ \frac{1}{n\binom{2n}n}=\frac{1}{4^{n}}\sqrt{\frac{\pi}{n}}\left[ 1+O\left( \frac{1}{n} \right) \right]$ as $n\to\infty $, we see that all the infinite sums of interest indeed converge for the prescribed ranges of $s$, $\xi$, and $\tau$. \begin{enumerate}[leftmargin=*,  label={(\alph*)},ref=(\alph*),
widest=a, align=left]\item First, we consider the $s=0$ case in \eqref{eq:KWY_inv_binom}. By \eqref{eq:nonlin_Euler_per}, we have {\small\begin{align}\begin{split}{}& \sum_{n=1}^\infty \left[ \prod_{j=1}^M\mathsf H_{n-\smash[t]{1}}^{(r_j)} \right]\frac{1}{n\binom{2n}{n}}\left[ -\frac{(1-\xi)^{2}}{\xi} \right]^{n}\\={}&-\frac{(1-\xi)^{2}}{2\xi}\sum_{n=1}^\infty \left[ \prod_{j=1}^M\mathsf H_{n-\smash[t]{1}}^{(r_j)} \right]\left[ -\frac{(1-\xi)^{2}}{\xi} \right]^{n-1}\int_0^1 x^{n-1}(1-x)^{n-1}\D x\\\in{}&-\frac{(1-\xi)^{2}}{2\xi}\Span_{\mathbb Q}\left\{ \int_0^1L_{\bm A}\left(-\frac{(1-\xi)^{2}}{\xi}x(1-x)\right)\frac{\D x}{1+\frac{(1-\xi)^{2}}{\xi}x(1-x)}\left| w(\bm A)=\sum_{j=1}^M r_j\right.\right\}.\end{split}
\end{align}}Like the proof of Theorem \ref{thm:Zk(6)}, we will compute the integrals over multi-dimensional polylogarithms using Brown's method. By direct applications of the GPL recursion \eqref{eq:GPL_rec} \big[with $z =-\frac{(1-\xi)^{2}}{\xi}x(1-x)$\big], we can show inductively that \begin{align}\begin{split}
&L_{\bm A}\left(-\frac{(1-\xi)^{2}}{\xi}x(1-x)\right)\\\in{}&\Span_{\mathbb Z}\left\{ G(\beta_{1},\dots, \beta_{w(\bm A)};x) \left| \begin{smallmatrix}\beta_{1},\dots, \beta_{w(\bm A)}\in \big\{0,1,\frac{1}{1-\xi},\frac\xi{\xi-1}\big\} \\ \beta_{w(\bm A)}\neq0\end{smallmatrix} \right.\right\}.\end{split}
\end{align} Furthermore,  the partial fraction expansion\begin{align}
\frac{-\frac{(1-\xi)^{2}}{2\xi}}{1+\frac{(1-\xi)^{2}}{\xi}x(1-x)}=\frac{1-\xi}{1+\xi}\left(\frac{1}{x-\frac{1}{1-\xi}}-\frac{1}{x-\frac\xi{\xi-1}}\right)
\end{align} tells us that \begin{align}\begin{split}&
-\frac{1+\xi}{1-\xi}\frac{(1-\xi)^{2}}{2\xi}\int_0^1L_{\bm A}\left(-\frac{(1-\xi)^{2}}{\xi}x(1-x)\right)\frac{\D x}{1+\frac{(1-\xi)^{2}}{\xi}x(1-x)}\\\in{}& \mathfrak \Span_{\mathbb Q}\left\{G(\beta_{1},\dots, \beta_{w(\bm A)+1};1)  \left|\begin{smallmatrix}\beta_1,\dots,\beta_{w(\bm A)+1}\in\left\{0,1,\frac{1}{1-\xi},\frac\xi{\xi-1}\right\} \\\beta_{1}\in\left\{\frac{1}{1-\xi},\frac\xi{\xi-1}\right\},\beta_{w(\bm A)+1}\neq0\\\end{smallmatrix}\right.\right\}\\\subseteq{}&\mathfrak g_{w(\bm A)+1}^{(\xi)}(2)\cap  \mathfrak g_{w(\bm A)+1,1}^{\big\{1,\frac{1}{1-\xi},\frac{\xi}{\xi-1}\big\}},\end{split}\label{eq:LA_xi1}
\end{align}where  the ``$\in$'' step comes from the GPL recursion in \eqref{eq:GPL_rec}, and the ``$\subseteq$'' step draws on another round of  fibration with respect to $\xi$ \big[which is deducible from repeated applications of the $z=1 $ and $z=\xi $ cases of the GPL recursion \eqref{eq:GPL_rec} to the GPL differential form \eqref{eq:GPL_diff_form} for  $ \D\log(\alpha_{j-1}-\alpha_{j}),\D\log(\alpha_{j+1}-\alpha_{j})\in\big\{\D\log\frac{1}{1-\xi},\D\log\frac{\xi}{\xi-1},\D\log\frac{1+\xi}{1-\xi},0\big\} $\big].

Then, we work with positive integers  $ s$. By the GPL recursion in \eqref{eq:GPL_rec}, we  have [cf.\ \eqref{eq:LA_promo}]\begin{align}\begin{split}&
\sum_{n=1}^\infty \left[ \prod_{j=1}^M\mathsf H_{n-\smash[t]{1}}^{(r_j)} \right]\frac{z^{n}}{n}\\={}&\sum_{n=1}^\infty \left[ \prod_{j=1}^M\mathsf H_{n-\smash[t]{1}}^{(r_j)} \right]\int_{0}^z x^{n-1}\D x\in \Span_{\mathbb Q}\left\{ \int_0^z\frac{L_{\bm A}(x)}{1-x}\D x\left| w(\bm A)=\sum_{j=1}^M r_j\right.\right\}\\\subseteq{}&\Span_{\mathbb Q}\left\{ L_{\bm A}(z)\left| w(\bm A)=1+\sum_{j=1}^M r_j\right.\right\}.\end{split}
\end{align}Subsequent inductions in the spirit of \eqref{eq:LA_promo1} bring us \begin{align}
\sum_{n=1}^\infty \left[ \prod_{j=1}^M\mathsf H_{n-\smash[t]{1}}^{(r_j)} \right]\frac{z^{n}}{n^s}\in{}&\Span_{\mathbb Q}\left\{ L_{\bm A}(z)\left| w(\bm A)=s+\sum_{j=1}^M r_j\right.\right\}
\end{align}for all $ s\in\mathbb Z_{>0}$. Accordingly, we arrive at
\begin{align}\begin{split}&
\sum_{n=1}^\infty \left[ \prod_{j=1}^M\mathsf H_{n-\smash[t]{1}}^{(r_j)} \right]\frac{1}{n^{s+1}\binom{2n}{n}}\left[ -\frac{(1-\xi)^{2}}{\xi} \right]^{n}\\={}&
\sum_{n=1}^\infty \left[ \prod_{j=1}^M\mathsf H_{n-\smash[t]{1}}^{(r_j)} \right]\frac{1}{n^{s}}\left[ -\frac{(1-\xi)^{2}}{\xi} \right]^{n}\int_0^1 x^{n}(1-x)^{n}\frac{\D x}{x(1-x)}
\\\in{}&\Span_{\mathbb Q}\left\{ \int_0^1L_{\bm A}\left(-\frac{(1-\xi)^{2}}{\xi}x(1-x)\right)\frac{\D x}{x(1-x)}
\left| w(\bm A)=s+\sum_{j=1}^M r_j\right.\right\},
\end{split}
\end{align}
where \begin{align}\begin{split}&
\int_0^1L_{\bm A}\left(-\frac{(1-\xi)^{2}}{\xi}x(1-x)\right)\frac{\D x}{x(1-x)}\\\in{}& \mathfrak \Span_{\mathbb Q}\left\{G(\beta_{1},\dots, \beta_{w(\bm A)+1};1)  \left|\begin{smallmatrix}\beta_1,\dots,\beta_{w(\bm A)+1}\in\left\{0,1,\frac{1}{1-\xi},\frac\xi{\xi-1}\right\} \\\beta_{1}\in\left\{0,1\right\},\beta_{w(\bm A)+1}\neq0\\\end{smallmatrix}\right.\right\}\\\subseteq{}& \mathfrak g_{w(\bm A)+1}^{(\xi)}(2)\cap  \mathfrak g_{w(\bm A)+1,1}^{\big\{1,\frac{1}{1-\xi},\frac{\xi}{\xi-1}\big\}}\end{split}
\label{eq:LA_xi2}\end{align}follows from recursions and fibrations of  GPLs (along with logarithmic regularizations \cite[\S2.3]{Panzer2015}  when $\beta_1=1$).

This completes the proof of \eqref{eq:KWY_inv_binom}. To generalize it to \eqref{eq:KWY_inv_binom_star}, simply note that \begin{align}
\begin{split}&
L_{\bm A}\left(-\frac{(1-\xi)^{2}}{\xi}x(1-x)\right)f_{M'}(x)\\\in{}& \mathfrak \Span_{\mathbb Q}\left\{ \vphantom{\Big\{}G(\beta_{1},\dots, \beta_{w(\bm A)};x)  \left|\beta_1,\dots,\beta_{w(\bm A)}\in\left\{0,1,\frac{1}{1-\xi},\frac\xi{\xi-1}\right\} \right. \right\}\end{split}
\end{align}follows from the condition $ f_{M'}(x)\in\mathfrak H^{(x)}_{M'}(1)$ and the GPL shuffle algebra.

\item Combining \eqref{eq:nonlin_Euler_per_alt} with the observation that \begin{align}
{{\mathsf H}}_{\vphantom{1}n-1}^{(r)}=2^{r-1}\left[{{\mathsf H}}_{{2}n-1}^{(r)}-\smash[t]{\overline{\mathsf H}}_{2n-1}^{(r)}\right],\quad r\in\mathbb Z_{>0},\label{eq:2n-1_H}
\end{align} we may put down \begin{align}\begin{split}&
\sum_{n=1}^\infty \left[\prod_{j=1}^M{{\mathsf H}}_{\smash[t]{}n-1}^{( {r}_{j})}\right]\left\{\prod_{\smash[b]{j'}=1}^{M'}\Big[{{\mathsf H}}_{\smash[t]{2}n-1}^{( {r}'_{j'})}+\smash[t]{\overline{\mathsf H}}_{\smash[t]{2}n-1}^{( {r}'_{j'})}\Big]\right\}z^{2n-1}\\\in{}&\Span_{\mathbb Q}\left\{ \left.\frac{g(z)}{1-z} -\frac{g(-z)}{1+z}\right|g(z)\in\mathfrak G^{(z)}_{\sum_{j=1}^M r_j+\sum_{j'=1}^{M'} r'_{j'}}(2)\right\}.\end{split}\label{eq:Gsymsum}
\end{align}When $s=0$, to evaluate the integral representation for the left-hand side of \eqref{eq:KWY_inv_binom_alt}, we need \begin{align}\begin{split}
{}&\frac{1-\tau^{2}}{1+\tau^{2}}\frac{1+\tau^2}{\tau}\int_0^1\frac{G\left(\alpha_1,\dots, \alpha_w;\pm\frac{1+\tau^2}{\tau}\sqrt{X(1-X)}\right)\D X}{\sqrt{X(1-X)}\left[1\mp\frac{1+\tau^2}{\tau}\sqrt{X(1-X)}\right]}\\={}&\frac{2(1-\tau^2)}{\tau}\int_{1/2}^{1}\frac{G\left(\alpha_1,\dots, \alpha_w;\pm\frac{1+\tau^2}{\tau}\sqrt{x(1-x)}\right)\D x}{\sqrt{x(1-x)}\left[1\mp\frac{1+\tau^2}{\tau}\sqrt{x(1-x)}\right]}\\={}&4\int_0^1G\left(\alpha_1,\dots,\alpha_w;\pm\frac{(1+\tau^2)t}{(1+t^{2})\tau}\right)\left( \frac{1}{\tau\mp t}-\frac{1}{\frac{1}{\tau}\mp t} \right)\D t,\end{split}
\end{align}where $ x=\frac{1}{1+t^2}$, $ \alpha_1,\dots, \alpha_w\in\{-1,0,1\},\alpha_w\neq0$ and $ w=\sum_{j=1}^M r_j+\sum_{j'=1}^{M'} r'_{j'}$. Again,  by  recursive applications of  \eqref{eq:GPL_rec} for $z =\pm\frac{(1+\tau^2)t}{(1+t^{2})\tau}$, we can show that \begin{align}\begin{split}
&G\left(\alpha_1,\dots,\alpha_w;\pm\frac{(1+\tau^2)t}{(1+t^{2})\tau}\right)\\\in{}&\Span_{\mathbb Z}\left\{ G(\beta_{1},\dots, \beta_{w};t) \left| \begin{smallmatrix}\beta_{1},\dots, \beta_{w}\in\left\{ 0,\tau ,i,-\frac{1}{\tau},-\tau,-i,\frac{1}{\tau}\right\} \\ \beta_{w}\neq0\end{smallmatrix} \right.\right\}.\end{split}
\end{align} Therefore, the left-hand side of \eqref{eq:KWY_inv_binom_alt} belongs to the $ \mathbb Q$-vector space{\small\begin{align}
\Span_{\mathbb Q}\left\{G(\beta_{1},\dots, \beta_{k+1};1)\left|\beta_1,\dots, \beta_{k+1}\in\left\{ 0,\tau ,i,-\frac{1}{\tau},-\tau,-i,\frac{1}{\tau}\right\}\right.\right\},\label{eq:horQvec}
\end{align}}when $s=0$.
 Fibrating the GPLs in \eqref{eq:horQvec} with respect to $\tau$, we realize that  \eqref{eq:horQvec}  is a subspace of $ \mathfrak g^{(\tau)}_{k+1}(4)$.

Then we handle  \eqref{eq:KWY_inv_binom_alt} for $ s\in\mathbb Z_{>0}$. By the GPL recursion in \eqref{eq:GPL_rec},  we may promote  \eqref{eq:Gsymsum} to\begin{align}\begin{split}&
\sum_{n=1}^\infty \left[ \prod_{j=1}^M\mathsf H_{n-\smash[t]{1}}^{(r_j)} \right]\left\{\prod_{\smash[b]{j'}=1}^{M'}\Big[{{\mathsf H}}_{\smash[t]{2}n-1}^{( {r}'_{j'})}+\smash[t]{\overline{\mathsf H}}_{\smash[t]{2}n-1}^{( {r}'_{j'})}\Big]\right\}\frac{z^{2n}}{n^{s}}\\\in{}&\mathfrak  \Span_{\mathbb Q}\left\{ g(z)+g(-z)\left|g(z)\in\mathfrak G^{(z)}_{s+\sum_{j=1}^M r_j+\sum_{j'=1}^{M'} r'_{j'}}(2)\right.\right\}\end{split}\label{eq:G_promo}
\end{align} for all $ s\in\mathbb Z_{>0}$. Accordingly, the integral representation of\begin{align}
\sum_{n=1}^\infty \left[ \prod_{j=1}^M\mathsf H_{n-\smash[t]{1}}^{(r_j)} \right]\left\{\prod_{\smash[b]{j'}=1}^{M'}\Big[{{\mathsf H}}_{\smash[t]{2}n-1}^{( {r}'_{j'})}+\smash[t]{\overline{\mathsf H}}_{\smash[t]{2}n-1}^{( {r}'_{j'})}\Big]\right\}\frac{1}{n^{s+1}\binom{2n}{n}}\left( \frac{1+\tau^{2}}{\tau} \right)^{2n}\label{eq:hor0}
\end{align} requires us to evaluate\begin{align}\begin{split}&
\int_0^1\frac{G\left(\widetilde{\alpha}_1,\dots, \widetilde{\alpha}_{w+s};\pm \frac{1+\tau^{2}}{\tau}\sqrt{X(1-X)}\right)}{X(1-X)}\D X\\={}&4\int_0^1G\left(\widetilde{\alpha}_1,\dots, \widetilde{\alpha}_{w+s};\pm\frac{(1+\tau^2)t}{(1+t^{2})\tau}\right)\frac{\D t}{t},\end{split}
\end{align} where $ \widetilde{\alpha}_1,\dots, \widetilde{\alpha}_{w+s}\in \{-1,0,1\},\alpha_w\neq0$ and $ w=\sum_{j=1}^M r_j+\sum_{j'=1}^{M'} r'_{j'}$.  Therefore, the left-hand side of \eqref{eq:KWY_inv_binom_alt}  is still in the $ \mathbb Q$-vector space $ \mathfrak g^{(\tau)}_{k+1}(4)$.

This completes the proof for a weaker version of    \eqref{eq:KWY_inv_binom_alt}. The same arguments extend to the $  \mathfrak g^{(\tau)}_{k+1}(4)$ part of  \eqref{eq:KWY_inv_binom_alt_star_star}, since   $ g_{M''}(x)\in\mathfrak G_{M''}^{(x)}(1)$ entails $ g_{M''}\big(\frac{1}{1+t^2}\big)+ g_{M''}\big(1-\frac{1}{1+t^2}\big)\in\mathfrak g_{M''}[i,-i;t](2)$.

To verify    \eqref{eq:KWY_inv_binom_alt} and    \eqref{eq:KWY_inv_binom_alt_star_star} in their entirety, we need contour deformations. Concretely speaking, if we have $f _k(z)\in\mathfrak G_k^{(z)}(2)$ and    $ g_{M''}(z)\in\mathfrak G_{M''}^{(z)}(1)$  for $ |z|<1$, then we can deform the contour of\begin{align}\begin{split}&
\int_0^\infty f_{k}\left(\pm\frac{(1+\tau^2)t}{(1+t^{2})\tau}\right)g_{M''}\left(\frac{1}{1+t^2}\right)h(t)\D t\\={}&\left( \int_{-i0^+}^{\infty-i0^+}-  \int_{i0^+}^{\infty+ i0^+}\right)f_{k}\left(\pm\frac{(1+\tau^2)t}{(1+t^{2})\tau}\right)g_{M''}\left(\frac{1}{1+t^2}\right)h(t)\frac{\log(-t)\D t}{2\pi i},\\{}&\text{where }h(t)\in\left\{ \frac{1}{\tau\mp t}-\frac{1}{\frac{1}{\tau}\mp t},\frac{1}{t} \right\}\end{split}
\end{align} to tight loops enveloping the branch cuts, so that  the net contribution to the integral results from the following jump discontinuities  \cite{Maitre2012,Panzer2015}\begin{align}
\left\{\begin{array}{@{}l}f_{k}(\pm (\varXi+i0^+))-f_{k}(\pm (\varXi-i0^+))\in\pi i\mathfrak g^{(\pm1/\varXi)}_{k-1}(2),\\[5pt]g_{M''}(\varXi+i0^+)-g_{M''} (\varXi-i0^+)\in\pi i\mathfrak g^{(1/\varXi)}_{M''-1}(1),\end{array} \right.
\end{align}for $\varXi>1 $ (with the understanding that $ \mathfrak g_{-1}^{(z)}(1)=\{0\}$). Here, in the complex $t$-plane, the end points of these branch cuts belong to the set $ \left\{ i,-i,\tau ,-\tfrac{1}{\tau},-\tau,\tfrac{1}{\tau}\right\}$. Recursive applications of   \eqref{eq:GPL_rec} with $z=i$ and $z=t$ lead us to \begin{align}
\mathfrak g^{(\pm1/\varXi)}_{k-1}(2)\subseteq\mathfrak g_{k-1}^{(t)}(4)\quad\text{and}\quad \mathfrak g^{(1+t^{2})}_{M''-1}(1)\subseteq\mathfrak g_{M''-1}^{(t)}(4)
\end{align} for $ \varXi=\frac{(1+\tau^2)t}{(1+t^{2})\tau}$, hence the full characterization stated in \eqref{eq:KWY_inv_binom_alt} and    \eqref{eq:KWY_inv_binom_alt_star_star}.
\qedhere\end{enumerate}\end{proof}
\begin{corollary}[Inverse binomial sums via Goncharov--Deligne periods]Let $ \bm r=( r_1,\dots,r_M)$ and  $\bm r'=( {r}'_1,\dots,{r}'_{{M'}})$ be vectors whose components are positive integers. Define their $1$-norms by $ \Vert \bm r\Vert_1\colonequals\sum_{j=1}^Mr_j$ and  $ \Vert \bm r'\Vert_1\colonequals\sum_{j=1}^{M'}r'_j$. We will examine the following inverse binomial sums \begin{align}
\mathsf s_{\bm r}(u,s)\colonequals {}&\sum_{n=1}^\infty \left[ \prod_{j=1}^M\mathsf H_{n-\smash[t]{1}}^{(r_j)} \right]\frac{u^{n}}{n^{s+1}\binom{2n}{n}},\\\mathsf S_{\bm r;\bm r'}(U,s)\colonequals{}&\sum_{n=1}^\infty \left[\prod_{j=1}^M\mathsf H_{n-\smash[t]{1}}^{(r_j)}\right]\left\{\prod_{\smash[b]{j'}=1}^{M'}\Big[{{\mathsf H}}_{\smash[t]{2}n-1}^{( {r}'_{j'})}+\smash[t]{\overline{\mathsf H}}_{\smash[t]{2}n-1}^{( {r}'_{j'})}\Big]\right\}{\frac{U^{n}}{n^{s+1}\binom{2n}{n}}},
\end{align} for certain special values of $ u=-\frac{(1-\xi)^{2}}{\xi}$ and $ U=\left( \frac{1+\tau^{2}}{\tau} \right)^{2}$.\begin{enumerate}[leftmargin=*,  label=\emph{(\alph*)},ref=(\alph*),
widest=a, align=left]
\item
For  $s\in\mathbb Z_{\geq0}$, $ k=s+\Vert \bm r\Vert_1$, we have the following particular cases of \eqref{eq:KWY_inv_binom},  where $ \xi\in\big\{\varrho\colonequals e^{\pi i/3},1-\rho\colonequals\frac{3-\sqrt{5}}{2}=\rho^2,i,\omega\colonequals e^{2\pi i/3}\big\}$:\footnote{In this work, the  Greek letters in the notations for roots of unity  $ \omega\colonequals e^{2\pi i/3}$, $ \varsigma\colonequals e^{2\pi i/5}$, $\varrho\colonequals e^{2\pi i/6}$, and $ \theta\colonequals e^{2\pi i/8}$ are  chosen for their visual  resemblance
to the corresponding Arabic numerals. The  notation $\rho\colonequals\frac{\sqrt{5}-1}{2}$ is standard in modern literature on polylogarithms \cite{Lewin1991coll,BorweinBroadhurstKamnitzer2001}. } {\allowdisplaybreaks\begin{align}
\big(i\sqrt{3} \big)^{\delta_{0,s}}\mathsf s_{\bm r}(1,s)\in{}& \mathfrak g^{(\varrho)}_{k+1}(2)\cap\mathfrak Z_{k+1}^{}(3),\label{eq:KWY0}
\\
\big(\sqrt{5} \big)^{\delta_{0,s}}\mathsf s_{\bm r}(-1,s)\in{}&  \mathfrak g^{(\rho^2)}_{k+1}(2)\cap\mathfrak Z_{k+1}^{}(10), \label{eq:genApery}\\
i^{\delta_{0,s}}\mathsf s_{\bm r}(2,s)\in {}&\mathfrak g^{(i)}_{k+1}(2)\cap\mathfrak Z_{k+1}^{}(4),\label{eq:KMY_i}\\
\big(i\sqrt{3} \big)^{\delta_{0,s}}\mathsf s_{\bm r}(3,s)\in {}&\mathfrak g^{(\omega)}_{k+1}(2)\cap\mathfrak Z_{k+1}^{}(6).\label{eq:KMY_omega}
\end{align}}For   $s\in\mathbb Z_{>0}$, $ k=s+\Vert \bm r\Vert_1$, we have  \begin{align}
\mathsf s_{\bm r}(4,s)=\sum_{n=1}^\infty \left[ \prod_{j=1}^M\mathsf H_{n-\smash[t]{1}}^{(r_j)} \right]\frac{2^{2n}}{n^{s+1}\binom{2n}{n}}\in{}&\mathfrak  Z_{k+1}^{}(2),\label{eq:KWY_-1}
\intertext{ along with the following quadratic analog of inverse binomial sums:}
\sum_{n=1}^\infty \left[ \prod_{j=1}^M\mathsf H_{n-\smash[t]{1}}^{(r_j)} \right]\frac{2^{4n}}{n^{s+2}\binom{2n}{n}^2}\in {}&\mathfrak Z_{k+2}(4).\label{eq:KWY_quad}
\end{align}
 \item

For  $s\in\mathbb Z_{\geq0}$, $ k=s+\Vert \bm r\Vert_1+\Vert \bm r'\Vert_1$, we have the following particular cases of \eqref{eq:KWY_inv_binom_alt},  where $ \tau\in\big\{\varrho\colonequals e^{\pi i/3},i\rho\colonequals i\frac{\sqrt{5}-1}{2},\theta\colonequals e^{\pi i/4},i/\varrho=e^{ \pi i/6}\big\}$: {\allowdisplaybreaks\begin{align}
\big(i\sqrt{3} \big)^{\delta_{0,s}}\mathsf S_{\bm r;\bm r'}(1,s)\in {}&\mathfrak g^{(\varrho)}_{k+1}(4)\cap\mathfrak Z_{k+1}^{}(12),\label{eq:KWY1}\\
\big(\sqrt{5} \big)^{\delta_{0,s}}\mathsf S_{\bm r;\bm r'}(-1,s)\in {}& \mathfrak g_{k+1}^{(i\rho)}(4)\cap\mathfrak Z_{k+1}(60),\label{eq:Zk60}
\\
i^{\delta_{0,s}}\mathsf S_{\bm r;\bm r'}(2,s)\in{}& \mathfrak g^{(\theta)}_{k+1}(4)\cap\mathfrak Z_{k+1}^{}(8),\\\big(i\sqrt{3} \big)^{\delta_{0,s}}\mathsf S_{\bm r;\bm r'}(3,s)\in{}& \mathfrak g^{(i/\varrho)}_{k+1}(4)\cap\mathfrak Z_{k+1}(6).\label{eq:Zk6i}\end{align} }For  $s\in\mathbb Z_{>0}$, $ k=s+\Vert \bm r\Vert_1+\Vert \bm r'\Vert_1$, we have \begin{align}
\mathsf S_{\bm r;\bm r'}(4,s)=\sum_{n=1}^\infty \left[ \prod_{j=1}^M\mathsf H_{n-\smash[t]{1}}^{(r_j)} \right]\left\{\prod_{\smash[b]{j'}=1}^{M'}\Big[{{\mathsf H}}_{\smash[t]{2}n-1}^{( {r}'_{j'})}+\smash[t]{\overline{\mathsf H}}_{\smash[t]{2}n-1}^{( {r}'_{j'})}\Big]\right\}\frac{2^{2n}}{n^{s+1}\binom{2n}{n}}\in{}&\mathfrak Z_{k+1}^{}(4),\label{eq:KWY_-1_alt}
\intertext{as well as the following quadratic analog of inverse binomial sums:}
\sum_{n=1}^\infty \left[ \prod_{j=1}^M\mathsf H_{n-\smash[t]{1}}^{(r_j)} \right]\left\{\prod_{\smash[b]{j'}=1}^{M'}\Big[{{\mathsf H}}_{\smash[t]{2}n-1}^{( {r}'_{j'})}+\smash[t]{\overline{\mathsf H}}_{\smash[t]{2}n-1}^{( {r}'_{j'})}\Big]\right\}\frac{2^{4n}}{n^{s+2}\binom{2n}{n}^2}\in{}& \mathfrak Z_{k+2}(4).\label{eq:KWY_quad'}
\end{align}\end{enumerate}

\end{corollary}
\begin{proof} \begin{enumerate}[leftmargin=*,  label={(\alph*)},ref=(\alph*),
widest=a, align=left]
\item
\begin{table}[t]\caption{Selected closed forms  for inverse binomial sums, where $ \omega\colonequals e^{2\pi i/3}$, $ \varrho\colonequals e^{\pi i/3}$, $\rho\colonequals\frac{\sqrt{5}-1}{2}$,  $ \varsigma\colonequals e^{2\pi i/5}$,  $ G\colonequals\I\Li_2(i)$,  $ \lambda\colonequals\log2$, and $ \varLambda\colonequals\log3$\label{tab:inv_binom}}\vspace{-1em}\begin{scriptsize}\begin{align*}\begin{array}{c|l}\hline\hline (r,m,\xi)&{  \displaystyle\sum_{n=1}^\infty\left[\mathsf H_{n-1}^{(r)}\right]^m\frac{1}{n\binom{2n}{n}}\left[ -\frac{(1-\xi)^{2}}{\xi} \right]^{n-1}\vphantom{\frac{\frac1\int}{\frac\int1}}=\frac{1}{2}\int_0^1\mathsf G^{(r),m}\left(-\frac{(1-\xi)^{2}}{\xi}x(1-x)\right)\D x\vphantom{\frac{\int}{\int}}}\\\hline(1,1,\varrho)&\frac{2\I\Li_{2}(\omega)}{\sqrt{3}}-\frac{\pi\varLambda}{3\sqrt{3}}\vphantom{\dfrac{\int}{1}}\\[5pt](1,2,\varrho)&-\frac{2^{3}\I L_{2,1}(\omega)}{\sqrt{3}}-\frac{2^{2}\varLambda\I \Li_{2}(\omega)}{\sqrt{3}}+\frac{\pi\varLambda^{2}}{3\sqrt{3}}+\frac{\pi^{3}}{2\cdot3^3\sqrt{3}}\\[5pt](1,3,\varrho)&-\frac{3^{2}\cdot13\I\Li_{4}(\omega)}{2^2 \sqrt{3}}+\frac{2^{4}\cdot3\I L_{2,1,1}(\omega)}{\sqrt{3}}+\frac{7^{2}\pi\zeta_{3}}{3^{2}\sqrt{3}}+\frac{2^3\cdot3\varLambda\I L_{2,1}(\omega)}{\sqrt{3}}\\&{}+\left( \frac{5\pi^{2}}{2\cdot3} +2\cdot3\varLambda^{2}\right)\frac{\I\Li_{2}(\omega)}{\sqrt{3}}-\left( \frac{\pi^{2}}{2\cdot3}+\varLambda^{2} \right)\frac{\pi\varLambda}{3\sqrt{3}}\\[5pt](1,4,\varrho)&-\frac{2^2\cdot3^2\cdot79\I L_{4,1}(\omega)}{13\sqrt{3}}-\frac{2\cdot3^{2}\cdot109\I L_{3,2}(\omega)}{13\sqrt{3}}-\frac{2^7\cdot3\I L_{2,1,1,1}(\omega)}{\sqrt{3}}+\frac{3^{2}\cdot13\varLambda\I \Li_4(\omega)}{\sqrt{3}}\\&{}+\frac{2^5\pi\R L_{3,1}(\omega)}{\sqrt{3}}-\frac{2^6\cdot3\varLambda\I L_{2,1,1}(\omega)}{\sqrt{3}}+\frac{2^3\zeta_3\I \Li_2(\omega)}{\sqrt{3}}-\frac{2^2\cdot7^2\pi\zeta_3\varLambda}{3^2\sqrt{3}}\\&{}-\frac{2^2(5\pi^2+2^2\cdot3^2\varLambda^2)\I L_{2,1}(\omega)}{3\sqrt{3}}-\frac{3\pi[\I \Li_2(\omega)]^2}{2\sqrt{3}}-\left( \frac{5\pi^2}{3} +2^{2}\varLambda^2\right)\frac{2\varLambda\I\Li_2(\omega)}{\sqrt{3}} \\&{}+\left(\frac{\pi^{2}}{3}+\varLambda^2 \right)\frac{\pi\varLambda^2}{3\sqrt{3}}-\frac{2267 \pi ^5}{2^2 \cdot3^4 \cdot5\cdot13\sqrt{3}}\\[5pt](2,1,\varrho)&\frac{\pi^3}{2\cdot3^4\sqrt{3}}\\[5pt](2,2,\varrho)&\frac{2^2\cdot3^3\I L_{4,1}(\omega)}{13\sqrt{3}}+\frac{2^2\cdot3^2\I L_{3,2}(\omega)}{13\sqrt{3}}+\frac{2^{2}\zeta_3\I \Li_2(\omega)}{\sqrt{3}}-\frac{3\pi[\I \Li_2(\omega)]^2}{2\sqrt{3}}-\frac{5^{2}\pi^5}{2^2\cdot3^6\cdot13\sqrt{3}}\\[5pt]\hline (1,1,\rho^2)&\frac{2^2}{\sqrt{5}}\R\Li_{1,1}\left(-\frac{1}{\varsigma^2}=e^{\pi i/5},\varsigma\right)+\frac{11 \pi ^2}{2\cdot5^{2} \sqrt{5}}\vphantom{\frac{\frac1\int}{1\frac\int1}}\\[5pt](1,2,\rho^2)&-\frac{2^{3}}{\sqrt{5}}\left[ 2^{2}\R\Li_{1,1,1}\left( \varsigma,1,-\frac{1}{\varsigma^2} \right)+\R\Li_{1,1,1}\left( -\varsigma,1,-\frac{1}{\varsigma^2} \right)\right]-\frac{5\cdot11}{3\sqrt{5}}\R\Li_3(\varsigma)-\frac{2^{6}\zeta_3}{5\sqrt{5}}\\&{}+\frac{2^{2}\pi\left[ \I\Li_2(\varsigma)+2\I\Li_2(\varsigma^2) \right]}{\sqrt{5}}+\frac{2^{4}}{\sqrt{5}}[\R\Li_1(\varsigma)]^{2}\R[\Li_1(\varsigma)-\Li_1(\varsigma^2)]\\&{}-\frac{2 \pi ^2}{3\cdot 5^2\sqrt{5}}\R[37\Li_1(\varsigma)-11\Li_1(\varsigma^2)]\\[5pt](2,1,\rho^2)&\frac{2^{2}}{3\sqrt{5}}\left\{ \R[\Li_1(\varsigma^2)-\Li_1(\varsigma)] \right\}^3\\[5pt]\hline (1,1,i)&G-\frac{\pi\lambda}{2^{2}}\vphantom{\frac{\int}1}\\[5pt](1,2,i)&-2^{2}\I L_{2,1}(i)-(2^{3}G-\pi\lambda)\frac{\lambda}{2^{2}}+\frac{\pi^3}{2^{5}}\\[5pt](1,3,i)&-3\cdot5\I \Li_{4}(i)+2^3\cdot3\I L_{2,1,1}(i)+\frac{7\cdot11\pi\zeta_3}{2^5}+2^2\cdot3\lambda\I L_{2,1}(i)+G(\pi^2+3\lambda^2)\\&{}-\left(\frac{3\pi^{2}}{2^{3}}+\lambda^{2} \right)\frac{\pi\lambda}{2^2}\\[5pt](1,4,i)&-\frac{2\cdot3^3\cdot17\I L_{4,1}(i)}{7} -\frac{2\cdot293\I L_{3,2}(i)}{7}-2^6\cdot3\I L_{2,1,1,1}(i)+2^2\cdot3\cdot5\lambda\I \Li_{4}(i)\\&{}+\frac{3\cdot5\pi\zeta_{-3,1}}{2}-2^5\cdot3\lambda\I L_{2,1,1}(i)-\left( \frac{5\cdot7\pi^{2}}{2^{2}}+2^{3}\cdot3\lambda^{2}\right)\I L_{2,1}(i)\\&{}+\frac{7\cdot11\zeta_3}{2^{4}}(3G-2\pi \lambda)-\frac{5\pi G^2}{2}-2^2 G\lambda(\pi^2+\lambda^2)+\frac{\pi\lambda^2}{2^2}\left( \frac{3\pi^{2}}{2^{2}}+\lambda^{2} \right)-\frac{13^{2}\pi^5}{2^{7}\cdot3\cdot7}\\[5pt](2,1,i)&\frac{\pi ^3}{2^5\cdot 3}\\[5pt](2,2,i)&2\cdot3 \I L_{4,1}(i)+2\I L_{3,2}(i)+\frac{5\cdot7\zeta_{3}G}{2^4}+\frac{\pi^2\I L_{2,1}(i)}{2^2}-\frac{\pi G^2}{2}-\frac{\pi^5}{2^7\cdot3^2}\\[5pt]\hline\hline\end{array}\end{align*}\end{scriptsize}\end{table}

With $ \xi=\varrho$ in  \eqref{eq:KWY_inv_binom}, we obtain $\big\{1,\tfrac{1}{1-\xi},\tfrac\xi{\xi-1}\big\} =\big\{ 1,\varrho,\tfrac{1}{\varrho} \big\} $. This hearkens back to  the scenario in  \eqref{eq:Zk3Au}, so  the left-hand side of  \eqref{eq:KWY0} is in $ \mathfrak Z_{k+1}^{}(3)$.

Setting $\xi=1-\rho=\frac{3-\sqrt{5}}{2}$ in  \eqref{eq:KWY_inv_binom}, we get $  \big\{1,\tfrac{1}{1-\xi},\tfrac\xi{\xi-1}\big\} =\big\{1,\tfrac{1+\sqrt{5}}{2},\tfrac{1-\sqrt{5}}{2}\big\}$. By a variation on the  arguments for \eqref{eq:Zk3Au}, we may use  GPL fibration in the variable  $z $  to prove\footnote{The rational functions in \eqref{eq:pentagon_fib} paraphrase \cite[Example 4.7]{Au2022a}. We also note that  the $M=1$ case of  \eqref{eq:genApery} has already been covered by \cite[Theorem  7.2]{Au2022a}. For the fibration here, the GPL recursion integrals  run from $-\varsigma$ (or $ -\varsigma^4$) to $z$, and the GPL differential forms involve terms like $ \D \log\big(1+\frac{(1+\varsigma )\varsigma (\varsigma+z)}{\varsigma^{4}+z}\big)=\D\log\frac{\varsigma^{2}+z}{\varsigma^{4}+z} $ and  $ \D \log\big(1+\frac{(1+\varsigma )\varsigma (\varsigma^{3}+z)}{\varsigma^{6}+z}\big)=\D\log\frac{\varsigma^{4}+z}{\varsigma^{6}+z} $.
} \begin{align}
\mathfrak g_{k+1,1}^{\big\{1,-\frac{(1+\varsigma )\varsigma (\varsigma+z)}{\varsigma^{4}+z},-\frac{(1+\varsigma )\varsigma (\varsigma^{3}+z)}{\varsigma^{6}+z}\big\}}\subseteq\mathfrak g_{k+1}\left[1,-\varsigma,-\varsigma^{2},-\varsigma^{3},-\varsigma^{4};z\right](10),\label{eq:pentagon_fib}
\end{align}  for  $\varsigma=e^{2\pi i/5}$, while checking that \begin{align}
-\frac{(1+\varsigma )\varsigma (\varsigma+z)}{\varsigma^{4}+z}=\frac{1+\sqrt{5}}{2},\quad -\frac{(1+\varsigma )\varsigma (\varsigma^{3}+z)}{\varsigma^{6}+z}=\frac{1-\sqrt{5}}{2}
\end{align}for $z=1$.  This identifies the left-hand side of \eqref{eq:genApery} as a CMZV of level 10  (cf.\ \cite[\S6]{BorweinBroadhurstKamnitzer2001}).

Specializing  \eqref{eq:KWY_inv_binom} to $ \xi\in\{i,\omega,-1\}$, while noting that   $ \pi i\in\mathfrak Z_{1}(N),N\in\mathbb Z_{\geq3}$ [cf.\ \eqref{eq:piIZ1}] and $ \mathfrak g^{(-1)}_{k+1}(2)\cap\mathbb R=\mathfrak Z_{k+1}(2)$, one can verify \eqref{eq:KMY_i}--\eqref{eq:KWY_-1}. Here, the relation \eqref{eq:KMY_i} [resp.\ \eqref{eq:KMY_omega}] is  consistent with \cite[Theorem 4.1]{Au2020}  \big(resp.\ \cite[Proposition 7.11]{Au2022a}   involving $ a_{n,1}=\frac{2}{n\binom{2n}{n}}\big[\mathsf H^{(1)}_{n-1}-\mathsf H^{(1)}_{2n-1}\big]$\big) for the $M=1 $ case.
The Euler--Ap\'ery-type series in \eqref{eq:KWY_-1} (cf.\ \cite[(1.1)]{WangXu2021}) were previously known to be AMZVs (namely CMZVs of level 2).

Combining the following  integral along a straight line contour (on which $ |1-z^2|\leq2|z|$ holds) \begin{align}
\int_1^i\left( \frac{1-z^2}{iz} \right)^{2n-1}\frac{\D z}{iz}=\int_{1}^i\left[ -\frac{(1-z^2)^2}{z^2} \right]^{n}\frac{\D z}{1-z^2}=-\frac{2^{4n-2}}{n\binom{2n}{n}}
\end{align}with the relation $ \mathfrak g_{k+1}^{(z^2)}(2)\subseteq\mathfrak g_{k+1}^{(z)}(4)$ (provable by  variable substitutions in the integrals  defining GPLs), we can confirm \eqref{eq:KWY_quad} by integrating   \eqref{eq:KWY_inv_binom}.

It is clear from the analysis so far that \eqref{eq:KWY_inv_binom} and \eqref{eq:KWY0}--\eqref{eq:KWY_quad} extend retroactively to the cases where $M=0$ (thus covering Zucker's sums \cite{Zucker1985} \begin{align}
\sum_{n=1}^\infty\frac{1}{n^{s+1}\binom{2n}n}
\end{align} and the like), and that the designated $ \mathbb Q$-vector spaces are not affected if we replace any single factor of   $ \mathsf H_{\vphantom{1}n-1}^{(r)}=\mathsf H_{\vphantom{1}n}^{(r)}-\frac{1}{n^r}$ by  $\mathsf H_{\vphantom{1}n}^{(r)}$.   \begin{table}[th!p]\caption{Selected closed forms for quadratic analogs of inverse binomial sums, where  $ G\colonequals\I\Li_2(i)$ and   $ \lambda\colonequals\log2$\label{tab:quad_inv_binom}}\begin{scriptsize}\begin{align*}\begin{array}{c|l@{}}\hline\hline (r,m,\ell)&{  \displaystyle\sum_{n=1}^\infty\left[\mathsf H_{n}^{(r)}\right]^m\frac{2^{4n}}{n^{\ell+2}\binom{2n}{n}^2}\vphantom{\frac{\frac1\int}{\frac\int\int}}=-2\int_1^i\left[\int_0^1\mathsf G^{(r),m}_\ell\left(-\frac{(1-z^{2})^{2}}{z^{2}}x(1-x)\right)\frac{\D x}{x(1-x)}\right]\frac{\D z}{1-z^2}\vphantom{\frac{\int}{\int}}}\\\hline(1,1,1)&2^4\cdot7\zeta_{-3,1}-2^{9}\R L_{2,1,1}(i)+2^5G(G+\pi\lambda)-2^{3}\cdot7\zeta_3\lambda-\frac{13\pi^4}{3^2\cdot5}\vphantom{\dfrac{\int}1}\\[5pt](1,1,2)&\frac{5\cdot211\zeta_5}{2^2}+2^5\cdot3^2\zeta_{-3,1,1}+2^9\R L_{2,2,1}(i)-2^5\cdot3^2\cdot5\pi\I \Li_{4}(i)+2^7\zeta_{-3,1}\lambda\\&{}+2^9\cdot3\pi \I L_{2,1,1}(i)+\zeta_3(11\cdot23\pi^2-2^4\cdot7\lambda^2)+2^9(G+\pi \lambda)\I L_{2,1}(i)+2^7G^2\lambda\\&{}+2^4\pi G(3\pi^2+2^2\lambda^2)+\frac{2\cdot7\pi^4\lambda}{3^2\cdot5}\\[5pt](1,1,3)&2^4\cdot3^3\cdot17\zeta_{-5,1}+2^{10}\cdot3\R L_{4,2}(i)+2^7\cdot3^{2} \zeta_{-3,1,1,1}+2^{12}\R L_{2,2,1,1}(i)+5\cdot211\zeta_5\lambda\\&{}+\frac{2^9\cdot3\cdot83\pi\I L_{4,1}(i) }{7}+\frac{2^{8}\cdot13\cdot23\pi\I L_{3,2}(i)}{7}+2^{7}\cdot3^{2}\zeta_{-3,1,1}\lambda+2^{11}\lambda\R L_{2,2,1}(i)\\&{}+2^{14}\pi\I L_{2,1,1,1}(i)-2^7\cdot3^2\cdot5\pi\lambda\I \Li_4(i)-2^4\zeta_{-3,1}(3^{2}\cdot5\pi^2-2^4\lambda^2)\\&{}+2^{11}(2G+3\pi \lambda)\I L_{2,1,1}(i)-4271\zeta_3^2+2^{2}\zeta_3\lambda\left( 11\cdot23\  \pi ^2-\frac{2^4 \cdot7\lambda^{2}}{3}  \right)\\{}&{}+2^{9}\pi\left( \frac{2^2G\lambda}{\pi}+\pi^2+2\lambda^2 \right)\I L_{2,1}(i)+2^6G^2(\pi^2+2^2\lambda^2)+\frac{2^{6}\pi G\lambda}{3}( 3^{2}\pi^{2}+2^{2}\lambda^{2})\\&{}+\frac{2^{2}\cdot7\pi^4\lambda^2}{3^2\cdot5}+\frac{118171\pi^6}{2^2\cdot3^4\cdot5\cdot7}\\[5pt](1,2,1)&\frac{7\cdot313\zeta_5}{2}+2^6\cdot17\zeta_{-3,1,1}+2^{10}\cdot3\R L_{2,2,1}(i)-2^5\cdot3^3\cdot5\pi\I \Li_{4}(i)+2^{9}\cdot7\pi \I L_{2,1,1}(i)\\{}&{}+2^7\zeta_{-3,1}\lambda+2^2\zeta_{3}\left( \frac{601\pi^{2}}{3} -2^{2}\cdot7\lambda^{2}\right)+2^8(2^3G+3\pi\lambda)\I L_{2,1}(i)+2^8 G^2\lambda\\&{}+2^2\pi G\left( \frac{5\cdot17\pi^{2}}{3} +2^{4}\lambda^{2}\right)+\frac{2\cdot7\pi^{4}\lambda}{3^2\cdot5}\\[5pt](1,2,2)&2^7\cdot13\cdot29\zeta_{-5,1}+2^{9}\cdot37\R L_{4,2}(i)+2^{8}\cdot17 \zeta_{-3,1,1,1}+2^{13}\cdot3\R L_{2,2,1,1}(i)\\&{}+2\cdot7\cdot313\zeta_{5}\lambda+\frac{2^8\cdot5^3\cdot11\pi\I L_{4,1}(i)}{7}+\frac{2^9\cdot11\cdot43\pi\I L_{3,2}(i)}{7}+2^{8}\cdot17\zeta_{-3,1,1}\lambda\\[2pt]&{}+2^{12}\cdot3\lambda\R L_{2,2,1}(i)+2^{12}\cdot13\pi\I L_{2,1,1,1}(i)-2^{7}\cdot3^3\cdot5\pi\lambda\I\Li_4(i)\\{}&{}-2^8\zeta_{-3,1}(2\cdot5\pi^2-\lambda^2)+2^{11}(2^2\cdot3G+7\pi \lambda)\I L_{2,1,1}(i)-2\cdot3^2\cdot1511\zeta_{3}^2\\{}&{}+\frac{2^4\zeta_3\lambda}{3}(601\pi^2-2^2\cdot7\lambda^2)-2^{12}[\I L_{2,1}(i)]^2+2^5\left(\frac{157\pi^{3}}{3}+2^{8}G\lambda+2^{4}\cdot3\pi\lambda^{2}\right)\I L_{2,1}(i)\\&{}+2^7G^2(3\pi^2+2^2\lambda^2)+\frac{2^4\pi G\lambda}{3}(5\cdot17\pi^2+2^4\lambda^2)+\frac{2^2\cdot7\pi^4\lambda^2}{3^{2}\cdot5}+\frac{11\cdot263\cdot409\pi^6}{2^3\cdot3^4\cdot5\cdot7}\\[5pt](1,3,1)&2^4\cdot29\cdot263\zeta_{-5,1}+2^9\cdot3\cdot31\R L_{4,2}(i)+2^7\cdot79\zeta_{-3,1,1,1}+2^{12}\cdot3\cdot5\R L_{2,2,1,1}(i)\\&{}+10321\zeta_{5}\lambda+\frac{2^{10}\cdot3^2\cdot103\pi\I L_{4,1}(i)}{7}+\frac{2^8\cdot2503\pi\I L_{3,2}(i)}{7}+2^{7}\cdot79\zeta_{-3,1,1}\lambda\\&{}+2^{11}\cdot3\cdot5\lambda\R L_{2,2,1}(i)+2^{13}\cdot17\pi\I L_{2,1,1,1}(i)-2^8\cdot3(5^2G+47\pi \lambda)\I\Li_4(i)\\{}&{}-2^{4}\zeta_{-3,1}(5\cdot83\pi^2-2^4\lambda^2)+2^{11}(2\cdot3\cdot7G+13\pi\lambda)\I L_{2,1,1}(i)-3\cdot5\cdot4561\zeta_3^2\\{}&{}+2^5\cdot109\pi\zeta_3G+\frac{2^{2}\zeta_3\lambda}{3}(5113\pi^2-2^4\cdot7\lambda^2)-2^{13}\cdot3[\I L_{2,1}(i)]^2\\&{}+2^6\pi\left( \frac{2^5\cdot3^2G\lambda}{\pi} +3\cdot23\pi^2+2^5\lambda^2\right)\I L_{2,1}(i)+2^{6}\cdot3 G^2(3^2\pi^{2}+2^2\lambda^2)\\&{}+\frac{2^4\pi G\lambda}{3}(3\cdot53\pi^2+2^4\lambda^2)+\frac{2^2\cdot7\pi^4\lambda^2}{3^2\cdot5}+\frac{13\cdot77213\pi^6}{2^3\cdot3^3\cdot5\cdot7}\\[5pt]\hline (2,1,1)&3\cdot7^2\zeta_5+2^7\zeta_{-3,1,1}-2^5\cdot13\pi \I\Li_4(i)+2^{7}\zeta_{-3,1}\lambda+2^9\pi\I L_{2,1,1}(i)\vphantom{\frac{\frac11}1}\\{}&{}+\frac{2\zeta_{3}}{3}(89\pi^2-2^3\cdot3\cdot7\lambda^2)+2^8\pi\lambda\I L_{2,1}(i)+2^2\pi G\left(\frac{13\pi^2}{3}+2^4\lambda ^2\right)+\frac{2\cdot7\pi^4\lambda}{3^2\cdot5}\\[5pt](2,1,2)&-2^6\cdot13\zeta_{-5,1}+2^9\zeta_{-3,1,1,1}+2^2\cdot3\cdot7^2\zeta_5\lambda+\frac{2^8\cdot5\cdot31\pi\I L_{4,1}(i)}{7}+\frac{2^9\cdot41\pi\I L_{3,2}(i)}{7}\\&{}+2^{9}\zeta_{-3,1,1}\lambda+2^{12}\pi\I L_{2,1,1,1}(i)-2^7\cdot13\pi\lambda\I \Li_4(i)-2^{5}\zeta_{-3,1}(5\pi^2-2^3\lambda^2)\\&{}+2^{11}\pi \lambda\I L_{2,1,1}(i)+2\cdot151\zeta_3^2+\frac{2^3\zeta_3\lambda}{3}(89\pi^2-2^3\cdot7\lambda^2)\\&{}+\frac{2^{5}\pi}{3}(13\pi^2+2^4\cdot3\lambda^2)\I L_{2,1}(i)+\frac{2^4\pi G\lambda}{3}(13\pi^2+2^4\lambda^2)+\frac{2^2\cdot7\pi^{4}\lambda^2}{3^2\cdot5}+\frac{197\pi^6}{2^2\cdot3^2\cdot7}\\[5pt]\hline\hline
\end{array}\end{align*}
\end{scriptsize}\end{table}\item
   With  $ \pi i\in\mathfrak Z_{1}(N),N\in\mathbb Z_{\geq3}$ [cf.\ \eqref{eq:piIZ1}] in mind, we can deduce    \eqref{eq:KWY1}--\eqref{eq:KWY_-1_alt} from     \eqref{eq:KWY_inv_binom_alt}, except that we need the next two  paragraphs to account for the CMZV portrayal in \eqref{eq:Zk60} and \eqref{eq:Zk6i}.

Setting $ R_n(z)\colonequals i\big( z e^{n\pi i/5}+\tfrac{1}{z e^{n\pi i/5}}\big)$, we consider\begin{align}\left.
\mathfrak g_{k+1,2}^{\left\{ i,-i,\tau ,-\frac{1}{\tau},-\tau,\frac{1}{\tau}\right\}}\right|_{\tau=i\rho}={}&\left.\mathfrak g_{k+1,2}^{\left\{i,-i,R_0(z),R_1(z),R_2(z),R_3(z)\right\}}\right|_{z=e^{\pi i/5}}.\label{eq:pentagon_prep}
\end{align}The right-hand side of the equation above is a   $ \mathbb Q$-linear  subspace of $ \mathfrak Z_{k+1}(60)$, as one can check by the fibration basis with respect to $z$.\footnote{To show that $ \left.\mathfrak g_{k+1,2}^{\left\{iz,-iz,zR_0(z),zR_1(z),zR_2(z),zR_3(z)\right\}}\right|_{z=e^{\pi i/5}}\subseteq \mathfrak Z_{k+1}(60) $, one may apply the GPL recursion \eqref{eq:GPL_rec} to the GPL differential form \eqref{eq:GPL_diff_form}, while noting that $ \D\log(i z+z R_1(z))=\D \log(z-e^{7\pi i/15})+\D \log(z+e^{2\pi i/15})$ and so forth.}

A na\"ive set inclusion  $  \mathfrak g^{(i/\varrho)}_{k+1}(4)\subseteq\mathfrak Z_{k+1}(12)$ [based on a weaker form of      \eqref{eq:KWY_inv_binom_alt}] is  sharpened by the following invocation of the scaling property  \cite[(2.3)]{Frellesvig2016} for GPLs:
\begin{align}
\left.
\mathfrak g_{k+1,2}^{\left\{ i,-i,\tau ,-\frac{1}{\tau},-\tau,\frac{1}{\tau}\right\}}\right|_{\tau=i/\varrho}\subseteq{}&\mathfrak g_{k+1,2}^{\left\{1,\varrho,\varrho^{2},-1,-\varrho,-\varrho^{2}\right\}}\subseteq\mathfrak Z_{k+1}(6),
\end{align}so the statement in  \eqref{eq:Zk6i} is true.\begin{table}[t]\caption{Continuation of Table \ref{tab:inv_binom}, with $ \omega\colonequals e^{2\pi i/3}$, $ \varrho\colonequals e^{\pi i/3}$, $ \theta=e^{\pi i/4}$, $ G\colonequals\I\Li_2(i)$,  $ \lambda\colonequals\log2$, and $ \varLambda\colonequals\log3$\label{tab:inv_binom2}}\begin{scriptsize}\begin{align*}\begin{array}{c|l}\hline\hline (r,m,\tau)&\begin{array}{@{}l}  \displaystyle\quad\sum_{n=1}^\infty\left[\mathsf H_{2n-1}^{(r)}\right]^m\frac{1}{n\binom{2n}{n}}\left( \frac{1+\tau^{2}}{\tau} \right)^{2n-1}\vphantom{\frac{\frac1\int}{\frac\int1}}\\\displaystyle=\int_0^1\left[\mathsf G^{(r),m}\left(\frac{(1+\tau^2)t}{(1+t^{2})\tau}\right)-\mathsf G^{(r),m}\left(-\frac{(1+\tau^2)t}{(1+t^{2})\tau}\right)\right]\frac{\D t}{1+t^{2}}\vphantom{\frac{\int}{\int}}\end{array}\\\hline(1,1,\varrho)&\frac{7}{2\sqrt{3}}\I\Li_{2}(\omega)-\frac{\pi\varLambda}{3\sqrt{3}}\vphantom{\dfrac{\int}{1}}\\[5pt](1,2,\varrho)&-\frac{2\cdot7}{\sqrt{3}}\I L_{2,1}(\omega)-\frac{2^{5}}{3\sqrt{3}}\I L_{2,1}(i)+\frac{2^3}{\sqrt{3}}\I \Li_{1,1,1}\big(\frac{i}{\varrho},1,-1\big)-\frac{7\varLambda}{\sqrt{3}}\I \Li_2(\omega)\\[2pt]&{}-\frac{2^4G\lambda}{3\sqrt{3}}+\frac{2\pi\lambda^2}{\sqrt{3}}+\frac{\pi\varLambda^2}{3\sqrt{3}}+\frac{2^{2}\pi\lambda}{\sqrt{3}}\R\Li_1\big(\frac{i}{\varrho}\big)-\frac{\pi}{3\sqrt{3}}\big[\R\Li_1\big(\frac{i}{\varrho}\big)\big]^2-\frac{97\pi^3}{2^4\cdot3^4\sqrt{3}}\\[5pt](2,1,\varrho)&\frac{2^{5}}{3\sqrt{3}}\I L_{2,1}(i)-\frac{2^3}{\sqrt{3}}\I \Li_{1,1,1}\big(\frac{i}{\varrho},1,-1\big)+\frac{2^4G\lambda}{3\sqrt{3}}-\frac{2\pi\lambda^2}{\sqrt{3}}-\frac{2^{2}\pi\lambda}{\sqrt{3}}\R\Li_1\big(\frac{i}{\varrho}\big)\\&{}+\frac{\pi}{3\sqrt{3}}\big[\R\Li_1\big(\frac{i}{\varrho}\big)\big]^2+\frac{7\cdot13\pi^{3}}{2^{4}\cdot3^{3}\sqrt{3}}\\[5pt]\hline(1,1,\theta)&\frac{3G}{\sqrt{2}}-\frac{\pi\lambda}{2^2\sqrt{2}}\vphantom{\dfrac{\int}{1}}\\[5pt](1,2,\theta)&-\frac{7}{\sqrt{2}}\I L_{2,1}(i)+\frac{2}{\sqrt{2}}\I \Li_{2,1}(\theta,\theta)-\frac{2\cdot3}{\sqrt{2}}\I L_{2,1}(\theta)+\frac{2}{\sqrt{2}}[\lambda+3\R \Li_1(\theta)]\I\Li_{2}(\theta)\\[2pt]{}&-\frac{7G\lambda}{2\sqrt{2}}+\frac{\pi\lambda^2}{2^{3}\sqrt{2}}-\frac{\pi}{2^2\sqrt{2}}[\R \Li_1(\theta)-\R \Li_1( i\theta)]^2+\frac{3\cdot11\pi^{3}}{2^{8}\sqrt{2}}\\[5pt](2,1,\theta)&\frac1{\sqrt{2}}\I L_{2,1}(i)-\frac{2}{\sqrt{2}}\I \Li_{2,1}(\theta,\theta)+\frac{2\cdot3}{\sqrt{2}}\I L_{2,1}(\theta)-\frac{2}{\sqrt{2}}[\lambda+3\R \Li_1(\theta)]\I\Li_{2}(\theta)\\&{}+\frac{G\lambda}{2\sqrt{2}}+\frac{\pi}{2^2\sqrt{2}}[\R \Li_1(\theta)-\R \Li_1( i\theta)]^2+\frac{61\pi^{3}}{2^{8}\cdot3\sqrt{2}}
\\[5pt]\hline\big(1,1,\frac{i}{\varrho}\big)&2^{2}\I\Li_{2}(\omega)+\frac{\pi\varLambda}{3}\vphantom{\dfrac{\int}{1}}\\[5pt]\big(1,2,\frac{i}{\varrho}\big)&13\I L_{2,1}(\omega)+\frac{13\varLambda}{2}\I\Li_2(\omega)+\frac{\pi\varLambda^{2}}{2^{2}\cdot3}+\frac{59\pi^3}{2^2\cdot3^4}\\[5pt]\big(2,1,\frac{i}{\varrho}\big)&-5\I L_{2,1}(\omega)-\frac{5\varLambda}{2}\I\Li_2(\omega)+\frac{\pi\varLambda^{2}}{2^{2}\cdot3}+\frac{5\cdot7\pi^3}{2^2\cdot3^4}\\[5pt]\hline\hline
\end{array}\end{align*}\end{scriptsize}
\end{table}

With \begin{align}
\int_1^i\left( \frac{1+\tau^{2}}{\tau} \right)^{2n}\frac{i\D \tau}{1+\tau^2}=-\frac{2^{4n-2}}{n\binom{2n}{n}}
\end{align}in hand,
we can prove \eqref{eq:KWY_quad'} by integrating \eqref{eq:KWY_inv_binom_alt} along a straight line segment (on which $ |1+\tau^2|\leq2|\tau|$) and applying Panzer's logarithmic regularization \cite[\S2.3]{Panzer2015} at $\tau=i$.

It is clear that the $ \mathbb Q$-vector spaces appearing on the right-hand sides of   \eqref{eq:KWY1}--\eqref{eq:KWY_quad'}    incorporate their counterparts in \eqref{eq:KWY0}--\eqref{eq:KWY_quad}  as subsets. Therefore, in view of  \eqref{eq:2n-1_H}, we can replace any occurrence of  $ {{\mathsf H}}_{\smash[t]{2}n-1}^{( r)}+\smash[t]{\overline{\mathsf H}}_{\smash[t]{2}n-1}^{(r)}$ in    \eqref{eq:KWY1}--\eqref{eq:KWY_quad'}     by   $ {{\mathsf H}}_{\smash[t]{2}n-1}^{( r)} $, without affecting the corresponding CMZV characterizations.
    \qedhere\end{enumerate}\end{proof}\begin{remark}With Panzer's \texttt{HyperInt} \cite{Panzer2015} and Au's \texttt{MultipleZetaValues} \cite{Au2022a}, one can verify all the entries of Tables \ref{tab:inv_binom}--\ref{tab:inv_binom2}. \eor\end{remark}
\begin{remark}In \cite[(2.1)]{Kalmykov2007}, Kalmykov--Ward--Yost considered\begin{align}
\sum_{n=1}^\infty \left[\prod_{j=1}^M\mathsf H_{n-\smash[t]{1}}^{(r_j)}\right]\left[\prod_{j'=1}^{M'}\mathsf H_{\smash[t]{2}n-\smash[t]{1}}^{(r_{j'}')}\right]\frac{\big(\frac{1-\tau^{2}}{1+\tau^{2}} \big)^{\delta_{0,s}}}{n^{s+1}\binom{2n}{n}}\left( \frac{1+\tau^{2}}{\tau} \right)^{2n}\label{eq:KMYform}
\end{align}among other things. In their capacities as $ \mathbb Q$-linear combinations of the left-hand side in \eqref{eq:KWY1}, these Kalmykov--Ward--Yost  series are also explicitly representable as members of  $\mathfrak g^{(\tau)}_{k+1}(4)$ for $ s\in\mathbb Z_{\geq0}$ and $ k=s+\sum_{j=1}^M r_j+\sum_{{j}'=1}^{{M'}} {r}'_{j'}$. For  $\tau=\varrho$ and small values of $k$, Kalmykov--Veretin  \cite[Appendix B]{KalmykovVeretin2000} proposed closed-form formulae for some special cases of \eqref{eq:KMYform}, based on numerical evidence. Diligent readers may check these experimental findings analytically, using  Au's package \cite{Au2022a}. For generic  $\tau$ and small values of $k$, Davydychev--Kalmykov \cite{DavydychevKalmykov2004} evaluated  \eqref{eq:KMYform} in terms of classical polylogarithms and their associates.      \eor\end{remark}

\begin{remark}
We note that \begin{align}\begin{split}{}&
{{\mathsf H}}_{n-\vphantom{1}\smash[t]{\frac12}}^{(r)}+2\delta_{1,r}\log2-2^{r-1}\left[{{\mathsf H}}_{2n-1}^{(r)}+\smash[t]{\overline{\mathsf H}}_{2n-1}^{(r)}\right]\\={}&
\begin{cases}0, &r=1,  \\
2(1-2^{r-1})\zeta_r, & r\in\mathbb Z_{>1}
\end{cases}\end{split}\label{eq:xi_r}
\end{align}is always a member of $  \mathfrak g_{r}^{(\varrho)}(2)$.  Therefore, the relation in   \eqref{eq:KWY1}  can be rewritten as
\begin{align}
\sum_{n=1}^\infty \left[ \prod_{j=1}^M\mathsf H_{n-\smash[t]{1}}^{(r_j)} \right]\left\{\prod_{\smash[b]{j'}=1}^{M'}\Big[{{\mathsf H}}_{n-\smash[t]{\frac12}}^{( {r}'_{j'})}+2\delta_{1,r'_{j'}}\log2\Big]\right\}\frac{\big(i\sqrt{3} \big)^{\delta_{0,s}}}{n^{s+1}\binom{2n}{n}}\in \mathfrak g^{(\varrho)}_{k+1}(4)\cap \mathfrak Z_{k+1}^{}(12),\tag{\ref{eq:KWY1}$'$}
\end{align}
 whose  $ s=0$ case combines with \eqref{eq:log3H} into\begin{align}
i\sqrt{3}\sum_{n=0}^\infty c_n\prod_{j=1}^M\mathfrak h_n^{(r_j)}\in \mathfrak g^{(\varrho)}_{k+1}(4)\cap  \mathfrak Z_{k+1}^{}(12).\label{eq:Z12_inv_binom}
\end{align} Unfortunately, this result  is not sharp enough  [cf.\ \eqref{eq:W3_Z_k(6)} in \S\ref{subsec:eps_expn}] to recover the  MPLs at sixth roots exhibited in Theorem \ref{thm:Zk(6)}, let alone the third roots sharpened by Theorem \ref{thm:Zk(3)}.

We initially confirmed the sufficiency of third roots  in the MPL characterizations of  $ m_k(1+x_1+x_2)$ for small positive  integers $k$\ by checking manually that these series representations [cf.\ \eqref{eq:W3deriv}] of hyper-Mahler measures fit into the framework of  \eqref{eq:KWY_inv_binom_star} and   \eqref{eq:KWY0}. Later on, we performed  variable substitutions  and  contour deformations on the Borwein--Borwein--Straub--Wan integrals (Theorem \ref{thm:Zk(3)}) that led us to $ \pi i\,m_k^{}(1+x_1+x_2)\in\mathfrak Z_{k+1}^{}(3)$ as a whole, for all positive integers $k$.
   \eor\end{remark}

In parallel to the inverse binomial sums treated in Corollary  \ref{cor:inv_binom_sum}, we will soon present period characterizations for the binomial sums, the latter of which played their part in  some hyper-Mahler measures studied by Kurokawa--Lal\'in--Ochiai \cite[Theorem 19]{KurokawaLalinOchiai2008}. \begin{corollary}[Binomial sums via GPLs]\label{cor:binom_sum}As before, we set    $ r_1,\dots,r_M\in\mathbb Z_{>0}$ and  $ {r}'_1,\dots,{r}'_{{M'}}\in\mathbb Z_{>0}$. \begin{enumerate}[leftmargin=*,  label=\emph{(\alph*)},ref=(\alph*),
widest=a, align=left]
\item If $ 4|\chi|\leq |1+\chi|^{2}$ and $s\in\mathbb Z_{>0} $, then we have (cf.\ \cite[(1.1) and  (5.2b)]{Kalmykov2007})\begin{align}\sum_{n=1}^\infty \left[ \prod_{j=1}^M\mathsf H_{n-\smash[t]{1}}^{(r_j)} \right]\frac{\binom{2n}{n}}{n^{s}}\left[ \frac{\chi}{(1+\chi)^{2}} \right]^{n}\in  \mathfrak g^{(\chi)}_{k}(2)\cap\mathfrak g_{k,1}^{\left\{-1,\chi,\frac{1}{\vphantom{1}\smash[b]{\chi}}\right\}},\label{eq:KWY2}\end{align}where $ k=s+\sum_{j=1}^Mr_j$. \item If $ 2|\upsilon|\leq |1+\upsilon^{2}| $ and $s\in\mathbb Z_{>0} $, then we have\begin{align}\begin{split}&
\sum_{n=1}^\infty \left[\prod_{j=1}^M{{\mathsf H}}_{\smash[t]{}n-1}^{( {r}_{j})}\right]\left\{\prod_{\smash[b]{j'}=1}^{M'}\Big[{{\mathsf H}}_{\smash[t]{2}n-1}^{( {r}'_{j'})}+\smash[t]{\overline{\mathsf H}}_{\smash[t]{2}n-1}^{( {r}'_{j'})}\Big]\right\}\frac{\binom{2n}{n}}{n^{s}}\left( \frac{\upsilon}{1+\upsilon^{2}} \right)^{2n}\\\in{}&\mathfrak g_{k}^{(\upsilon)}(4)\cap\mathfrak g_{k,2}^{{{\left\{ i,-i,\upsilon,-\frac{1}{\upsilon},-\upsilon,\frac{1}{\upsilon}\right\}}} },\end{split}\label{eq:KWY2_alt}
\end{align}where $ k=s+\sum_{j=1}^Mr_j+\sum_{{j}'=1}^{{M'}} {r}'_{j'}$.\end{enumerate}\end{corollary}\begin{proof}All the infinite series of our concern are convergent, thanks to   the asymptotic behavior $ \binom{2n}n={\frac{4^{n}}{\sqrt{n\pi}}}\left[ 1+O\left( \frac{1}{n} \right) \right]$ as $ n\to\infty$.\begin{enumerate}[leftmargin=*,  label={(\alph*)},ref=(\alph*),
widest=a, align=left]
\item When $s\in\mathbb Z_{>0} $,  the left-hand side of \eqref{eq:KWY2} is \begin{align}\begin{split}&
\sum_{n=1}^\infty \left[ \prod_{j=1}^M\mathsf H_{n-\smash[t]{1}}^{(r_j)} \right]\left[ \frac{16\chi}{(1+\chi)^{2}} \right]^{n}\int_0^1\frac{x^{{n}}(1-x)^n}{\pi n^{s}}\frac{\D x}{\sqrt{x(1-x)}}\\\in{}&\Span_{\mathbb Q}\left\{ \int_0^1L_{\bm A}\left(\frac{16\chi x(1-x)}{(1+\chi)^{2}}\right)\frac{\D x}{\pi\sqrt{x(1-x)}}\left| w(\bm A)=s+\sum_{j=1}^M r_j\right.\right\}.\end{split}
\end{align} Setting $ x=\frac{1}{2}-\frac{1}{4i}\big( \sqrt{t} -\frac{1}{\sqrt{t}}\big)$ for $t$ on the unit circle, we have \begin{align}
\int_0^1L_{\bm A}\left(\frac{16\chi x(1-x)}{(1+\chi)^{2}}\right)\frac{\D x}{\pi\sqrt{x(1-x)}}={\counterint}_{\!\!\!\!|t|=1}L_{\bm A}\left( \varXi(\chi,t) \right)\frac{\D t}{2\pi it},
\end{align}where $\varXi(\chi,t)\colonequals \frac{(1+t)^{2}\chi}{(1+\chi)^{2}t} $.

The following fibration of the multi-dimensional polylogarithm\begin{align}\begin{split}&
L_{\bm A}\left( \varXi(\chi,t) \right)\\\in{}&\Span_{\mathbb Q}\left\{ \left.\left(\pi i\frac{\I  \varXi(\chi,t) }{|\I \varXi(\chi,t) |}\right)^\ell Z_j\log^r\frac{1}{ \varXi(\chi,t)}\right|\begin{smallmatrix}\ell,j,r\in\mathbb Z_{\geq0}\\Z_j\in\mathfrak Z_j(1)\\\ell+j+r=w(\bm A)\end{smallmatrix} \right\}\\{}&+\Span_{\mathbb Q}\left\{ \left.\left(\pi i\frac{\I  \varXi(\chi,t) }{|\I \varXi(\chi,t) |}\right)^\ell Z_jL_{\bm B}\left( \frac{1}{ \varXi(\chi,t)} \right)\log^r\frac{1}{ \varXi(\chi,t)}\right| \begin{smallmatrix}\ell,j,r\in\mathbb Z_{\geq0};w(\bm B)\in\mathbb Z_{>0}\\Z_j\in\mathfrak Z_j(1)\\\ell+j+w(\bm B)+r=w(\bm A)\end{smallmatrix}\right\}\end{split}\label{eq:recip_fib}
\end{align} highlights its branch cuts   \cite{Maitre2012,Panzer2015},   on which $\varXi(\chi,t)>1$. If $ f(\varXi(\chi,t))$ belongs to the first summand on the right-hand side of \eqref{eq:recip_fib}, then \begin{align}
{\counterint}_{\!\!\!\!|t|=1}f\left( \varXi(\chi,t) \right)\frac{\D t}{2\pi it}\in\mathfrak g_{w(\bm A)}^{(\chi)}(2)\cap\mathfrak g_{w(\bm A),1}^{\left\{-1,\chi,\frac{1}{\vphantom{1}\smash[b]{\chi}}\right\}}\label{eq:f_int_C}
\end{align}follows from the Borwein--Straub relation [the ``$\in$'' part of \eqref{eq:BorweinStraubZ1}] and the fact that $ \log\frac{\chi}{(1+\chi)^{2}}=G(0;\chi)-2G(-1;\chi)+2\pi i\mathbb Z\in\mathfrak g_{1}^{(\chi)}(2)\cap\mathfrak g_{1,1}^{\left\{-1,\chi,\frac{1}{\vphantom{1}\smash[b]{\chi}}\right\}}$. If $ f(\varXi(\chi,t))$ belongs to the second summand on the right-hand side of \eqref{eq:recip_fib}, then  \eqref{eq:f_int_C} remains valid, for two reasons: (1)~One can shrink the contour  to  tight loops wrapping around the branch cuts inside the unit circle,  whose boundary points belong to $ \big\{0,-1,\chi,\frac{1}{\chi}\big\} $; (2)~Across the branch cuts, the integrand $ f(\varXi(\chi,t))$ has jumps  in the $ \mathbb Q$-vector space $\pi i \mathfrak g_{w(\bm A)-1}^{(1/\varXi(\chi,t))}(1) $, so \eqref{eq:f_int_C} results from fibrations of GPLs with respect to $t$ and integrations of GPLs along the branch cuts.

 The arguments in the last paragraph take us to the the right-hand side of \eqref{eq:KWY2}.     \item When $s\in\mathbb Z_{>0} $,  the left-hand side of \eqref{eq:KWY2_alt} is \begin{align}\begin{split}&
\sum_{n=1}^\infty \left[ \prod_{j=1}^M\mathsf H_{n-\smash[t]{1}}^{(r_j)} \right]\left\{\prod_{\smash[b]{j'}=1}^{M'}\Big[{{\mathsf H}}_{\smash[t]{2}n-1}^{( {r}'_{j'})}+\smash[t]{\overline{\mathsf H}}_{\smash[t]{2}n-1}^{( {r}'_{j'})}\Big]\right\}\frac{1}{ n^{s}}{\counterint}_{\!\!\!\!|\tau|=1}\left[ \frac{(1+\tau^2)\upsilon}{(1+\upsilon^2)\tau} \right]^{2n}\frac{\D\tau}{2\pi i\tau}\\\in{}&\Span_{\mathbb Q}\left\{\left. {\counterint}_{\!\!\!\!|\tau|=1}\frac{G\left( \widetilde{\alpha}_1,\dots, \widetilde{\alpha}_{w};\frac{(1+\tau^2)\upsilon}{(1+\upsilon^2)\tau} \right)\D \tau}{2\pi i\tau}\right|\begin{smallmatrix}\widetilde{\alpha}_1,\dots, \widetilde{\alpha}_{w-1}\in\{-1,0,1\};\widetilde{\alpha}_w\in\{-1,1\}\\ w=s+\sum_{j=1}^M r_j+\sum_{{j}'=1}^{{M'}} {r}'_{j'}\end{smallmatrix}\right\},\end{split}
\end{align}  in the light of \eqref{eq:G_promo}.

 To verify the claim in   \eqref{eq:KWY2_alt}, use GPL fibrations and contour deformations as in our proof of \eqref{eq:KWY2}.   \qedhere\end{enumerate}\end{proof}\begin{remark}For the treatment of the $s=0$ case in binomial sums, see \cite[\S2]{Zhou2023SunCMZV}. The contour deformation techniques in the proof above can also be generalized to infinite sums involving $ \binom{3n}n$ and $ \binom{4n}{2n}$ \cite[\S3]{Zhou2023SunCMZV}.\eor\end{remark}
 \subsection{Infinite series related to $ m_k(1+x_1+x_2+x_{3})$\label{subsec:seriesW4}}Akin to the descriptions at the beginning of the last subsection, we once evaluated $ m_k(1+x_1+x_2+x_{3}),1\leq k\leq6$   by taking $ k$-th order derivatives of Crandall's hypergeometric representation  \cite[(6.8)]{BSWZ2012} $ W_4(s)=U_4(s)+V_4(s)$, where \begin{align}\label{eq:U4Crandall}
U_{4}(s)\colonequals{}&\frac{2^s}{\cot\frac{\pi  s}{2}}  \left[ \frac{\Gamma \left(1+\frac{s}{2}\right)}{\sqrt{\pi}\Gamma \left(\frac{3+s}{2}\right)} \right]^3{_4F_3}\left(\!\!\left.\begin{smallmatrix}\frac{1}{2},\frac{1}{2},\frac{1}{2},1+\frac{s}{2}\\[2pt]\frac{3+s}{2},\frac{3+s}{2},\frac{3+s}{2}\end{smallmatrix}\right|1\right),\\\label{eq:V4Crandall}V_{4}(s)\colonequals{}&\frac{2^s \Gamma \left(\frac{1+s}{2}\right) }{\sqrt{\pi } \Gamma \left(1+\frac{s}{2}\right)}{_4F_3}\left(\!\!\left.\begin{smallmatrix}\frac{1}{2},-\frac{s}{2},-\frac{s}{2},-\frac{s}{2}\\[2pt]1,1,\frac{1-s}{2}\end{smallmatrix}\right|1\right).
\end{align}During these evaluations, the derivative\begin{align}
\left.\frac{\D^k U_4(s)}{\D s^k}\right|_{s=0}\quad\left[ \text{resp. }\left.\frac{\D^k V_4(s)}{\D s^k}\right|_{s=0} \right]
\end{align} produced  a $ \mathbb Q$-linear combination of\begin{align}
\pi^{2(r_{0}-1)}{}&\sum_{n=1}^\infty \frac{\prod_{j=1}^\ell\left[\mathsf H_{n-\smash[t]{1}}^{(r_j)}-3{{\mathsf H}}_{\smash[t]{}n-\smash[t]{\frac12}}^{( {r}_{j})}+2\delta_{1,r_{j}}\log 2\right]}{(2n-1)^{3}}\label{eq:u4}\\\left(\text{resp. }\pi^{2\widetilde r_0} \vphantom{{\frac{\prod_{j=1}^{\widetilde{\ell}} \left[3\mathsf H_{n-\smash[t]{1}}^{(\widetilde{r}_j)}-{{\mathsf H}}_{\smash[t]{}n-\smash[t]{\frac12}}^{(\widetilde{ {r}}_{j})}-2\delta_{1,\widetilde{r}_{j}}\log 2\right]}{n^3}}}\right.{}&\!\!\left.\sum_{n=1}^\infty\frac{\prod_{j=1}^{\widetilde{\ell}} \left[3\mathsf H_{n-\smash[t]{1}}^{(\widetilde{r}_j)}-{{\mathsf H}}_{\smash[t]{}n-\smash[t]{\frac12}}^{(\widetilde{ {r}}_{j})}-2\delta_{1,\widetilde{r}_{j}}\log 2\right]}{n^3}\right), \label{eq:v4}\end{align}where $ 2r_0+\sum_{j=1}^\ell r_j=k-1$, $ 2\widetilde{r}_0+\sum_{j=1}^{\widetilde{\ell}} \widetilde{r}_j=k+1$, and  $ r_0,\widetilde{r}_0\in\mathbb Z_{\geq0}$. (Here, both $ \ell $ and $\widetilde \ell$ may be zero, whereupon the empty products are equal to 1.)

The next corollary will allow us to   strengthen Theorem \ref{thm:Zk(2)} as\begin{align}
\left.\frac{\D^k U_4(s)}{\D s^k}\right|_{s=0}\in\frac{1}{\pi^2}\mathfrak Z_{k+2}(2),\quad \left.\frac{\D^k V_4(s)}{\D s^k}\right|_{s=0}\in\mathfrak Z_{k}(2).\label{eq:UV_AMZV}
\end{align}Such AMZV characterizations will also apply to  Broadhurst's series in \eqref{eq:BroadhurstEuler}.

\begin{corollary}[Euler sums, HPLs, and Goncharov--Deligne periods]\label{cor:odd_var}\begin{enumerate}[leftmargin=*,  label=\emph{(\alph*)},ref=(\alph*),
widest=a, align=left]
\item For $ s\in\mathbb Z_{\geq0}$,   $ r_1,\dots,r_M\in\mathbb Z_{>0}$, and  $ {r}'_1,\dots,{r}'_{{M'}}\in\mathbb Z_{>0}$, we have\begin{align}
\sum_{n=1}^\infty \left[\prod_{j=1}^M\mathsf H_{n-\smash[t]{1}}^{(r_j)}\right]\left\{\prod_{\smash[b]{j'}=1}^{M'}\Big[{{\mathsf H}}_{\smash[t]{2}n-1}^{( {r}'_{j'})}+\smash[t]{\overline{\mathsf H}}_{\smash[t]{2}n-1}^{( {r}'_{j'})}\Big]\right\}\frac{z^{2n-1}}{(2n-1)^{s+1}}\in\mathfrak G^{(z)}_{k+1}(2),\label{eq:BE_HPL}
\end{align}where $ |z|<1$, $ k=s+\sum_{j=1}^M r_j+\sum_{{j}'=1}^{{M'}} {r}'_{j'}$. For $ s\in\mathbb Z_{\geq0}$,   $ r_1,\dots,r_M\in\mathbb Z_{>0}$, and  $ {r}'_1,\dots,{r}'_{{M'}}\in\mathbb Z_{>0}$, the result above extends to \begin{align}
\sum_{n=1}^\infty \left[\prod_{j=1}^M\mathsf H_{n-\smash[t]{1}}^{(r_j)}\right]\left\{\prod_{\smash[b]{j'}=1}^{M'}\Big[{{\mathsf H}}_{\smash[t]{2}n-1}^{( {r}'_{j'})}+\smash[t]{\overline{\mathsf H}}_{\smash[t]{2}n-1}^{( {r}'_{j'})}\Big]\right\}\frac{1}{(2n-1)^{s+2}}\in{}&\mathfrak Z_{k+2}(2),\label{eq:BE_Zk(2)}\\\sum_{n=1}^\infty \left[\prod_{j=1}^M\mathsf H_{n-\smash[t]{1}}^{(r_j)}\right]\left\{\prod_{\smash[b]{j'}=1}^{M'}\Big[{{\mathsf H}}_{\smash[t]{2}n-1}^{( {r}'_{j'})}+\smash[t]{\overline{\mathsf H}}_{\smash[t]{2}n-1}^{( {r}'_{j'})}\Big]\right\}\frac{(-1)^{n}}{(2n-1)^{s+1}}\in{}& i\mathfrak Z_{k+1}(4),\label{eq:BE_Zk(4)}
\end{align}where $ k=s+\sum_{j=1}^M r_j+\sum_{{j}'=1}^{{M'}} {r}'_{j'}$.\item For $ s\in\mathbb Z_{\geq0}$,   $ r_1,\dots,r_M\in\mathbb Z_{>0}$, and  $ {r}'_1,\dots,{r}'_{{M'}}\in\mathbb Z_{>0}$, we have\begin{align}
\sum_{n=1}^\infty \left[\prod_{j=1}^M\mathsf H_{n-\smash[t]{1}}^{(r_j)}\right]\left\{\prod_{\smash[b]{j'}=1}^{M'}\Big[{{\mathsf H}}_{\smash[t]{2}n-1}^{( {r}'_{j'})}+\smash[t]{\overline{\mathsf H}}_{\smash[t]{2}n-1}^{( {r}'_{j'})}\Big]\right\}\frac{z^{2n}}{n^{s+1}}\in\mathfrak G^{(z)}_{k+1}(2),\label{eq:BE_HPL_even}
\end{align}where $ |z|<1$, $ k=s+\sum_{j=1}^M r_j+\sum_{{j}'=1}^{{M'}} {r}'_{j'}$. For $ s\in\mathbb Z_{\geq0}$,   $ r_1,\dots,r_M\in\mathbb Z_{>0}$, and  $ {r}'_1,\dots,{r}'_{{M'}}\in\mathbb Z_{>0}$, the result above extends to \begin{align}
\sum_{n=1}^\infty \left[\prod_{j=1}^M\mathsf H_{n-\smash[t]{1}}^{(r_j)}\right]\left\{\prod_{\smash[b]{j'}=1}^{M'}\Big[{{\mathsf H}}_{\smash[t]{2}n-1}^{( {r}'_{j'})}+\smash[t]{\overline{\mathsf H}}_{\smash[t]{2}n-1}^{( {r}'_{j'})}\Big]\right\}\frac{1}{n^{s+2}}\in{}&\mathfrak Z_{k+2}(2),\label{eq:BE_Zk(2)_even}\\\sum_{n=1}^\infty \left[\prod_{j=1}^M\mathsf H_{n-\smash[t]{1}}^{(r_j)}\right]\left\{\prod_{\smash[b]{j'}=1}^{M'}\Big[{{\mathsf H}}_{\smash[t]{2}n-1}^{( {r}'_{j'})}+\smash[t]{\overline{\mathsf H}}_{\smash[t]{2}n-1}^{( {r}'_{j'})}\Big]\right\}\frac{(-1)^{n}}{n^{s+1}}\in{}& \mathfrak Z_{k+1}(4),\label{eq:BE_Zk(4)_even}
\end{align}where $ k=s+\sum_{j=1}^M r_j+\sum_{{j}'=1}^{{M'}} {r}'_{j'}$.\end{enumerate}
\end{corollary}

\begin{proof}\begin{enumerate}[leftmargin=*,  label={(\alph*)},ref=(\alph*),
widest=a, align=left]
\item To prove \eqref{eq:BE_HPL} for $s=0$, divide both sides of \eqref{eq:Gsymsum} by $z$ and integrate. For progress onto all  $ s\in\mathbb Z_{>0}$, apply the operations $ f(z)\mapsto \int_0^z f(x)\frac{\D x}{x}$ iteratively. In fact, this method allows us to deduce a slightly stronger statement:\begin{align}\begin{split}
&\sum_{n=1}^\infty \left[\prod_{j=1}^M\mathsf H_{n-\smash[t]{1}}^{(r_j)}\right]\left\{\prod_{\smash[b]{j'}=1}^{M'}\Big[{{\mathsf H}}_{\smash[t]{2}n-1}^{( {r}'_{j'})}+\smash[t]{\overline{\mathsf H}}_{\smash[t]{2}n-1}^{( {r}'_{j'})}\Big]\right\}\frac{z^{2n-1}}{(2n-1)^{s+1}}\\\in{}&\mathfrak  \Span_{\mathbb Q}\left\{ g(z)-g(-z)\left|g(z)\in\mathfrak G^{(z)}_{s+1+\sum_{j=1}^M r_j+\sum_{j'=1}^{M'} r'_{j'}}(2)\right.\right\}.\end{split}\label{eq:Sun_prep}
\end{align}

When $ s\in\mathbb Z_{\geq0}$ (resp.\ $ s\in\mathbb Z_{\geq1}$), the integral representation for the left-hand side of  \eqref{eq:BE_HPL} converges at $ z=i$ (resp.\ $ z=1$), according to variations on the proof of Theorem \ref{thm:nonlin_Euler}. Thus, the period characterizations in \eqref{eq:BE_Zk(2)} and \eqref{eq:BE_Zk(4)}  immediately follow. \item These are collateral results of  \eqref{eq:G_promo}.   \qedhere\end{enumerate} \end{proof}\begin{remark}We note that  Xu \cite{Xu2021} (resp.\ Xu--Wang  \cite{XuWang2022}) studied the $ M'=0$ (resp.\ $M=0$) cases of the series \eqref{eq:BE_Zk(2)} and \eqref{eq:BE_Zk(4)}, using different methods based on Hoffman's multiple $t$-values \cite{Hoffman2019}. A recent result of Au \cite[Theorem 1.1]{Au2022b} extends \eqref{eq:BE_Zk(2)} and \eqref{eq:BE_Zk(4)} to  $ \mathfrak Z(N)$ with $ N\in\mathbb Z_{\geq2}$.   \eor\end{remark}
\begin{remark}
 Since $ \log 2\in\mathfrak Z_1(2)$, one can confirm \eqref{eq:UV_AMZV} by combining \eqref{eq:BE_Zk(2)} and \eqref{eq:BE_Zk(2)_even} with \eqref{eq:xi_r}. Unlike the discrepancy between Theorem \ref{thm:Zk(3)} and \eqref{eq:Z12_inv_binom}, the AMZV characterizations in Theorem  \ref{thm:Zk(2)}  remain effective down to the individual non-linear Euler sums [see \eqref{eq:u4} and \eqref{eq:v4}] that constitute the $k$-th order derivatives in  \eqref{eq:UV_AMZV}.  \eor\end{remark}

The techniques in the last four corollaries can be combined to take some recently proven series evaluations \cite{Campbell2022WangChu,Charlton2022SunConj}  a fair distance.
\begin{corollary}[GPL and CMZV representations of generalized Sun series]\label{cor:Sun} If    $ r_1,\dots,r_M\in\mathbb Z_{>0}$,   $ {r}'_1,\dots,{r}'_{{M'}}\in\mathbb Z_{>0}$, $ |1+\tau^2|\leq 2|\tau| $, and $s\in\mathbb Z_{\geq0} $, then we have\begin{align}\begin{split}&
\sum_{n=1}^\infty \left[\prod_{j=1}^M\mathsf H_{n-\smash[t]{1}}^{(r_j)}\right]\left\{\prod_{\smash[b]{j'}=1}^{M'}\Big[{{\mathsf H}}_{\smash[t]{2}n-1}^{( {r}'_{j'})}+\smash[t]{\overline{\mathsf H}}_{\smash[t]{2}n-1}^{( {r}'_{j'})}\Big]\right\}\frac{\binom{2n-2}{n-1}\left( \frac{1+\tau^2}{\tau} \right)^{2n-1}}{2^{4(n-1)}(2n-1)^{s+1}}\\\in{}& i\left(\mathfrak g^{(\tau)}_{k+1}(4)\cap\mathfrak g_{k+1,2}^{\left\{ i,-i,\tau ,-\frac{1}{\tau},-\tau,\frac{1}{\tau}\right\}}\right),\end{split}\label{eq:Sun_gen}
\end{align}where  $ k=s+\sum_{j=1}^M r_j+\sum_{{j}'=1}^{{M'}} {r}'_{j'}$.

 For $ \tau=\varrho\colonequals e^{\pi i/3}$, we have {\small\begin{align}
\sum_{n=1}^\infty \left[\prod_{j=1}^M\mathsf H_{n-\smash[t]{1}}^{(r_j)}\right]\left\{\prod_{\smash[b]{j'}=1}^{M'}\Big[{{\mathsf H}}_{\smash[t]{2}n-1}^{( {r}'_{j'})}+\smash[t]{\overline{\mathsf H}}_{\smash[t]{2}n-1}^{( {r}'_{j'})}\Big]\right\}\frac{\binom{2n-2}{n-1}}{2^{4(n-1)}(2n-1)^{s+1}}\in i\left[\mathfrak g^{(\varrho)}_{k+1}(4)\cap\mathfrak Z_{k+1}(12)\right],\label{eq:Sun12}
\end{align}}where  $ k=s+\sum_{j=1}^M r_j+\sum_{{j}'=1}^{{M'}} {r}'_{j'}$. It is possible to use CMZVs at a lower level to represent a special subclass of \eqref{eq:Sun12}, namely  \begin{align}
\sum_{n=1}^\infty \left[\prod_{j=1}^M\mathsf H_{n-\smash[t]{\frac12}}^{(r_j)}\right]\frac{\binom{2n-2}{n-1} }{2^{4(n-1)}(2n-1)^{s+1}}\in i\mathfrak Z_{k+1}(3),\label{eq:Sun12'}\tag{\ref{eq:Sun12}$'$}
\end{align}where  $ k=s+\sum_{j=1}^M r_j$.

 For  $ \tau=i\rho\colonequals i\frac{\sqrt{5}-1}{2}$, we have {\small\begin{align}
\sum_{n=1}^\infty \left[\prod_{j=1}^M\mathsf H_{n-\smash[t]{1}}^{(r_j)}\right]\left\{\prod_{\smash[b]{j'}=1}^{M'}\Big[{{\mathsf H}}_{\smash[t]{2}n-1}^{( {r}'_{j'})}+\smash[t]{\overline{\mathsf H}}_{\smash[t]{2}n-1}^{( {r}'_{j'})}\Big]\right\}\frac{\binom{2n-2}{n-1}\left( -1 \right)^{n-1}}{2^{4(n-1)}(2n-1)^{s+1}}\in \mathfrak g^{(i\rho)}_{k+1}(4)\cap\mathfrak Z_{k+1}(60),\label{eq:Sun60}
\end{align}}where  $ k=s+\sum_{j=1}^M r_j+\sum_{{j}'=1}^{{M'}} {r}'_{j'}$. If $ k+1\leq3$, then
  the left-hand side of \eqref{eq:Sun60} belongs to  $ \mathfrak Z_{k+1}(10)$.

   \end{corollary}
\begin{proof}Noting that \begin{align}
\frac{\binom{2n-2}{n-1}}{2^{4(n-1)}}=\int_0^1\left[\sqrt{X(1-X)}\right]^{2n-1}\frac{\D X}{\pi X(1-X)},
\end{align}we may identify the left-hand side of \eqref{eq:Sun_gen} with a member of \begin{align}
\Span_{\mathbb Q}\left\{ \left.\int_0^\infty\left[ g\left( \frac{(1+\tau^2)t}{(1+t^{2})\tau} \right)-g\left(- \frac{(1+\tau^2)t}{(1+t^{2})\tau} \right) \right]\frac{\D t}{\pi t}\right|g(z)\in \mathfrak G^{(z)}_{k+1}(2)\right\},
\end{align}by virtue of the symmetry displayed in \eqref{eq:Sun_prep} and a variable substitution $ X=\frac{1}{1+t^2}$. This effectively reduces our target to an integral\begin{align}
\int_{-\infty}^\infty G\left( \alpha_{1},\dots, \alpha_{k+1};\frac{(1+\tau^2)t}{(1+t^{2})\tau} \right) \frac{\D t}{\pi t},\quad \alpha_{1},\dots, \alpha_{k+1}\in\{-1,0,1\},
\end{align}whose contour   can be deformed to loops wrapping around the branch cuts in the upper half-plane, on which   $ \frac{(1+\tau^2)t}{(1+t^{2})\tau}\in(-\infty,-1)\cup(1,\infty)\subset\mathbb R$. The end points of these branch cuts are in $ \big\{i,\tau,\frac1\tau,-\tau,-\frac{1}{\tau}\big\}\cap\{w\in\mathbb C|\I w>0\}$. As we fibrate with respect to $t$ (resp.\ $\tau$) before (resp.\ after) integration over branch cuts, we arrive at the claim in the right-hand side of  \eqref{eq:Sun_gen}.

The rationales behind the right-hand sides of  \eqref{eq:Sun12} and \eqref{eq:Sun60} are similar to those for \eqref{eq:KWY1} and \eqref{eq:Zk60}.

\begin{table}[t]
\caption{Selected closed forms for   binomial sums of Sun's type, where  $ \omega\colonequals e^{2\pi i/3}$, $ \varrho\colonequals e^{\pi i/3}$, $\rho\colonequals\frac{\sqrt{5}-1}{2}$,  $ \varsigma\colonequals e^{2\pi i/5}$,  $ G\colonequals\I\Li_2(i)$, and   $ \lambda\colonequals\log2$\label{tab:binom_sums}}

\begin{scriptsize}
\begin{align*}
\begin{array}{c|ll}
\hline\hline s&\begin{array}{@{}l}\displaystyle \quad \sum_{n=1}^\infty\frac{\binom{2n-2}{n-1}\vphantom{\frac{\frac11}{}}}{2^{4(n-1)}(2n-1)^{s+1}}\\\displaystyle=\frac{1}{2^ss!}\int_1^\varrho\left[ \log\left( -\frac{u}{(1-u)^2} \right) \right]^s\frac{\D u}{iu}\vphantom{\frac{}{\frac11}}\end{array}&\begin{array}{@{}l}\displaystyle \quad \sum_{n=1}^\infty \mathsf H^{(1)}_{n-\smash[t]{\frac12}}\frac{\binom{2n-2}{n-1}\vphantom{\frac{\frac11}{}}}{2^{4(n-1)}(2n-1)^{s}}\\\displaystyle=\frac{1}{2^{s-1}}\int_1^\varrho\mathsf L_s^{(1)}\left( -\frac{(1-u)^2}{u} \right)\frac{\D u}{iu}\end{array}\\\hline
1&\frac{3\I\Li_2(\omega)}2\vphantom{\dfrac{\int}{1}}&\I\Li_2(\omega)\\[5pt]2&\frac{7\pi^3}{2^3\cdot3^3}&3\I L_{2,1}(\omega)+\frac{\pi^3}{3^3}{}\\[5pt]3&\frac{3^3\I \Li_4(\omega)}{2^5}+\frac{\pi\zeta_3}{2^{2}\cdot3}&\frac{3\I \Li_4(\omega)}{2^3}+\frac{\pi\zeta_3}{2^{2}\cdot3}\\[5pt]4&-\frac{3^4\I L_{4,1}(\omega)}{2^3\cdot13}-\frac{3^3\I L_{3,2}(\omega)}{2^3\cdot13}+\frac{229\pi^{5}}{2^{7}\cdot3^2\cdot5\cdot13}&-\frac{3^4\I L_{4,1}(\omega)}{2^2\cdot13}-\frac{3^2\cdot5^2\I L_{3,2}(\omega)}{2^4\cdot13}+\frac{5\cdot19\pi^{5}}{2^{6}\cdot3^4\cdot13}\\[5pt]\hline\hline\multicolumn{3}{c}{}\\\hline\hline s&\multicolumn{2}{l}{\displaystyle  \sum_{n=1}^\infty \mathsf H^{(1)}_{2n-1}\frac{\binom{2n-2}{n-1}\vphantom{\frac{\frac11}{}}}{2^{4(n-1)}(2n-1)^{s}}=2\int_{\varrho}^i\mathsf L_s^{(1)}\left( \frac{1+t^{2}}{t} \right)\frac{\D t}{it}\vphantom{\frac{}{\frac11}}}\\\hline1&\multicolumn{2}{l}{\frac{2^3G}{3}-2\I \Li_2(\omega)\vphantom{\dfrac{\frac11}{}}}\\[5pt]2&\multicolumn{2}{l}{3\I L_{2,1}(\omega)-\frac{2^5\I L_{2,1}(i)}{3}+2^3\I L_{2,1}\big(\frac{i}{\varrho}\big)+2^{3}\I\Li_{1,1,1}\big(\frac{i}{\varrho},1,-1\big)-\frac{2^{4}G}{3}\big[\lambda+\R\Li_1\big(\frac{i}{\varrho}\big)\big]}\\&\multicolumn{2}{l}{{}+\pi \R\Li_1\big(\frac{i}{\varrho}\big)\big[\frac{1}{3}\R\Li_1\big(\frac{i}{\varrho}\big)+2^{2}\lambda\big]+2\pi\lambda^2-\frac{11\cdot19\pi^{3}}{2^{4}\cdot3^4}}\\[5pt]\hline\hline\multicolumn{3}{c}{}\\\hline\hline s&\begin{array}{@{}l}\displaystyle \quad \sum_{n=1}^\infty\frac{\binom{2n-2}{n-1}(-1)^{n-1}\vphantom{\frac{\frac11}{}}}{2^{4(n-1)}(2n-1)^{s+1}}\\\displaystyle=\frac{2}{s!}\int_{i\rho}^i\mathsf \log^{s}\left( \frac{t}{i(1+t^{2})} \right)\frac{\D t}{t}\vphantom{\frac{}{\frac11}}\end{array}&\begin{array}{@{}l}\displaystyle \quad \sum_{n=1}^\infty \mathsf H^{(1)}_{2n-1}\frac{\binom{2n-2}{n-1}(-1)^{n-1}\vphantom{\frac{\frac11}{}}}{2^{4(n-1)}(2n-1)^{s}}\\\displaystyle=2\int_{i\rho}^i\mathsf L_s^{(1)}\left( \frac{i(1+t^{2})}{t} \right)\frac{\D t}{t}\end{array}\\\hline 1&\frac{\pi^2}{2\cdot5}&-2^2\R\Li_{1,1}\left( -\frac{1}{\varsigma^{2}} ,\varsigma\right)+3\left\{\R[\Li_1(\varsigma^2)-\Li_1(\varsigma)]\right\}^{2}\vphantom{\frac{\frac\int1}{1}}\\&&{}-2^{2}\left\{ [\R\Li_1(\varsigma^2)]^{2} -[\R\Li_1(\varsigma)]^{2} \right\}-\frac{2^{2}\pi^2}{3\cdot5^2}\\[5pt]2&\frac{5^{2}\R\Li_3(\varsigma)}{2^{2}\cdot3}+\frac{\zeta_3}{2}&-5\R\Li_3(\varsigma)+\frac{2^3\zeta_3}{5}\\[5pt]\hline\hline\end{array}\end{align*}
\end{scriptsize}
\end{table}

Enlisting the help from \eqref{eq:Hhalf_prod}, we may equate the left-hand side of  \eqref{eq:Sun12'}  with\footnote{To construct Tables  \ref{tab:binom_sums} and \ref{tab:binom_sums_quad}, we need the following special cases of $ \mathsf L_s^{(1)}(x)=\frac{1}{2\pi i}\Big[\mathsf G_s^{(1),1}\left(\frac1{x-i0^+}\right)-\mathsf G_s^{(1),1}\left(\frac1{x+i0^+}\right)\Big]$ for $0<x<1$: $\mathsf L_1^{(1)}(x)= \Li_1(x)$, $ \mathsf L_2^{(1)}(x)=-\Li_2(x)+\frac{\pi^2}{6}$, $\mathsf L_3^{(1)}(x)=\Li_3(x)-\frac{\pi^2\log x}{6}-\zeta_3 $, and  $ \mathsf L_4^{(1)}(x)=-\Li_4(x)+\zeta_3\log x+\frac{\pi^2\log ^{2}x}{12}+\frac{\pi^{4}}{90}$.} \begin{align}\begin{split}&\frac{1}{2^{s}}
\sum_{n=1}^\infty \frac{\binom{2n-2}{n-1} }{2^{4(n-1)}}\int_{0}^1 \frac{\mathsf L_{s+1}^{(r_1,\dots,r_M)}(x^2)}{2n-1}\D x^{2n-1}\\={}&\frac{1}{2^{s-1}}\int_{\tau=i}^{\tau=\varrho}\mathsf L_{s+1}^{(r_1,\dots,r_M)}\left( \frac{(1+\tau^{2})^2}{\tau^{2}} \right)\D \arcsin\frac{1+\tau^{2}}{2\tau},\\\end{split}
\end{align}  where $ \mathsf L_{s+1}^{(r_1,\dots,r_M)}(X)\in\mathfrak g_{k}^{(X)}(1)$ and $ k=s+\sum_{j=1}^M r_j$. Introducing  a new variable $ u=-\tau^2$, one can reveal the expression above as member of {\small\begin{align}\begin{split}{}&
\Span_{\mathbb Q}\left\{\left.(\pi i)^{k-\ell}Z_{\ell-j}\int_{1}^{1/\varrho}G\left( \alpha_{1},\dots,\alpha_j;-\frac{(1-u)^2}{u} \right)\frac{\D u}{ iu}\right|\begin{smallmatrix}\alpha_{1},\dots,\alpha_j\in\{0,1\}\\Z_{\ell-j}\in\mathfrak Z_{\ell-j}(1)\\0\leq j\leq\ell\leq k\end{smallmatrix}\right\}\\\subseteq{}&\Span_{\mathbb Q}\left\{(\pi i)^{k-\ell}Z_{\ell-j}\left.\int_{1}^{1/\varrho}G\left( \alpha_{1},\dots,\alpha_j;u \right)\frac{\D u}{iu}\right|\begin{smallmatrix}\alpha_{1},\dots,\alpha_j\in\left\{0,1,\varrho=\frac{\omega^{2}-1}{z-1},\frac{1}{\varrho}=\frac{z}{\omega^{2}}\frac{\omega^{2}-1}{z-1}\right\}\\Z_{\ell-j}\in\mathfrak Z_{\ell-j}(1)\\z=\omega=e^{2\pi i/3};0\leq j\leq\ell\leq k\end{smallmatrix}\right\}\\\subseteq{}&\left.\frac{\mathfrak g_{k+1}\big[ 1,\omega,\omega^{2};z \big](3)}{i}\right|_{z=\omega} \subseteq i\mathfrak Z_{k+1}(3).\end{split}
\end{align}}

If the weight $ k+1$ does not exceed $3$ in the case of  \eqref{eq:Sun60}, then  we are effectively working with  $ \mathfrak g_{k+1,2}^S$, where the subset $S$ of $ \left\{ i,-i,\tau ,-\frac{1}{\tau},-\tau,\frac{1}{\tau}\right\}=\big\{ i,-i,i\rho,-\frac{1}{i\rho},\linebreak -i\rho,\frac{1}{i\rho}\big\}$ contains no more than $3$ members. By virtue of  GPL rescaling and the  fact that $ \log\rho=\R[\Li_1(e^{4\pi i/5})-\Li_1(e^{2\pi i/5})]\in\mathfrak Z_1(5)$,  we only need to check that $G(\bm \alpha;1) \in\mathfrak Z_{w(\bm \alpha)}(10)$ when  the components of $ \bm \alpha$ involve no more than  $2$ different numbers from one of the following sets:\begin{align}
\big\{ 1,-1,\rho,\tfrac{1}{\rho},-\rho,-\tfrac{1}{\rho}\big\},\quad \big\{ \tfrac{1}{\rho},-\tfrac{1}{\rho},1,\tfrac{1}{\rho^{2}},-1,-\tfrac{1}{\rho^{2}}\big\},\quad \big\{ \rho,-\rho,\rho^{2},1,-\rho^{2},-1\big\}.
\end{align}This can be done automatically by Au's \texttt{IterIntDoableQ} function.
 \end{proof}
\begin{remark}As revealed by  Table \ref{tab:binom_sums}, two $ \mathfrak Z_{k+1}(10)$ cases of  \eqref{eq:Sun60}  cover
 an infinite  series evaluated in \cite[\S4]{Charlton2022SunConj}:\begin{align}
\sum_{n=0}^\infty\frac{\binom{2n}{n}\left[ 5\mathsf H^{(1)}_{2n+1} +\frac{12}{2n+1}\right]}{(2n+1)^2(-16)^n}=14\zeta_3.
\end{align}Charlton--Gangl--Lai--Xu--Zhao \cite[\S\S2--3]{Charlton2022SunConj} verified another infinite series due to Zhi-Wei Sun, namely\begin{align}
\sum_{n=0}^\infty\frac{\binom{2n}{n}\left[ 9\mathsf H^{(1)}_{2n+1} +\frac{32}{2n+1}\right]}{(2n+1)^316^n}=40\I\Li_4(i)+\frac{5\pi\zeta_3}{12}\label{eq:Sun_beta4}
\end{align} by maneuvering identities of polylogarithms. Au \cite[Corollary 4.5]{Au2022b} achieved the same by an application of  the Wilf--Zeilberger method. At the time of this writing, we do not have an automated proof of \eqref{eq:Sun_beta4}, since Au's \texttt{MultipleZetaValues} package currently does not  simplify CMZVs in $ \mathfrak Z_{4}(12)$.   \eor\end{remark}

\begin{table}[t!p]\caption{Selected closed forms for  quadratic analogs of binomial sums, where  $ G\colonequals\I\Li_2(i)$ and   $ \lambda\colonequals\log2$\label{tab:binom_sums_quad}}

\begin{tiny}\begin{align*}
\begin{array}{c|l@{}l@{}}
\hline\hline s&\begin{array}{@{}l}\displaystyle \quad \sum_{n=1}^\infty\frac{\binom{2n-2}{n-1}^2\vphantom{\frac{\frac11}{}}}{2^{4n}(2n-1)^{s+1}}\\\displaystyle=\frac{1}{2^{2}\pi s!}\int_1^i\left[ \int_i^\tau\log^s\frac{(1+\tau^{2})t}{(1+t^{2})\tau}\frac{\D t}{t}\ \right]\frac{\D \tau}{1+\tau^2}\vphantom{\frac{}{\frac11}}\end{array}& \begin{array}{@{}l}\displaystyle \quad \sum_{n=1}^\infty\frac{\binom{2n}{n}^2\vphantom{\frac{\frac11}{}}}{2^{4n}n^{s+1}}\\\displaystyle=\frac{2^{s+4}}{\pi s!}\int_{0}^1\left[\int_0^\upsilon\log^s\frac{(1+t^{2})\upsilon}{(1+\upsilon^2)t}\frac{t\D t}{1+t^{2}}^{}\right]\frac{\D\upsilon}{1+\upsilon^2}\vphantom{\frac{}{\frac11}}\end{array}\\\hline0&\frac{G}{2^2\pi}\vphantom{\dfrac{\int^2}1}&-\frac{2^3G}{\pi}+2^{2} \lambda\\[5pt] 1 &\frac{\I L_{2,1}(i)}{\pi}+\frac{G\lambda}{2\pi}&\frac{2^6\I L_{2,1}(i)}{\pi}+\frac{2^5G\lambda}{\pi}-2^3\lambda^2\\[5pt]2&-\frac{3\I \Li_4(i)}{\pi}+\frac{2^2\I L_{2,1,1}(i)}{\pi}+\frac{5\cdot7\zeta_3}{2^6}&\frac{2^7\cdot3\I \Li_4(i)}{\pi}-\frac{2^9\I L_{2,1,1}(i)}{\pi}-2\cdot3\cdot11\zeta_3\\&{}+\frac{2\lambda\I L_{2,1}(i)}{\pi}+\frac{G}{2\pi}\left( \frac{\pi^2}{2^2} +\lambda^2\right)&{}-\frac{2^8\lambda\I L_{2,1}(i)}{\pi}-\frac{2^{6}G}{\pi}\left( \frac{\pi^2}{2^2} +\lambda^2\right)+\frac{2^5\lambda^3}{3}\\[5pt]3&\frac{2^4\cdot3^2\I L_{4,1}(i)}{7\pi}+\frac{2^2\cdot19\I L_{3,2}(i)}{7\pi}&\frac{2^{12}\cdot3^2\I L_{4,1}(i)}{7\pi}+\frac{2^{10}\cdot19\I L_{3,2}(i)}{7\pi}\\&{}+\frac{2^{4}\I L_{2,1,1,1}(i)}{\pi}-\frac{2\cdot3\lambda\I \Li_4(i)}{\pi}-\frac{5\zeta_{-3,1}}{2^3}&{}+\frac{2^{12}\I L_{2,1,1,1}(i)}{\pi}-\frac{2^{9}\cdot3\lambda\I \Li_4(i)}{\pi}\\&{}+\frac{2^3\lambda\I L_{2,1,1}(i)}{\pi}+\frac{5\cdot7\zeta_{3}\lambda}{2^5}&{}-2^{5}\cdot5\zeta_{-3,1}+\frac{2^{11}\lambda\I L_{2,1,1}(i)}{\pi}\\&{}+\left( \frac{\pi^{2}}{2^{2}}+\lambda^2 \right)\frac{2\I L_{2,1}(i)}{\pi}+\frac{G\lambda}{\pi}\left( \frac{\pi^{2}}{2^{2}} +\frac{\lambda^{2}}{3}\right)&{}+2^{3}\cdot3\cdot11\zeta_{3}\lambda+\left( \frac{\pi^{2}}{2^{2}}+\lambda^2 \right)\frac{2^9\I L_{2,1}(i)}{\pi}\\&{}+\frac{73\pi^{4}}{2^{8}\cdot3^2\cdot7}&{}+\frac{2^{8}G\lambda}{\pi}\left( \frac{\pi^{2}}{2^{2}} +\frac{\lambda^{2}}{3}\right)-\frac{2^{5}\lambda^4}{3}+\frac{73\pi^4}{3^2\cdot7}\\[5pt]\hline\hline\multicolumn{3}{c}{}\\
\hline\hline s&\begin{array}{@{}l}\displaystyle \quad \sum_{n=1}^\infty\frac{\binom{2n-2}{n-1}^2\mathsf H^{(1)}_{n-\smash[t]{\frac{1}{2}}} \vphantom{\frac{\frac11}{}}}{2^{4n}(2n-1)^{s}}\\\displaystyle=\frac{1}{2^{s+1}\pi}\int_1^i\left[ \int_i^\tau\mathsf L_s^{(1)}\left(\frac{(1+t^{2})^{2}\tau^{2}}{(1+\tau^{2})^{2}t^{2}}\right)\frac{\D t}{t}\ \right]\frac{\D \tau}{1+\tau^2}\vphantom{\frac{}{\frac11}}\end{array}& \begin{array}{@{}l}\displaystyle \quad \sum_{n=1}^\infty\frac{\binom{2n}{n}^2\mathsf H^{(1)}_{n} \vphantom{\frac{\frac11}{}}}{2^{4n}n^{s}}\\\displaystyle=\frac{2^{4}}{\pi}\int_{0}^1\left[\int_0^\upsilon \mathsf L_s^{(1)}\left(\frac{(1+\upsilon^2)^{2}t^{2}}{(1+t^{2})^{2}\upsilon^{2}}\right)\frac{t\D t}{1+t^{2}}^{}\right]\frac{\D\upsilon}{1+\upsilon^2}\vphantom{\frac{}{\frac11}}\end{array}\\\hline1 &\frac{2^{2}\I L_{2,1}(i)}{\pi}+\frac{G\lambda}{\pi}\vphantom{\frac{\frac\int1}1}&-\frac{2^6\I L_{2,1}(i)}{\pi}-\frac{\pi^2}{2\cdot3}\\[5pt]2&-\frac{37\I\Li_4(i)}{2\pi}+\frac{2^3\cdot3\I L_{2,1,1}(i)}{\pi}+\frac{3\cdot5\cdot7\zeta_{3}}{2^{5}}&-\frac{2^{5}\cdot5^2\I\Li_4(i)}{\pi}+\frac{2^{10}\I L_{2,1,1}(i)}{\pi}+3\cdot47\zeta_{3}\\&{}+\frac{2^3\lambda\I L_{2,1}(i)}{\pi}+\frac{3G}{\pi}\left( \frac{\pi^{2}}{2^{2}} +\frac{\lambda^{2}}{3}\right)&{}+\frac{2^8\lambda\I L_{2,1}(i)}{\pi}+2^5\pi G+\frac{2\pi^2\lambda}{3}\\[5pt]3&\frac{2^3\cdot149\I L_{4,1}(i)}{7\pi}+\frac{2\cdot313\I L_{3,2}(i)}{7\pi}&-\frac{2^{10}\cdot113\I L_{4,1}(i)}{7\pi}-\frac{2^{8}\cdot3\cdot79\I L_{3,2}(i)}{7\pi}\\[2pt]&{}+\frac{2^7\I L_{2,1,1,1}(i)}{\pi}-\frac{37\lambda\I\Li_4(i)}{\pi}-5\zeta_{-3,1}&{}-\frac{2^{12}\cdot3\I L_{2,1,1,1}(i)}{\pi}+\frac{2^{7}\cdot5^{2}\lambda\I\Li_4(i)}{\pi}\\[2pt]&{}+\frac{2^4\cdot3\lambda\I L_{2,1,1}(i)}{\pi}+\frac{3\cdot5\cdot7\zeta_{3}\lambda}{2^{4}}&{}+2^{5}\cdot3\cdot5\zeta_{-3,1}-\frac{2^{12}\lambda\I L_{2,1,1}(i)}{\pi}\\&{}+(\pi^2+2\lambda^2)\frac{2^2\I L_{2,1}(i)}{\pi}+\frac{G\lambda}{\pi}\left( \frac{3\pi^{2}}{2} +\frac{2\lambda^{2}}{3}\right)&{}-2^2\cdot3\cdot47\zeta_3\lambda-(3\pi^2+2^2\lambda^2)\frac{2^7\I L_{2,1}(i)}{\pi}\\&{}+\frac{5\cdot233\pi^{4}}{2^{9}\cdot3^2\cdot7}&{}-2^{7}\pi G\lambda-\frac{2^2\pi^2\lambda^{2}}{3}-\frac{29\cdot151\pi^4}{2^2\cdot3^2\cdot5\cdot7}\\[5pt]\hline\hline\multicolumn{3}{c}{}\\\hline\hline s&\begin{array}{@{}l}\displaystyle \quad \sum_{n=1}^\infty\frac{\binom{2n-2}{n-1}^2\mathsf H^{(1)}_{2n-1} \vphantom{\frac{\frac11}{}}}{2^{4n}(2n-1)^{s}}\\\displaystyle=\frac{1}{2^{2}\pi}\int_1^i\left[ \int_i^\tau\mathsf L_s^{(1)}\left(\frac{(1+t^{2})\tau}{(1+\tau^{2})t}\right)\frac{\D t}{t}\ \right]\frac{\D \tau}{1+\tau^2}\vphantom{\frac{}{\frac11}}\end{array}& \begin{array}{@{}l}\displaystyle \quad \sum_{n=1}^\infty\frac{\binom{2n}{n}^2\mathsf H^{(1)}_{2n} \vphantom{\frac{\frac11}{}}}{2^{4n}n^{s}}\\\displaystyle=\frac{2^{s+3}}{\pi}\int_{0}^1\left[\int_0^\upsilon \mathsf L_s^{(1)}\left(\frac{(1+\upsilon^2)t}{(1+t^{2})\upsilon}\right)\frac{t\D t}{1+t^{2}}^{}\right]\frac{\D\upsilon}{1+\upsilon^2}\vphantom{\frac{}{\frac11}}\end{array}\\\hline1 &\frac{5\I L_{2,1}(i)}{2\pi}+\frac{G\lambda}{2\pi}+\frac{\pi^2}{2^7}\vphantom{\frac{\frac\int1}1}&-\frac{2^4\cdot3\I L_{2,1}(i)}{\pi}-\frac{\pi^2}{2^2\cdot3}\\[5pt]2&-\frac{5^2\I \Li_4(i)}{2\pi}+\frac{2^4\I L_{2,1,1}(i)}{\pi}+\frac{5\cdot7\zeta_3}{2^4}&-\frac{2^5\cdot19\I\Li_4(i)}{\pi}+\frac{2^8\cdot3\I L_{2,1,1}(i)}{\pi}+109\zeta_{3}\\&{}+\frac{5\lambda\I L_{2,1}(i)}{\pi}+\frac{G}{2\pi}(\pi^2+\lambda^2)+\frac{\pi^2\lambda}{2^6}&{}+\frac{2^{6}\cdot3\lambda\I L_{2,1}(i)}{\pi}+2^3\cdot3\pi G+\frac{\pi^2\lambda}{3}\\[5pt]3&\frac{2^6\cdot13\I L_{4,1}(i)}{7\pi}+\frac{2^2\cdot109\I L_{3,2}(i)}{7\pi}&-\frac{2^{11}\cdot43\I L_{4,1}(i)}{7\pi}-\frac{2^{10}\cdot3^2\cdot5\I L_{3,2}(i)}{7\pi}\\[2pt]&{}+\frac{2^3\cdot11\I L_{2,1,1,1}(i)}{\pi}-\frac{5^{2}\lambda\I\Li_4(i)}{\pi}&{}-\frac{2^{10}\cdot3^2\I L_{2,1,1,1}(i)}{\pi}+\frac{2^{7}\cdot19\lambda}{\pi}\I\Li_4(i)\\&{}-\frac{5\cdot11\zeta_{-3,1}}{2^{4}}+\frac{2^5\lambda}{\pi}\I L_{2,1,1}(i)+\frac{5\cdot7\zeta_{3}\lambda}{2^3}&{}+2^{3}\cdot3^{2}\cdot5\zeta_{-3,1}-\frac{2^{10}\cdot3\lambda}{\pi}\I L_{2,1,1}(i)\\&{}+\left( \frac{11\pi^{2}}{2^{2}} +5\lambda^{2}\right)\frac{\I L_{2,1}(i)}{\pi}+\frac{G\lambda}{\pi}\left( \pi^{2}+\frac{\lambda^{2}}{3} \right)&{}-2^{2}\cdot109\zeta_{3}\lambda-2^{5}\cdot3\left( 3\pi^{2}+2^{2}\lambda^{2}\right)\frac{\I L_{2,1}(i)}{\pi}\\&{}+\frac{\pi^{2}\lambda^2}{2^6}+\frac{821\pi^4}{2^9\cdot3^2\cdot7}&{}-2^{5}\cdot 3\pi G\lambda-\frac{2\pi^{2}\lambda^2}{3}-\frac{29\cdot113\pi^4}{2^2\cdot3^2\cdot5\cdot7}\\[5pt]\hline\hline
\end{array}\end{align*}

\end{tiny}\end{table}
\begin{corollary}[Generalized Wang--Chu series as Goncharov--Deligne periods]For $ s\in\mathbb Z_{\geq0}$,   $ r_1,\dots,r_M\in\mathbb Z_{>0}$, and  $ {r}'_1,\dots,{r}'_{{M'}}\in\mathbb Z_{>0}$, we have (cf.\ \cite[(17), (18), (27)--(29)]{Campbell2022WangChu})\begin{align}
\sum_{n=1}^\infty \left[\prod_{j=1}^M\mathsf H_{n-\smash[t]{1}}^{(r_j)}\right]\left\{\prod_{\smash[b]{j'}=1}^{M'}\Big[{{\mathsf H}}_{\smash[t]{2}n-1}^{( {r}'_{j'})}+\smash[t]{\overline{\mathsf H}}_{\smash[t]{2}n-1}^{( {r}'_{j'})}\Big]\right\}\frac{\binom{2n-2}{n-1}^{2}}{2^{4n}(2n-1)^{s+1}}\in{}&\frac{ \mathfrak Z^{}_{k+2}(4)}{\pi i},\label{eq:WangChu_gen}\\\sum_{n=1}^\infty \left[\prod_{j=1}^M{{\mathsf H}}_{\smash[t]{}n-1}^{( {r}_{j})}\right]\left\{\prod_{\smash[b]{j'}=1}^{M'}\Big[{{\mathsf H}}_{\smash[t]{2}n-1}^{( {r}'_{j'})}+\smash[t]{\overline{\mathsf H}}_{\smash[t]{2}n-1}^{( {r}'_{j'})}\Big]\right\}\frac{\binom{2n}{n}^2}{2^{4n}n^{s+1}}\in{}&\frac{\mathfrak Z_{k+2}^{}(4)}{\pi i},\label{eq:binom_sum_quad}
\end{align}where  $ k=s+\sum_{j=1}^M r_j+\sum_{{j}'=1}^{{M'}} {r}'_{j'}$. \end{corollary}
\begin{proof}Observe that \begin{align}
\frac{\pi}{2}\binom{2n-2}{n-1}=\int_0^{\pi/2}(2\cos\phi)^{2n-2}\D\phi={}\int_{1}^{i}\left( \frac{1+\tau^2}{\tau} \right)^{2n-1}\frac{\D\tau}{i(1+\tau^2)},
\end{align}where the integral contour in the complex $\tau$-plane is a straight line segment, on which $ |1+\tau^2|\leq2|\tau|$. Hence, the last corollary allows us to equate the left-hand side of \eqref{eq:WangChu_gen} with a member of the following $ \mathbb Q$-vector space:\begin{align}
\frac1\pi\Span_{\mathbb Q}\left\{\left.\int_{1}^{i}\frac{(\pi i)^{\ell}G\left( \alpha_{1},\dots,\alpha_{k+1-\ell};\tau\right)\D\tau}{1+\tau^2}\right|\begin{smallmatrix} \alpha_{1},\dots,\alpha_{k+1-\ell}\in\{0,1,i,-1,-i\}\\0\leq\ell\leq k+1\end{smallmatrix}\right\}.
\end{align}Appealing to the  GPL recursion in \eqref{eq:GPL_rec} and applying Panzer's logarithmic regularization \cite[\S2.3]{Panzer2015} at $\tau=i$, we may recognize the expression above as a  subspace of $ \frac{1}{\pi i}\mathfrak Z_{k+2}(4)$.

To prove \eqref{eq:binom_sum_quad}, first verify\begin{align}
\int_0^1\left( \frac{\upsilon}{1+\upsilon^2} \right)^{2n}\frac{\D\upsilon}{1+\upsilon^2}=\frac{1}{2}\int_0^{1/2}[x(1-x)]^{n-\frac{1}{2}}\D x=\frac{\pi \binom{2n}{n}}{2^{2(2n+1)}}
\end{align} by a substitution $ \upsilon=\sqrt{x}/\sqrt{1-x}$, then integrate the relation \eqref{eq:KWY2_alt} in Corollary \ref{cor:binom_sum}(b).    \end{proof}\begin{remark} Some special cases of  \eqref{eq:WangChu_gen} and \eqref{eq:binom_sum_quad} have been investigated by Cantarini--D'Aurizio \cite[\S\S2--3]{CantariniAurizio2019} in the Fourier--Legendre framework. See Table  \ref{tab:binom_sums_quad} for more examples.   \eor\end{remark}

\subsection{An outlook on $ \varepsilon$-expansions of hypergeometric expressions\label{subsec:eps_expn}}
As pointed out in  \cite[\S3]{Kalmykov2007}, a common source for   non-linear Euler sums and related series in high energy physics is  an      $ \varepsilon$-expansion in the following form:  \begin{align}\begin{split}&
\left.\frac{\partial^{\ell}}{\partial\varepsilon^\ell}\right|_{\varepsilon=0} {_p}F_{p-1}\left(\!\!\left.\begin{smallmatrix}a_{1}+r_1\varepsilon,\dots,a_{p}+r_p\varepsilon\\[2pt]b_{1}+s_{1}\varepsilon,\dots,b_{p-1}+s_{p-1}\varepsilon\end{smallmatrix}\right| z\right)\\\colonequals{}&\left.\frac{\partial^{\ell}}{\partial\varepsilon^\ell}\right|_{\varepsilon=0}\left[1+\sum_{n=1}^{\infty}\frac{\prod_{j=1}^p(a_{j}+r_{j}\varepsilon)_n}{\prod_{k=1}^{p-1}(b_{k}+s_{k}\varepsilon)_n}\frac{z^{n}}{n!}\right]\in\mathbb C ,\quad   \ell\in\mathbb Z_{\geq0},\end{split}\label{eq:pFq_deriv}
\end{align}where  the (generalized) hypergeometric  series  $  {_p}F_{p-1}$ is defined with   \begin{align}
\begin{cases}2a_j\in\mathbb Z,r_{j}\in\mathbb Q,\ & 1\leq j\leq p, \\
2b_k\in\mathbb Z,s_{k}\in\mathbb Q, & 1\leq k\leq p-1, \\
\end{cases}\label{eq:pFq_para}
\end{align}   and  $ (a)_k\colonequals \Gamma(a+k)/\Gamma(a)$.

Up to this point, the studies of non-linear Euler sums [together with their (inverse) binomial variations] in this section allow us to evaluate some (but not all)     $ \varepsilon$-expansions  via multiple polylogarithms. We  emphasize ``some (but not all)'', because we have only handled certain series whose the summands involve powers of binomial coefficients    $ \binom{2n}n^\mathscr N$ for $ \mathscr N\in\{-2,-1,0,1,2\}$. These closed-form series  are further subdivided into two classes:\begin{itemize}
\item
If  $ \mathscr N\in\{-1,0,1\}$ and  $ |z|\leq 1$, then we have polylogarithmic evaluations of \eqref{eq:pFq_deriv}, provided that the choice of parameters in \eqref{eq:pFq_para} turns the right-hand side of \eqref{eq:pFq_deriv} into a $ \mathbb Q$-linear combination of  \eqref{eq:nonlin_Euler_per}, \eqref{eq:nonlin_Euler_per_alt}, \eqref{eq:KWY_inv_binom}, \eqref{eq:KWY_inv_binom_alt}, \eqref{eq:KWY2}, \eqref{eq:KWY2_alt}, \eqref{eq:BE_HPL}, \eqref{eq:BE_HPL_even},  and \eqref{eq:Sun_gen}, while ensuring its convergence at   $ z=1$;\item If $ \mathscr N\in\{-2,2\}$  and  $z=1$,  then we have  polylogarithmic reductions of \eqref{eq:pFq_deriv} when   the parameters in \eqref{eq:pFq_para} are orchestrated in such a way that the right-hand side of \eqref{eq:pFq_deriv} becomes a $ \mathbb Q$-linear combination  of \eqref{eq:KWY_quad}, \eqref{eq:KWY_quad'}, \eqref{eq:WangChu_gen}, and  \eqref{eq:binom_sum_quad}.
\end{itemize}

Instead of a microscopic dissection of  \eqref{eq:pFq_deriv} into a  $ \mathbb Q$-linear combination of [(inverse) binomial variations on] non-linear Euler sums, one can also compute  $ \varepsilon$-expansions while treating the (generalized) hypergeometric series $_pF_{p-1}$ as a whole.

To illustrate the holistic perspective to  $_pF_{p-1}$, we consider a double integral representation\footnote{The author thanks an anonymous referee for suggesting this approach.}  of a zeta Mahler measure for $ s\in(-2,0)$:\begin{align}&
W_{3}(s)=-\frac{3^{s+\frac32}\sin\frac{\pi s}{2}}{\pi^{2}}\int_0^1\D t\int_0^1\D u\frac{2\left[\frac{1-u^{2}}{4-{(1-u^{2})t}{}}\frac{t}{1-t}\right]^{s/2}}{(1-t)\left[ 4-{(1-u^{2})t}{} \right]},\label{eq:W3_Euler_int}
\end{align}which descends from a formula of Borwein--Borwein--Straub--Wan \cite[(24), with $ z=\sqrt{1-u^2}$]{BBSW2012Mahler} and Euler's integral representation for $ _2F_1$ \cite[Theorem 2.2.1]{AAR}.
Exploiting a beta integral $\int_{0}^1\big( \frac{t}{1-t} \big)^{s/2}\frac{\D t}{1-t}=-\frac{\pi}{ \sin\frac{\pi s}{2}}$, we can extend \eqref{eq:W3_Euler_int} to \begin{align}
\begin{split}W_3(s)={}&\frac{3^{s+\frac32}}{\pi}\left(\int_0^1\frac{2\left(\frac{1-u^{2}}{3+u^{2}{}}\right)^{s/2}}{3+u^{2}}\D u-\frac{\sin\frac{\pi s}{2}}{\pi}\int_0^1\D t\int_0^1\D u\frac{2\left[\frac{1-u^{2}}{4-{(1-u^{2})t}{}}\frac{t}{1-t}\right]^{s/2}}{\left(t-\frac{4}{1-u^{2}}\right)(3+u^{2})}\right.\\{}&\left.-\frac{\sin\frac{\pi s}{2}}{\pi}\int_0^1\D t\int_0^1\D u\frac{2\left(\frac{t}{1-t} \right)^{s/2}\left\{\left[\frac{1-u^{2}}{4-{(1-u^{2})t}{}}\right]^{s/2}-\left(\frac{1-u^{2}}{3+u^{2}{}}\right)^{s/2}\right\}}{(1-t)(3+u^{2})}\right)\\{}&\end{split}\tag{\ref{eq:W3_Euler_int}$'$}\label{eq:W3_Euler_int'}
\end{align}for $ s\in(-2,2)$.  Differentiating \eqref{eq:W3_Euler_int'} under the integral sign, and fibrating with respect to $u$ after integrating over $t$, we find \begin{align}\begin{split}&
\left.\frac{\D^k W_{3}(s)}{\D s^k}\right|_{s=0}\\\in{}&\frac{\sqrt{3}}{\pi}\Span_{\mathbb Q}\left\{\pi^{2\ell}(\log3)^{m}\int_0^1\frac{G(\bm \alpha;u)G(\bm \beta;1)\D u}{3+u^{2}}\left|\begin{smallmatrix}\alpha_{1},\dots,\alpha_{w(\bm\alpha)}\in\{0,1,-1, i\sqrt{3},-i\sqrt{3}\}\\\beta_{1},\dots,\beta_{w(\bm\beta)}\in\{0,1,4\}\\\ell,m,w(\bm\alpha),w(\bm\beta)\in\mathbb Z_{\geq0}\\2\ell+m+w(\bm\alpha)+w(\bm\beta)=k\end{smallmatrix}\right. \right\}\\\subseteq{}&\frac{\mathfrak Z_{k+1}(6)}{\pi i}.\end{split}\label{eq:W3_Z_k(6)}
\end{align} Here, to confirm the ``$\in$'' step, build \begin{align}\begin{split}&
\Span_{\mathbb Q}\left\{G(\widetilde{\alpha}_{1},\dots,\widetilde{\alpha}_n;1)\left|\widetilde{\alpha}_{1},\dots,\widetilde{\alpha}_{n}\in\left\{ 0,1,\frac{4}{1-u^{2}} \right\}\right.\right\}\\\subseteq{}&\Span_{\mathbb Q}\left\{G(\bm \alpha;u)G(\bm \beta;1)\left|\begin{smallmatrix}\alpha_{1},\dots,\alpha_{w(\bm\alpha)}\in\{0,1,-1, i\sqrt{3},-i\sqrt{3}\}\\\beta_{1},\dots,\beta_{w(\bm\beta)}\in\{0,1,4\}\\w(\bm\alpha),w(\bm\beta)\in\mathbb Z_{\geq0};w(\bm\alpha)+w(\bm\beta)=n\end{smallmatrix}\right.\right\}\end{split}
\end{align}inductively on \eqref{eq:GPL_diff_form}, with logarithmic regularization \cite[\S2.3]{Panzer2015} when $ \widetilde{\alpha}_{1}=1$; to confirm the ``$ \subseteq$'' step, simply note
 that  $ \pi^2\in\mathfrak Z_2(1)$, $ \log 2\in\mathfrak Z_1(2)$, and $ \log3\in\mathfrak Z_1(3)$, before evaluating  \texttt{IterIntDoableQ[\textbraceleft0, 1, -1, I Sqrt[3], -I Sqrt[3]\textbraceright]} and   \texttt{IterIntDoableQ[\textbraceleft0, 1, 4\textbraceright]}  in  Au's \texttt{MultipleZetaValues} package \cite{Au2022a}. This recovers Theorem \ref{thm:Zk(6)}.

For  $ |z|\leq1$, an $ \varepsilon$-expansion in \eqref{eq:pFq_deriv} may also give rise to  summands  involving     $ \binom{2n}n^\mathscr N$ for $ \mathscr N\in\mathbb Z\smallsetminus\{-1,0,1\}$.  Such binomial variations on non-linear Euler sums may or may not be representable by  CMZVs,
according to Zhi-Wei\ Sun's recent experimental discoveries  (see \cite[\S\S2--4]{Sun2022} and \cite[\S6]{Sun2023}). Some of Sun's new identities  have been  verified analytically by  Campbell \cite{Campbell2023Sun}, Hou--Sun \cite{HouSun2023}, Wei \cite{Wei2023Sun,Wei2023c}, Li--Chu \cite{LiChu2023}, and Wei--Xu \cite{WeiXu2023a,WeiXu2023b}. For future research, it is possibly worth studying  (multiple) integral representations for these  newly discovered  $ \varepsilon$-expansions and their generalizations, in search of algorithmic reductions to CMZVs and non-CMZVs.

\subsection*{Acknowledgments}The current research grew out of my unpublished  notes on polylogarithms and multiple zeta values prepared in 2012 for Prof.\ Weinan E's working seminar on mathematical problems in quantum field theory at Princeton University. Some parts of this paper were written during my visits to Prof.\ E at Peking University in summers  of 2017 and 2022. I dedicate this work to his 60th birthday and his vision for automating mathematical research.

I am deeply indebted to three anonymous referees, whose perceptive comments helped me improve the presentation of this paper in many ways.

\end{document}